\documentclass{amsart}

\usepackage{amssymb,amsfonts,amsmath}
\usepackage{hyperref}
\usepackage{tikz}
\usepackage{mathtools}
\usepackage{stackrel}
\usepackage{empheq}
\usepackage{yhmath}

\usepackage{soul}

\usepackage{comment}

\usepackage[all]{xy}
\usetikzlibrary{matrix,arrows,decorations.pathmorphing,automata, matrix, positioning, calc, shapes.multipart}
\usepackage{tikz-cd}

\usetikzlibrary{decorations.pathreplacing}

\newcommand{\tikzmark}[1]{\tikz[overlay,remember picture] \node (#1) {};}

\newcommand*{\BraceAmplitude}{0.5em}
\newcommand*{\VerticalOffset}{0.5ex}
\newcommand*{\InsertUnderBrace}[4][]{%
    \begin{tikzpicture}[overlay,remember picture]
\draw [decoration={brace,amplitude=\BraceAmplitude},decorate, thick,draw=blue,text=black,#1]
        ($(#3)+(0,-\VerticalOffset)$) -- 
        ($(#2)+(0,-\VerticalOffset)$)
        node [below=\VerticalOffset, midway] {#4};
    \end{tikzpicture}%
}%

\DeclareMathOperator{\Id}{Id}

\DeclareMathOperator{\SYT}{SYT}
\DeclareMathOperator{\Hom}{Hom}
\DeclareMathOperator{\shift}{\mathtt{shift}}
\DeclareMathOperator{\End}{End}

\newcommand{\bt}{\mathbf{t}}

\newcommand{\Z}{\mathbb{Z}}
\newcommand{\M}{\mathbb{M}}
\newcommand{\C}{\mathbb{C}}
\newcommand{\Q}{\mathbb{Q}}
\newcommand{\PP}{\mathbb{P}}
\newcommand{\RR}{\mathbb{R}}
\newcommand{\KK}{\mathbb{K}}
\newcommand{\OO}{\mathbb{O}}

\newcommand{\sln}{\mathfrak{sl}_n}
\newcommand{\gln}{\mathfrak{gl}_n}

\newcommand{\cB}{\mathcal{B}}
\newcommand{\cI}{\mathcal{I}}
\newcommand{\cJ}{\mathcal{J}}
\newcommand{\cL}{\mathcal{L}}
\newcommand{\cO}{\mathcal{O}}
\newcommand{\cP}{\mathcal{P}}
\newcommand{\cS}{\mathcal{S}}
\newcommand{\cT}{\mathcal{T}}
\newcommand{\cX}{\mathcal{X}}
\newcommand{\cY}{\mathcal{Y}}
\newcommand{\comp}[1]{\mathcal{C}_{r}(#1)}
\newcommand{\FXi}{F_\mathcal{X}^{-1}}
\newcommand{\fxi}[1]{\wideparen{{#1}}} 

\newcommand{\FX}{F_\mathcal{X}}
\newcommand{\FY}{F_\mathcal{Y}}

\newcommand{\Pn}{\mathcal{P}(n)}
\newcommand{\Min}{\mathbf{\mathtt{L}}^+_{\mathrm{min}}(n)} 

\newcommand{\Hilb}{\mathrm{Hilb}}
\newcommand{\Hilbk}{\Hilb_k} 
\newcommand{\PHilb}{\mathrm{PHilb}}
\newcommand{\PHilbx}{\PHilb^{x}(C)}
\newcommand{\FHilb}{\mathrm{CPHilb}}
\newcommand{\Spr}{\mathrm{Spr}}
\newcommand{\Fl}{\mathcal{F}l}
\newcommand{\Sym}{\mathrm{Sym}}

\DeclareMathOperator{\rot}{rot}
\DeclareMathOperator{\fl}{fl}
\DeclareMathOperator{\id}{id}
\DeclareMathOperator{\gr}{gr}
\DeclareMathOperator{\triv}{triv}
\DeclareMathOperator{\gen}{gen}
\DeclareMathOperator{\sort}{sort}
\DeclareMathOperator{\Ind}{Ind}
\DeclareMathOperator{\cont}{ct}
\newcommand{\ct}[2]{\cont_{#1}({#2})}

\DeclareMathOperator{\Res}{Res}
\DeclareMathOperator{\rev}{rev}
\DeclareMathOperator{\ev}{ev_0}

\newtheorem{theorem}{Theorem}[section]
\newtheorem{proposition}[theorem]{Proposition}
\newtheorem{corollary}[theorem]{Corollary}
\newtheorem{lemma}[theorem]{Lemma}
\theoremstyle{definition}
\newtheorem{remark}[theorem]{Remark}

\newtheorem{definition}[theorem]{Definition}
\newtheorem{example}[theorem]{Example}

\newcommand{\Gr}{\mathcal{G}r} 
\newcommand{\Lnr}{\mathcal{L}} 
\newcommand{\cA}{\mathcal{A}}
\newcommand{\MV}[1]{{\color{blue} {MV  #1}}}
\newcommand{\MVcomment}[1]{} 
\newcommand{\EG}[1]{{\color{green}  EGors \textsf  #1}}

\newcommand{\affSn}{\widetilde \Sn} 
\newcommand{\affSnCox}{\widehat \Sn} 
\newcommand{\affSnplus}{\affSn^+}
\newcommand{\affSnminus}{\affSn^-}
\newcommand{\ba}{\mathbf{a}}
\newcommand{\bb}{\mathbf{b}}
\newcommand{\bc}{\mathbf{c}}
\newcommand{\bd}{\mathbf{d}}
\newcommand{\boldb}{\mathbf{b}}
\newcommand{\Sn}{\Sk{n}}
\newcommand{\Sk}[1]{\cS_{#1}}
\newcommand{\Sa}{\cS({\ba})} 
\newcommand{\Sd}{\cS({\bd})}
\newcommand{\Srevd}{\cS(\revd)}
\newcommand{\revd}{\bd^{\rev}}
\newcommand{\revdtwist}{\rev_{\bd}}  
\newcommand{\Pinc}{\cP^+(n)}
\newcommand{\Pdec}{\cP^-(n)}
\newcommand{\tav}[1]{\mathtt{t}_{#1}} 
\newcommand{\ta}[1]{\tav{{\mathbf{ #1}}}} 
\newcommand{\wa}[1]{\ww_{\mathbf{#1}}} 
\newcommand{\wb}[1]{\ww_{#1}} 
\newcommand{\ww}{\omega} 
\newcommand{\wt}{\mathtt{w}} 
\newcommand{\tw}{\mathtt{w}} 
\newcommand{\wta}{\wt(\ba)}
\newcommand{\p}{p} 
\newcommand{\perm}{p_m} 
\newcommand{\wdlong}{\ww_0^{\bd}} 

\newcommand{\lex}{>_{\mathtt lex}}
\newcommand{\glex}{\ge_{\mathtt lex}}
\newcommand{\bruhat}{<_{\mathtt B}}
\newcommand{\lebruhat}{\le_{\mathtt B}}
\newcommand{\llex}{<_{\mathtt lex}}

\newcommand{\RCA}{H_{t,c}}
\newcommand{\Htc}{H_{t,c}}
\newcommand{\Hnx}{H_n(\mathbf{x})}
\newcommand{\Hny}{H_n(\mathbf{y})}
\newcommand{\Hnu}{H_n(\mathbf{u})}
\newcommand{\Hd}{H(\bd,\mathbf{u})}
\newcommand{\Hrevd}{H(\revd,\mathbf{u})}
\newcommand{\A}{\mathcal{A}} 

\newcommand{\Lctriv}{L_c(\triv)}
\newcommand{\Mw}{M_{\wt}}

\newcommand{\Wm}{\mathcal{D}} 
\newcommand{\Vnl}{V_{(n-\ell, 1^{\ell})}}  
\newcommand{\Vnlo}{V_{(n-\ell+1, 1^{\ell-1})}} 

\newcommand{\Vnolo}{V_{(n-\ell, 1^{\ell-1})}}

\newcommand{\Hu}[1]{H(\mathbf{{#1}},\mathbf{u})}
\newcommand{\arxiv}[1]{}  

\begin{document}
\title[RCA and Parabolic Hilbert schemes]
{
 From representations of the rational Cherednik algebra to parabolic Hilbert schemes    via the Dunkl-Opdam subalgebra}

\author{E.~Gorsky}
\thanks{E.G. was supported by the NSF grants  DMS-1700814 , DMS-1760329}
\address {Department of Mathematics\\ University of California, Davis\\ One Shields Avenue\\ Davis CA 95616}
\email{egorskiy@math.ucdavis.edu}

\author{J.~Simental}
\address{Max Planck Institut für Mathematik  \\ Vivatsgasse 7 \\
53111 Bonn \\
Germany}
\email{simental@im.unam.mx}
\author{M.~Vazirani} 
\thanks{
M.V. was supported by the Simons Foundation Collaboration Grant for Mathematicians, award number 319233}
\address {Department of Mathematics\\ University of California, Davis\\ One Shields Avenue\\ Davis CA 95616}
\email{vazirani@math.ucdavis.edu}


\date{\today}

\maketitle

\begin{abstract}
In this note we explicitly construct an action of the rational Cherednik algebra $H_{1,m/n}(\Sk{n},\C^n)$ corresponding to the permutation representation of $\Sk{n}$ on the $\C^{*}$-equivariant homology of parabolic Hilbert schemes of points on the plane curve singularity $\{x^{m} = y^{n}\}$ for coprime $m$ and $n$. We use this to construct actions of quantized Gieseker algebras on parabolic Hilbert schemes on the same plane curve singularity, and actions of the Cherednik algebra at $t = 0$ on the equivariant homology of parabolic Hilbert schemes on the non-reduced curve $\{y^{n} = 0\}.$ Our main tool is the study of the combinatorial representation theory of the rational Cherednik algebra via the subalgebra generated by Dunkl-Opdam elements
\end{abstract}

\tableofcontents  
\section{Introduction}

\subsection{Hilbert schemes on singular curves} It is well-known and classical that, 
given a smooth algebraic curve, 
its Hilbert scheme of $k$ points 
is smooth and, in fact, isomorphic to the $k$-th symmetric product of the curve.
On the contrary, much less is known in the case of 
a singular curve. In particular, Maulik \cite{Maulik} proved a conjecture of Oblomkov and Shende \cite{OS} relating the Euler characteristics of Hilbert schemes of points on a plane curve singularity to the HOMFLY-PT polynomial of its link. A more general conjecture of Oblomkov, Rasmussen and Shende \cite{ORS,GORS} relates the homology of these Hilbert schemes to the HOMFLY-PT homology of the link.

One possible approach to understanding  
Hilbert schemes of curves
is by constructing an action of interesting algebras on their homology. Rennemo \cite{Rennemo} constructed an action of the two-dimensional Weyl algebra on the homology of the Hilbert scheme of 
an integral locally planar curve 
(see also \cite{MS,MY}), and Kivinen \cite{Kivinen} generalized this action to reduced locally planar curves with several components. 

In this paper, we relate the geometry of (parabolic) Hilbert schemes on singular curves to the representation theory of the type $A$ rational Cherednik algebra and other related algebras. 

More precisely, consider coprime positive integers $m$ and $n$, and let $C := \{x^{m} = y^{n}\}$ be a plane curve singularity in $\C^2$. We will always work locally near the origin, to simplify the notations we will always use $\cO_C$ for the completion of $\C[C]$ at $(0,0)$.

Note that for every ideal $I \subseteq \cO_C$ we have that $\dim(I/xI) = n$. We consider the \emph{parabolic Hilbert scheme} $\PHilb_{k, n+k}(C)$ that is the following moduli space of flags 
\begin{equation}\label{eq:par hilb}
\PHilb_{k,n+k}(C):=\left\{\cO_C\supset I_k\supset I_{k+1}\supset \cdots\supset I_{k+n}=xI_k\right\}
\end{equation}
\noindent where $I_s$ is an ideal in the ring of functions $\cO_C$ of codimension $s$. Moreover, we set $\PHilbx := \sqcup_{k} \PHilb_{k, n+k}(C)$. The natural $\C^*$ action on $C$ naturally lifts to $\PHilbx$. Since $m$ and $n$ are coprime, the fixed points are precisely the flags of monomial ideals. In particular, the classes of these fixed points form a basis for the localized equivariant cohomology. The first main result of this paper is the following.

\begin{theorem}\label{thm:hilb intro}
There is a geometric action of the rational Cherednik algebra \newline $H_{1,m/n}(\Sk{n}, \C^n)$ on the localized $\C^*$-equivariant homology of $\PHilb^{x}(C)$. Moreover, with this action $H_{*}^{\C^{*}}(\PHilb^{x}(C))$ gets identified with the simple highest weight module $L_{m/n}(\triv)$. 
\end{theorem}

Recall that the rational Cherednik algebra
$H_{t,c} := H_{t,c}(\Sk{n},\C^n)$ contains the trivial idempotent
$e := \frac{1}{n!}\sum_{p \in \Sk{n}} p$, and we can form the spherical
subalgebra $eH_{t,c}e$. As a consequence of Theorem \ref{thm:hilb
intro} we get that the spherical subalgebra acts on the equivariant
homology of the Hilbert scheme $\Hilb(C) := \sqcup_{k}\Hilbk(C)$.

\begin{corollary}\label{cor:hilb intro}
There is an action of the spherical rational Cherednik algebra $eH_{1,m/n}(\Sk{n}, \C^n)e$ on the localized $\C^*$-equivariant homology of $\Hilb(C)$. Moreover, with this action $H_{*}^{\C^*}(\Hilb(C))$ gets identified with $eL_{m/n}(\triv)$.  
\end{corollary}

\begin{remark}
By \cite{MY,MS} the homology of the Hilbert schemes of singular curves is closely related to the homology of the corresponding compactified Jacobian, equipped with a certain ``perverse'' filtration.  By \cite{GORS,OY,OY2,VV} the latter homology carries an action of the spherical trigonometric Cherednik algebra, \cite{CherednikBook}. Furthermore by  \cite{OY,OY2} the associated graded space admits a natural action of the spherical rational Cherednik algebra corresponding to the {\em reflection representation} of $\Sn$ (also known as spherical rational Cherednik algebra of $\sln$). The construction of this action uses global Springer theory developed by Yun \cite{Yun}.

The main advantage of our proof of Theorem \ref{thm:hilb intro} is that it does not use compactified Jacobians or perverse filtration at all. The generators of  $H_{1,m/n}(\Sk{n}, \C^n)$ are identified with certain explicit operators in the homology of $\PHilb^{x}(C)$.
\end{remark}


We explore some ramifications of this result. In the theory of rational Cherednik algebras there is a ``$t = 0$" and ``$t = 1$" dichotomy, see Section \ref{sec: RCA background}, and in the statement of Theorem \ref{thm:hilb intro} we assume that $t = 1$. While the representation theory of the Cherednik algebra is very sensitive to this dichotomy, we have a version of Theorem \ref{thm:hilb intro} in the $t = 0$ case.

To this end, consider the \emph{non-reduced curve} $C_{0} := \{y^{n} = 0\}$. The punctual Hilbert scheme on $C_{0}$ is the moduli space of finite-codimensional  ideals in the local ring $\cO_{C_{0}, 0} = \C[[x,y]]/(y^n)$, and we may define the parabolic (punctual) Hilbert scheme $\PHilb_{k, n+k}(C_{0})$ analogously to \eqref{eq:par hilb}. Again we set $\PHilb^{x}(C_{0}) := \sqcup_{k} \PHilb_{k, n+k}(C_{0})$. We show the following ``$t = 0$" (or ``$m = \infty$") analogue of Theorem \ref{thm:hilb intro}. 

\begin{theorem}\label{thm:hilb t=0 intro}
There is a geometric action of $H_{0,1}(\Sk{n}, \C^n)$ on the localized $\C^*$-equivariant cohomology of $\PHilb^{x}(C_0)$, where $C_0$ is the non-reduced curve $\{y^{n} = 0\}$. Moreover, with this action $H_{*}^{\C^{*}}(\PHilb^{x}(C_0))$ gets identified with the
polynomial representation $\Delta_{0,1}(\triv)$. 
\end{theorem}

Similarly to Corollary \ref{cor:hilb intro} we get an action of the
spherical subalgebra $eH_{0,1}e$ on the equivariant homology of
$\Hilb(C_0)$, and under this action $H_{*}^{\C^*}(\Hilb(C_0))$ gets
identified with the
polynomial representation of $eH_{0,1}e$.

\subsection{Quantized Gieseker varieties}

Another ramification of Theorem \ref{thm:hilb intro} connects parabolic Hilbert schemes to the representation theory of quantized Gieseker varieties. These are quantizations of the moduli space $\mathcal{M}(n, r)$ of rank $r$ torsion-free sheaves on $\mathbb{P}^{2}$ with fixed trivialization at infinity and second Chern class $c_{2} = n$. The quantization, denoted $\cA_{c}(n, r)$, depends on a parameter $c \in \C$, see Section \ref{sect:gieseker} for a precise definition. For example, when $r = 1$, $\mathcal{M}(n, 1)$ is simply the Hilbert scheme of $n$ points in $\C^2$ and  $\cA_{c}(n,1)$ is the spherical rational Cherednik algebra, see \cite{GG}. 

 There is currently no presentation of  the algebra $\cA_{c}(n, r)$ by generators and relations. 
Nevertheless, Losev \cite{L2} managed to classify all finite-dimensional representations for a slightly smaller algebra $\overline{\cA_{c}}(n,r)$ such that
$\cA_{c}(n, r)=\mathcal{D}(\C)\otimes \overline{\cA_{c}}(n,r)$. Namely, if $c=m/n$, $\gcd(m,n)=1$ and $c$ is not in the interval 
$(-r,0)$ then  $\overline{\cA_{c}}(n,r)$ has a unique irreducible finite-dimensional representation $\overline{\Lnr}_{\frac{m}{n}}(n,r)$,
 otherwise there are none.  Furthermore, the action of $GL(r)$ on $\mathcal{M}(n,r)$ induces a quantum comoment
 map $\mathfrak{gl}(r)\to \cA_c(n,r)$ and hence defines an action of $\mathfrak{gl}(r)$ on $\overline{\Lnr}_{\frac{m}{n}}(n,r)$. In \cite{EKLS} Etingof, Krylov, Losev and the second author computed the dimension and graded $\mathfrak{gl}(r)$ character of $\overline{\Lnr}_{\frac{m}{n}}(n,r)$. 

In this paper we give a geometric construction of this representation for $m,n>0$.

Fix an integer $r > 0$, and denote by $\comp{n} \subseteq \Z_{\geq 0}^{r}$ the set of $r$-tuples of non-negative integers that add up to $n$. For $\gamma = (\gamma_1, \dots, \gamma_r) \in \comp{n}$ we consider the following parabolic Hilbert scheme
\begin{multline}
\label{eq:par hilb comp}
\PHilb^{\gamma, x}(C) := \left\{\cO_C \supseteq J^{0} \supseteq J^{1} \supseteq \cdots \supseteq J^{r} = xJ^{0} \, : \right.
\\
\left. \dim(\cO_{C}/J^{0}) < \infty \;
\text{ and} \; \dim(J^{i-1}/J^{i}) = \gamma_{i}\right\}.
\end{multline}
\noindent  For example, $\PHilb^{x}(C) = \PHilb^{(1, \dots, 1),x}(C)$,
where $(1, 1, \dots, 1) \in \mathcal{C}_{n}(n)$ and $\Hilb(C) :=
\sqcup_{k \geq 0}\Hilbk(C)$ is $\PHilb^{(n),x}(C)$, 
where $(n) \in \mathcal{C}_{1}(n)$.
 We define the
\emph{compositional parabolic Hilbert scheme} of $C$ to be
\begin{equation}\label{eqn:comp hilb}
\FHilb^{r, x}(C) := \bigsqcup_{\gamma \in \comp{n}}\PHilb^{\gamma,x}(C).
\end{equation}

\begin{remark}
Note that if $\gamma_{i} \leq 1$ for every $i$ then we have a natural isomorphism $\PHilb^{\gamma,x}(C) = \PHilb^{x}(C)$. In particular, $\PHilb^{x}(C)^{\sqcup \binom{r}{n}} \subseteq \FHilb^{r, x}(C)$. Similarly, if there exists $i$ such that $\gamma_{i} = n$ and $\gamma_{j} = 0$ for $j \neq i$ then we have a natural isomorphism $\PHilb^{\gamma,x}(C) = \Hilb(C)$, so that $\Hilb(C)^{\sqcup r} \subseteq \FHilb^{r, x}(C)$. 
\end{remark}

\begin{remark}\label{rmk:x vs y}
Note that we have chosen one projection to define our parabolic Hilbert schemes. We could have instead chosen the other projection so that, for $\gamma \in \comp{m}$ we have the parabolic Hilbert scheme $\PHilb^{\gamma,y}(C)$, where the condition $J^{r} = xJ^{0}$ in \eqref{eq:par hilb comp} is now replaced by $J^{r} = yJ^{0}$. With this, we have
$$
\FHilb^{r, y}(C) := \bigsqcup_{\gamma \in \comp{m}}\PHilb^{\gamma,y}(C)
$$
\end{remark}






\begin{theorem}\label{thm:gieseker intro}
There is an action of the quantized Gieseker variety $\cA_{m/n}(n, r)$ on the localized $\C^*$-equivariant cohomology of $\FHilb^{r, y}(C)$. Moreover, with this action $H^{\C^{*}}_{*}(\FHilb^{r, y}(C))$ gets identified with the unique irreducible $\cA_{m/n}(n, r)$-module 
$$\Lnr_{\frac{m}{n}}(n,r)= \cO(\C)
\otimes \overline{\Lnr}_{\frac{m}{n}}(n,r)
$$
with Gelfand-Kirillov dimension 1 where elements of negative degree act locally nilpotently. The homology of $\PHilb^{\gamma, y}(C)$ is identified with $\gamma$-weight space for the $\mathfrak{gl}(r)$ action on $\Lnr_{\frac{m}{n}}(n,r)$.
\end{theorem}

\begin{remark}
The appearance of the scheme $\FHilb^{r,y}(C)$  in Theorem \ref{thm:gieseker intro} (as opposed to $\FHilb^{r,x}(C)$)
is explained as follows. Note that by Theorem \ref{thm:hilb intro} the algebra $H_{1, n/m}(\Sk{m}, \C^m)$ acts on the equivariant cohomology of $\PHilb^{y}(C)$. The $x$-$y$-switch in Theorem \ref{thm:gieseker intro} is then a geometric incarnation of the $n$-$m$-switch that appears in \cite[Corollary 2.18]{EKLS}. See Section \ref{sect:gieseker} for more details.
\end{remark}

\begin{example}
When $r = 1$, $\comp{m} = \{(m)\}$ and $\FHilb^{1, y}(C) = \Hilb(C)$. Since in this case $\cA_{m/n}(n, 1) = eH_{1,m/n}(\Sk{n}, \C^n)e$, we see that Corollary \ref{cor:hilb intro} is a special case of Theorem \ref{thm:gieseker intro}. 
\end{example}

\begin{example}
When $n=1$, the curve $C$ is smooth, and all the spaces $\PHilb^{\gamma}(C)$ are disjoint unions of $\Z_{\ge 0}$ copies of contractible spaces (labeled by $\dim \cO_C/J^0$). Therefore the homology of $\FHilb^{r, y}(C)$ can be naturally identified with
$$
H^{\C^{*}}_{*}(\FHilb^{r, y}(C))=H^{\C^{*}}_{*}(\mathrm{pt})\otimes \bigoplus_{\comp{m}} \C[X]\simeq H^{\C^{*}}_{*}(\mathrm{pt})\otimes S^m(\C^r)\otimes \C[X].
$$
On the other hand, $\overline{\cA_c}(1,r)$ is isomorphic to a certain quotient of $\mathcal{U}(\mathfrak{sl}(r))$, and $\overline{\Lnr}_{\frac{m}{n}}(n,r)\simeq S^m(\C^r).$
\end{example}

\subsection{Coulomb branches and generalized affine Springer fibers} From the action of a reductive group $G$ on a vector space $N$, Braverman, Finkelberg and Nakajima \cite{BFN} construct an associative algebra called the Coulomb branch algebra, which is modeled after the equivariant homology of the affine grassmannian of $G$, where multiplication is given by convolution. This algebra admits a natural quantization that appears when we take the loop rotation into account for the equivariance. Webster in \cite{W Koszul} generalized their construction by introducing a category of line defects, where the BFN quantized Coulomb branch algebra appears as the endomorphism algebra of an object. Roughly speaking, a line defect consists of the choice of a parahoric subgroup $P \subseteq G_{\KK}$ and a subspace $L \subseteq N_{\KK}$ preserved by $P$. The BFN quantized Coulomb branch algebra corresponds to the choice $P = G_{\OO}$ and $L = N_{\OO}$. It turns out that all of the algebras we work with in this paper appear as BFN quantized Coulomb branches or their generalizations:

\begin{itemize}
\item The spherical Cherednik algebra $eH_{1,c}(\Sk{n}, \C^n)e$ is the BFN quantized Coulomb branch for $G = \operatorname{GL}_{n}$ and $N = \C^n \oplus \mathfrak{gl}_{n}$, \cite{KN, W}. 
\item The full Cherednik algebra $H_{1,c}(\Sk{n}, \C^n)$ appears in the same setting as above, but choosing a nontrivial line defect associated to $P = I$, the standard Iwahori subgroup, and $L = \OO^{n} \oplus \mathfrak{i}$, where $\mathfrak{i}$ is the Lie algebra of the standard Iwahori, \cite{W, LW}.
\item The quantized Gieseker variety $\cA_{c}(n, r)$ is the BFN quantized Coulomb branch for $G = \operatorname{GL}_{n}^{\times r}$ and $N = \C^n \oplus \mathfrak{gl}_{n}^{\oplus r}$. This follows from results of \cite{NT} and \cite{L}. This is an example of {\em symplectic duality}  \cite{W Koszul} since $\cA_{c}(n, r)$ appears both as the quantized Higgs branch for the Jordan quiver and the quantized Coulomb branch for the cyclic quiver with $r$ nodes.
\end{itemize}

The recent paper \cite{HKW} constructs an action of the quantized Coulomb branch in the cohomology of generalized affine Springer fibers in the sense of \cite{GKM}, again by certain convolution diagrams. This has been extended to the parahoric setting in \cite{GK}. We identify the different parabolic Hilbert schemes we consider with generalized affine Springer fibers.

\begin{itemize}
\item For $\Hilb(C)$, this is \cite[Theorem 3.5]{GK}.
\item For $\PHilb^{x}(C)$, see Proposition \ref{prop:gasf}. 
\item For $\FHilb^{r,y}(C)$, see Proposition \ref{prop:gieseker gasf}. 
\end{itemize}

While we take this as a motivation for Theorems \ref{thm:hilb intro} and \ref{thm:gieseker intro}, our proofs do not use any of these technologies, in particular we do not obtain the action via convolution diagrams. The proofs of Theorems \ref{thm:hilb intro} and \ref{thm:hilb t=0 intro} are based on the study of the combinatorics of the various Hilbert schemes we consider, as well as the combinatorial representation theory of the rational Cherednik algebra. The development of this depends on a suitable presentation of this algebra, and we use work of Webster \cite{W}, and more recent work of LePage-Webster \cite{LW} to verify, in the case of the scheme $\PHilb^{x}(C)$, that our action coincides with the one constructed in \cite{GK} via convolution diagrams, see Section \ref{sec: comparison to GK}. 

 On the contrary, there is no known set of generators and relations for the algebra $\cA_{c}(n, r)$. However, we use Theorem \ref{thm:hilb intro} together with \cite[Theorem 2.17]{EKLS} that constructs representations of $\cA_{m/n}(n, r)$ starting from representations of $H_{n/m}(m)$ to prove Theorem \ref{thm:gieseker intro}.





\subsection{Rational Cherednik algebras} The main idea behind the proof
of Theorem \ref{thm:hilb intro} is to identify a basis in
$L_{m/n}(\triv)$ that corresponds to the fixed-point basis in
$H^{\C^{*}}_{*}(\sqcup_{k}\PHilb_{k, n+k}(C))$.  Our main tool to
construct this basis is a presentation of the rational Cherednik
algebra $H_{t,c}(\Sn,\C^n)$ that is better-suited for this purpose than
the usual presentation.
To lighten notation, we write $H_c = H_{t,c}(\Sn,\C^n)$ below.
 Recall that, in its usual presentation, the
algebra $H_c$ has generators $x_i$, $y_i$ ($i = 1, \dots, n$) and
$\Sn$. It is naturally graded, with $x_i$ of degree 1, $y_i$ of degree
$-1$ and $\Sn$ in degree zero. Dunkl and Opdam \cite{DO} constructed a
family of commuting operators $u_1,\ldots,u_n$ of degree 0 in $H_c$.
The algebra $H_c$ is, in fact, generated by $u_i$, the group algebra of
$\Sn$ and two additional generators 
$$
\tau := x_{1}(12\cdots n), \lambda := (12\cdots n)^{-1}y_{1}.
$$
It is clear that $\tau$, $\lambda$ and $\Sn$ already generate the algebra since one can obtain $x_1$ and $y_1$ (and hence all $x_i$ and $y_i$) using them. In Theorem \ref{thm:new presentation} we give a complete list of relations between $\tau, \lambda, u_i$ and the generators of $\Sn$.  This presentation of the algebra $H_{c}$ has already appeared in the more complicated cyclotomic setting in the work of Griffeth \cite{g2} and Webster \cite{W}.
 Since some relations become more transparent in the type $A$ setting, we present it in detail. The generators $u_i$ can be, in principle, eliminated, and the remaining relations 
are listed in Proposition \ref{prob: eliminate u}.

We use the presentation of the algebra $H_{c}$ via the Dunkl-Opdam operators to, in the case where $c$ is a rational number with denominator precisely $n$, simultaneously diagonalize the operators $u_{i}$ on the polynomial representation $\Delta_{c}(\triv)$ and give an explicit combinatorial description of the eigenvalues. We prove that the action of the operators $\tau$ and $\lambda$ sends an eigenvector to a multiple of another eigenvector, and describe the action of $\Sn$ on an eigenbasis explicitly. We remark that this has already appeared in work of Griffeth,  see \cite{griffeth comb, griffeth, g2} and Remark \ref{remark:griffeth}, but we reprove these results with combinatorics that are more amenable to our geometric goal. 


\begin{theorem}
\label{thm: intro 1}
Let $c=m/n,$ where $m,n \in \Z_{> 0}, \gcd(m,n)=1$.
Then the following holds:
\begin{itemize}
\item[(a)] $\Delta_c(\triv)$ has a basis $v_{\ba}$ labeled by  sequences $\ba=(a_1,\ldots,a_n)$ of nonnegative integers. 
The action of $u_i,\tau$ and $\lambda$ in this basis is given by
$$
u_i v_{\ba}=(a_{i} - (g_{\ba}(i) - 1)c)v_{\ba},\ \tau v_{\ba}=v_{\pi\cdot \ba},\ \lambda v_{\ba} = (a_{1} - (g_{\ba}(1) - 1)c)v_{\pi^{-1}\cdot \ba}
$$
where $\pi(a_1,\ldots,a_n)=(a_n+1,a_1,\ldots,a_{n-1})$ and $g_\ba$ is the minimal length permutation sorting the sequence $\ba$ to
be non-decreasing. 
\item[(b)] $L_c(\triv)$ has a basis $v_{\ba}$ labeled by sequences $(a_1,\ldots,a_n)$ such that $|a_i-a_j|\le m$ for all $i,j$ and
if $a_i-a_j=m$ then $i<j$. 
\end{itemize}
The action of $\Sn$ in the basis $v_\ba$ is given in Theorem \ref{thm:action}.
\end{theorem}

\begin{remark}
Note that $\pi^{-1}\cdot \ba$ is well defined unless $a_1=0$. In this case $a_{1} - (g_{\ba}(1) - 1)c=0$, so $\lambda\cdot v_{\ba}$ is well defined.
\end{remark}

The proof of Theorem \ref{thm:hilb intro} is based, roughly speaking, on the comparison of the basis of fixed points in $H^{*}(\sqcup_{k}\PHilb_{k, n+k}(C))$ with the basis given by Theorem \ref{thm: intro 1}(b). 

The proof of  Theorem \ref{thm: intro 1} uses a Mackey-type result for the algebra $H_{c}$. The algebra $\Hnu$ generated by $u_{1}, \dots, u_{n}$ and $\Sn$ is isomorphic to the degenerate affine Hecke algebra of rank $n$. 
In Theorem \ref{theorem-mackey} we construct a filtration of $\Res^{H_{c}}_{\Hnu}\Delta_{c}(\mu)$ by $\Hnu$-modules and explicitly describe the subquotients. As a consequence, we are able to give a combinatorial basis of all standard modules $\Delta_{c}(\mu)$.


\begin{remark}
\label{remark:griffeth}
In \cite[Theorem 5.1]{griffeth} Griffeth constructs, for generic values of the parameter $(t,c)$  an eigenbasis of every standard module, and in \cite{griffeth comb} he considers the case of the polynomial representation.  Both Theorem \ref{thm: intro 1} and the construction of a combinatorial basis for standard modules are a consequence of this 
and \cite[Theorem 7.5]{griffeth}
after specializing parameters. Our proof and construction of eigenbasis, using a Mackey-type formalism, is more conceptual and its combinatorics seem better-suited for geometric applications.
\end{remark}

As a further application of the combinatorics of the Dunkl-Opdam
presentation of the algebra $H_{c}$ we are able to give an explicit
combinatorial construction of all the maps appearing in the BGG
resolution of the module $L_{c}(\triv)$
for $c = m/n$,
and we show that the complex
formed by these maps is indeed exact. In particular, we give a new
construction of this resolution that avoids appealing to the
representation theory of finite Hecke algebras at roots of unity via
the Knizhnik-Zamolodchikov functor, which uses techniques of complex
analysis. Moreover, we are able to give a combinatorial basis in the
spirit of that of Theorem \ref{thm: intro 1} for \emph{every} simple
module $L_{c}(\mu)$, see Corollary \ref{cor:weights of simples}. 

\begin{remark}
More concretely, the standard modules $\Delta_c(n-\ell,1^{\ell})$ and
$\Delta_c(n-\ell+1,1^{\ell-1})$ have bases labeled by pairs $(\ba,T)$
and $(\ba',T')$ where $T$ and $T'$ are standard tableaux of the
corresponding hook shapes. We explicitly compute matrix elements of the
map between standard modules in this basis in the case $c=m/n$. As a consequence, we give
two labelings of the basis in  $L_c(n-\ell,1^{\ell})$ presented either
as a simple quotient of $\Delta_c(1^{\ell},n-\ell)$, or as the radical
of $\Delta_c(n-\ell-1,1^{\ell+1})$, and an explicit bijection between
them.  \end{remark}

\subsection{Relation to other work} Finally, we would like to comment on the relations of our work to the existing literature. As we have mentioned above, the Dunkl-Opdam presentation of the Cherednik algebra has already appeared in work of Griffeth and Webster, \cite{griffeth, g2, W}, where it has been used for different purposes. In particular, Griffeth \cite{griffeth, g2} uses the fact that the operators $u_{i}$ are self-adjoint with respect to the Shapovalov form to compute the norm of elements in standard modules, see also \cite{dunkl griffeth}, while Webster \cite{W} uses the Dunkl-Opdam subalgebra to give a concrete equivalence between the category $\cO$ and modules over the quiver Hecke algebra.

By \cite{MY,MS} the homology of the Hilbert schemes of singular curves is closely related to the homology of the corresponding compactified Jacobian which is  isomorphic to the 
 affine Springer fiber in the affine Grassmannian. One would expect a similar connection between our parabolic Hilbert schemes and affine Springer fibers in the affine flag variety.
 These affine Springer fibers do admit affine pavings, and the combinatorics of the affine cells was studied in detail in \cite{LS,ORS,GMV}.

The (co)homology of the  affine Springer fibers in affine flag variety was studied  in \cite{GORS,OY,OY2,VV} where it was proved that it carries an action of the trigonometric Cherednik algebra. Furthermore, this (co)homology has certain ``perverse'' filtration, and the associated graded space admits a natural action of the rational Cherednik algebra corresponding to the {\em reflection representation} of $\Sn$ (also known as rational Cherednik algebra of $\sln$). The construction of this action uses global Springer theory developed by Yun \cite{Yun}.  The combinatorics of finite dimensional representations of the rational Cherednik algebra for $\sln$ was studied by Shin \cite{gicheol}.

On the contrary, we find our construction to be more elementary than \cite{OY,OY2}. Indeed, in our construction of geometric operators $\tau$ and $\lambda$ we use neither perverse filtration nor global Springer theory. The combinatorial presentation of the algebra is easier in the $\gln$ setup. Still, we make an explicit comparison with the results of \cite{GMV} in Section \ref{sect:radical triv}, see Remark \ref{rmk:GMV}.




\subsection{Structure of the paper} The main body of the paper follows a reverse structure from the introduction. First we study the representation theory of rational Cherednik algebras and then we move on to Hilbert schemes. In Section \ref{sec: affSn} we look at the combinatorics of the affine symmetric group that will appear in the representation theory of the rational Cherednik algebra. We study the structure of the Cherednik algebra in Section \ref{sec: RCA background} and, in particular, we give its alternative presentation that is better suited for geometry. Sections \ref{sec-mackey}, \ref{sec:rep theory} and \ref{sect:pol rep} are devoted to the study of modules in category $\mathcal{O}$ of the Cherednik algebra. In Section \ref{sec-mackey} we prove a Mackey-type formula for induced representations of the Cherednik algebra and apply it to study standard modules in Section \ref{sec:rep theory}.  In particular we construct the BGG resolution of \cite{BEG} in a purely combinatorial manner. Since the polynomial representation occupies a special place in geometric applications, we specialize the results of Section \ref{sec:rep theory} to the polynomial representation in Section \ref{sect:pol rep}. 

We turn to Hilbert schemes in Section \ref{sect:geometry}. First, we examine the case of the reduced curve $C = \{x^{m} = y^{n}\}$ and prove Theorem \ref{thm:hilb intro}, see Theorem \ref{thm: geometric action}.
 In this section, we also compare the parabolic Hilbert scheme to generalized affine Springer fibers, in particular proving that they admit a paving by affine cells, see Section \ref{sec:gasf}. Section \ref{sect:gieseker} is devoted to the compositional parabolic Hilbert scheme $\FHilb^{r,y}(C)$.  We prove Theorem \ref{thm:gieseker intro} as Theorem \ref{thm:gieseker action} and also realize this Hilbert scheme as a generalized affine Springer fiber. Finally, we study the case of the non-reduced curve $C = \{y^{n} = 0\}$ in Section \ref{sect:m to infinity} where we prove Theorem \ref{thm:hilb t=0 intro}, see Theorem \ref{thm:hilb t=0}.

\section*{Acknowledgments}

We would like to thank Tudor Dimofte, Niklas Garner, Joel Kamnitzer, Oscar Kivinen, Ivan Losev, Alexei Oblomkov and Ben Webster for useful discussions. We would also like to thank Stephen Griffeth for comments on the relationship of this paper with some of his previous work. We are grateful to Sean Griffin for spotting a gap in a previous proof of Proposition \ref{prop:gasf}.
We would also like to thank the anonymous referees for their helpful suggestions on improving the organization and exposition of this paper.

\section{Affine permutations}\label{sec: affSn}

\subsection{The extended affine symmetric group}\label{sect:affine sn}

In this section, we study the combinatorics of the extended affine symmetric group $\affSn$. For more details, see \cite{green, Macdonald}. 

\begin{definition} The {\em extended affine symmetric group}
 is defined by the relations\footnote{We drop the first relation when $n=2$.}
$$\affSn = \left\langle \pi, s_i, \,\,  1\le i < n \quad\Big|\quad
\begin{array}{ll}
s_is_{i+1}s_i = s_{i+1}s_is_{i+1} &\textrm{for $1 \le i < n-1$},\\
s_is_j=s_js_i &\textrm{for $j \neq i \pm 1 $},\\
\pi s_i = s_{i+1}\pi &\textrm{for $1 \le i < n-1$,}
\\
s_i^2  = 1 &\textrm{for $i\in\Z/n\Z$}
\end{array}\right\rangle.$$
\end{definition}

Letting $s_0 = \pi^{-1} s_1 \pi$, we could consider generators
$s_i$ for $i \in \Z/n\Z$. In this case
$\affSn $  has as a subgroup the affine symmetric group
$\affSnCox  = \langle s_i \mid i \in \Z/n\Z \rangle$.
However for the purposes of this paper, we rarely take this
point of view.
Further, we will refer to elements $\p \in \affSn$ as affine
permutations, dropping the adjective ``extended."

We recall that $\affSn $ acts faithfully on $\Z$ by $n$-periodic permutations,
i.e. bijections $\p:\Z\to\Z$ such that $\p(i+n) =
\p(i)+n$.  For this action $\pi(i) = i+1.$
 It also acts on the set $\C^n$ via:
\begin{align}
s_i\cdot(\wt_1,\ldots,\wt_i,\wt_{i+1},\ldots \wt_n) &= (\wt_1, \ldots, \wt_{i+1}, \wt_i, \ldots, \wt_n)\nonumber
\\
\pi \cdot (\wt_1, \ldots, \wt_n) &= (\wt_n +t, \wt_1, \wt_2, \ldots, \wt_{n-1})
\label{AffSymmOnCn}
\end{align}
for a fixed parameter $t \in \C$.
It is  convenient to extend
an $n$-tuple to having coordinates indexed by all of $\Z$
via $\wt_{i+kn} = \wt_i - kt$. Then we may align the two actions, writing
$\p \cdot (\wt_1, \ldots, \wt_n) = (\wt_{\p^{-1}(1)}, \wt_{\p^{-1}(2)},  \ldots, \wt_{\p^{-1}(n)})$.
This is consistent with our
 conventions taken in Remark \ref{remark-ui for i in Z} below.

Just as with the finite symmetric group, it is convenient to
use window notation for affine permutations.
The window notation of $\p$ is given by $[p(1), p(2), \ldots, p(n)]$,
which determines $\p$ by periodicity.

We will follow the usual convention to extend Coxeter length $\ell$ from $\affSnCox$ to $\affSn$ by setting $\ell(\pi)= 0$.  Then length still counts the number of affine inversions, that is $\ell(\p) = \# \{(i,j) \in \Z^2 \mid 1 \le i \le n, i< j, \p(i) > \p(j) \}$.
\begin{definition}
Let us define the {\em degree}  of $\p \in \affSn$ to be
$\frac 1n  \sum_{i=1}^n (\p(i) -  i  )$.
\\
Let 
$\affSnplus$ denote the submonoid of affine permutations
$\p$ such that $i > 0 \implies \p(i) > 0$,
i.e., the entries of $\p$ in window notation are all positive.
\end{definition}

Note  $\p \in \affSnCox$ iff it has degree 0.
The only permutations in $\affSnplus$ of degree $0$ are those
in the finite symmetric group $\Sn$.

Let $\ta{a} \in \affSn$ denote {\em translation} by ${\ba} \in \Z^n$.
In other words, its window notation  is
$\ta{a} = [1+n a_1, 2+n a_2, \ldots, n+n a_n]$.
It is well-known that the
subgroup $\{ \ta{a} \mid \ \ba \in \Z^n\}$ generated by translations is normal in $\affSn$ and it is isomorphic to $\Z^n$. Moreover, there is a semi-direct product decomposition $\affSn \simeq \Z^n \rtimes \Sn$. This gives us the first part of the following Lemma, while the second assertion is obvious from the definitions.

\begin{lemma}
\label{lem:coset positive}
Any permutation $\ww\in \affSn$ can be uniquely written as $\ww=\ta{a} g$ for $g\in \Sn$, $\ba \in \Z^n$.
Furthermore, $\ww\in \affSnplus$ if and only if  $\ww=\ta{a} g$ and $a_i\ge 0$ for all $i$.
\end{lemma}

Let $\sort(\ba)$ denote the non-decreasing ordering of $\ba$, and $g_{\ba} \in \Sn$ the shortest element such that $g_{\ba} \cdot \ba = \sort(\ba)$. Note that the element $g_{\ba}$ is given by
\begin{equation}
\label{eq: def ga}
g_{\ba}(i) = \sharp\{j : a_{j} < a_{i}\} + \sharp\{j : j \leq i \, \text{and} \, a_{i} = a_{j}\}.
\end{equation}
We denote 
 $$
\wa{a}:=\ta{a} g_{\ba}^{-1}.
$$

\begin{remark}
Note that the element $g_{\ba} \in \Sn$ is uniquely specified by the requirement that the window notation of $\ta{a}g_{\ba}^{-1}$ is increasing. 
\end{remark}

The following proposition follows easily from the semi-direct product decomposition of $\affSn$. In other words, we have a  projection $\affSn \to \Z^n$ and the assignment $\ba \to \wa{\ba}$ is a right inverse to this map. 
\begin{proposition} 
$\wa{a}=\wa{b}$ if and only if  $\ba=\boldb$.
\end{proposition}

In fact, the following stronger statement holds. 

\begin{lemma}\label{lemma:w is very injective}
Assume $\wa{a}(i) = \wa{b}(i)$ for some $i \in \{1, \dots, n\}$. Then $g_{\ba}^{-1}(i) = g_{\boldb}^{-1}(i)$ and $a_{g_{\ba}^{-1}(i)} = b_{g_{\boldb}^{-1}(i)}$.
\end{lemma}
\begin{proof}
If $\wa{a}(i) = \wa{b}(i)$ then $g_{\ba}^{-1}(i) - g_{\boldb}^{-1}(i) = n(b_{g_{\boldb}^{-1}(i)} - a_{g_{\ba}^{-1}(i)})$. But $g_{\ba}^{-1}(i)$
and $ g_{\boldb}^{-1}(i) \in \{1, \dots, n\}$ so their difference is only divisible by $n$ if it is in fact $0$. The result follows. 
\end{proof}

Let $\Min$ denote the set $\Min = \{ \wa{a} \in \affSn \mid \ba \in \Z_{\ge 0}^n \}$.
Then $\ww = \wa{a} \in \affSnplus$ is a minimal length (left) coset representative
of $\affSn / \Sn$, i.e.,  we have $0 < \ww(1) < \ww(2) < \cdots < \ww (n)$.
Further note that the degree of $\wa{a}$ as well as that of $\ta{a}$
agrees with $||\ba|| := \sum a_{i}$.

It is easy to see the following holds.
\begin{lemma}
\label{claim:coset}
Let $\ww \in \Min$ be of degree $r > 0$.
Then there is a unique expression of the form
$$\ww =  
 (s_{\nu_r} \cdots s_2 s_1) \pi \cdots
 (s_{\nu_2} \cdots s_2 s_1) \pi
 (s_{\nu_1} \cdots s_2 s_1) \pi,$$
where $ 0\le \nu_{i+1}  \le\nu_i.$  In other words $\nu$ is a
partition with $\nu_1 < n$ and $r$ parts,
and we allow parts to be zero.
\end{lemma}
\begin{proof}
Let us induct on $r$, noting we exclude the case $r=0$ from the
hypotheses. 
This corresponds to $\ww = \id$.
For $r=1$ consider the window notation $\ww = [\ww(1), \cdots , \ww(n)]$. Recall $0 < \ww(1) < \ww(2) < \cdots < \ww(n)$.
In particular $0 \le \ww(i) -i$ but $n < \ww(n) $ since $\ww \neq \id$.
Since the degree of $\ww$ is $1$, $n = \sum_{i=1}^n (\ww(i) - i)$
which forces $\ww(n) \le 2n$ and hence $0 < \ww(n)-n \le n$.
Then $\ww \pi^{-1} =[\ww(n)-n, \ww(1), \cdots, \ww({n-1})] \in 
\affSnplus$ has degree $0$ and so $\ww \pi^{-1} \in \Sn$.
Let $k$ be maximal such that 
$\ww(k) < \ww(n)-n$ and $0$ otherwise, in which
case we have $\ww = \pi$. 
Then $\ww \pi^{-1} s_1 s_2 \cdots s_k \in \Sn \cap \Min = \{\id\}$
which implies $\ww = s_k \cdots s_2 s_1 \pi$. This proves the
base case. 

Next assume the claim holds for all affine permutations in
$\Min$ of degree $< r$.  Suppose $\ww$ has degree $r$.
Choose $k$ exactly as above, and note
 $\p = \ww \pi^{-1} s_1 s_2 \cdots s_k
\in \Min$ is of degree $r-1$. By the inductive hypothesis, the
claim holds for $\p$ with respect to a partition with $r-1$ parts
we renumber as $n > \nu_2 \ge \nu_3 \ge \cdots \ge \nu_r \ge 0$.
Thus 
$$\ww = \underbrace{(s_{\nu_r} \cdots s_2 s_1) \pi \cdots
 (s_{\nu_2} \cdots s_2 s_1) \pi}_{\p}
 (s_{k} \cdots s_2 s_1) \pi.$$
We need only show $k \ge \nu_2$ and then set $\nu_1 = k$.
Recall $\nu_2$ is maximal such that $\p({\nu_2}) < \p(n)-n$ and
recall $k$ is maximal such that $\ww({k}) < \ww(n)-n$.
By choice of $k$ we have $\p(k) = \ww(k) < \ww(n)-n$ and
$\p({k+1}) = \ww(n)-n$.
If $\nu_2 = k+1$ then $\p({\nu_2}) = \p({k+1}) = \ww(n)-n \ge \p(n)-n$
and if 
$\nu_2 > k+1$ then $\p({\nu_2}) = \ww({\nu_2 - 1}) > \ww(n)-n \ge \p_n-n$,
both of which are contradictions.
\end{proof}

\begin{remark}
Given $\ww \in \Min$, the partition $\nu$ can easily
be obtained from the inversions of $\ww$ as follows.
For the transposed partition $\nu^T$ which has $n-1$ parts,
$\nu^T_{i} = \#\{k < 1 \mid \ww(k) > i    \}$.
Observe the length $\ell(\ww) = |\nu|$. 
\end{remark}
\begin{example}
Let $n=5$, $\ba = (0,2,0,0,1)$. Thus $g_{\ba}=[1,5,2,3,4]$, $g_{\ba}^{-1}=[1,3,4,5,2]$, $\ta{a}=[1,12,3,4,10]$ and $\wa{a}=\ta{a}g_{\ba}^{-1} = [1,3,4,10,12]
$ which has reduced word  $ s_1 \pi s_3 s_2 s_1 \pi  s_3 s_2 s_1 \pi$, and so
$\nu = (3,3,1)$, $\nu^T = (3,2,2)$.
Note  $\ww(0)=7, \ww(-1)=5, \ww(-5)=2$ and
\begin{align*}
&
\{k < 1 \mid \ww(k) > 1    \} &&= \{0,-1,-5\} & &\nu^T_1 = 3&
\\
&
\{k < 1 \mid \ww(k) > 2    \} &&= \{0,-1\}&  &\nu^T_2 = 2&
\\
&
\{k < 1 \mid \ww(k) > 3    \} &&= \{0,-1\}& &\nu^T_3 = 2&
\\
&
\{k < 1 \mid \ww(k) > 4    \} &&= \emptyset& &\nu^T_4 = 0&
\\
&
\{k < 1 \mid \ww(k) > 5=n    \} &&= \emptyset& &\nu^T_5 = 0.&
\end{align*}
There are other ways to obtain $\nu$ from $\ba$, 
but discussing them is beyond the scope of this paper.
We will merely mention without proof one such  way.
Given $\ba$ construct its $n$-abacus (with beads at heights
determined by $\ba$) and then its corresponding $n$-core
partition. Next, following Lapointe-Morse \cite{LM}, 
remove all boxes from the $n$-core with hooklength $> n$
and left-justify the remaining boxes.
For the $\ba$ given above its $5$-core is $(4,3,1)$, from which
we remove its box in the upper left corner with hooklength $6$
leaving us with $\nu = (3,3,1)$.
\end{example}

\subsection{\texorpdfstring{$m$}{m}-stable and \texorpdfstring{$m$}{m}-restricted permutations}
\label{section: m-stable}

Here we recall some facts on $m$-stable and $m$-restricted affine
permutations from \cite{GMV}.

\begin{definition}(\cite{GMV})
We call an affine permutation $\omega$ $m$-{\em stable} if the inequality $\omega(x+m)>\omega(x)$ holds for all $x$.
We call an affine permutation $\omega$ $m$-{\em restricted} if for all $j<i$ one has $\omega(j)-\omega(i)\neq m$.
\end{definition}

It is clear that $\omega$ is $m$-stable if and only if $\omega^{-1}$ is $m$-restricted. Also, $\omega$ is $m$-stable if and only if
$$
\omega(\omega^{-1}(i)+m)>i\ \textrm{for}\ i=1,\ldots,n.
$$

\begin{definition}
We call a subset $\M\subset \Z$ $(m,n)$-{\em invariant} if $\M+n\subset \M$ and $\M+m\subset \M$.
\end{definition}

If $\omega \in \affSn$ is an $m$-stable permutation then for all $i$ the set 
$$
\M_{\omega}^{i}=\{x\in \Z: \omega(x)\ge i\}=\omega^{-1}[i,+\infty).
$$
is $(m,n)$-invariant. Indeed, if $\omega(x)\ge i$ then $\omega(x+n)=\omega(x)+n>i$ by definition of affine permutation and $\omega(x+m)>\omega(x)\ge i$ because $\omega$ is $m$-stable. 

Clearly, $\M_{\omega}^{i+n}=\M_{\omega}^{i}+n$ and $\omega$ is $m$-stable if and only if $\M_{\omega}^{i}$ is $(m,n)$-invariant for all $i$. This implies the following useful  proposition.

\begin{proposition}
\label{prop: m stable}
An affine permutation $\omega$ is $m$-stable if and only if the sets $\M_{\omega}^{i}$ are $(m,n)$-invariant for $i=1,\ldots,n$. 
\end{proposition}

Next, we would like to characterize $m$-stable and $m$-restricted permutations using window notation, assuming $\gcd(m,n)=1$. As in \cite{GMV}, we use the affine permutation
$$
\perm:=[0,m,\ldots,(n-1)m].
$$

\begin{lemma}
\label{lem: omega omegam}
Let $\omega\perm=[x_1,\ldots,x_n]$. Then $\omega$ is $m$-stable if and only if
$$
x_1\le x_2\le \cdots \le x_n\le x_1+mn.
$$
\end{lemma}

\begin{proof}
It is sufficient to check the condition $\omega(x+m)>\omega(x)$ for a single choice of $x$ in each remainder modulo $n$, in particular, for $x=0,m,2m,\ldots,(n-1)m$. Now for $1\le i\le n$ we have
$x_i=\omega(\perm(i))=\omega((i-1)m)$, so $\omega$ is $m$-stable if $x_1<\ldots<x_n$ and 
$$
x_n=\omega((n-1)m)<\omega(nm)=\omega(0)+nm=x_1+mn.
$$
\end{proof}

The condition $x_1<\ldots<x_n$ implies that  we can write
$$
\omega\perm=\ta{a}g_{\ba}^{-1},\ \omega=\ta{a}g_{\ba}^{-1}\perm^{-1},\ \omega^{-1}=\perm g_{\ba}\ta{-a}
$$
for some vector $\ba\in \Z^n$, and $g_{\ba}$ as above. 
We can write 
$$
\omega^{-1}(g^{-1}_{\ba}(i))=\perm(-na_{g^{-1}_{\ba}(i)}+i)=-na_{g^{-1}_{\ba}(i)}+m(i-1),
$$
so
\begin{equation}
\label{eq: omega alpha beta}
\omega^{-1}(i)=-na_i+m(g_{\ba}(i)-1),\ i=1,\ldots,n.
\end{equation} 
Hence, in window notation $\ww^{-1} =
[-n a_1 +m(g_{\ba}(1)-1), \ldots, -n a_n +m(g_{\ba}(n)-1)]$.



We get the following result:

\begin{lemma}
\label{lem: m stable in window notation}
Let $\gcd(m,n)=1$. A permutation $\omega$ is $m$-stable if and only if $\omega^{-1}$ can be written in the form \eqref{eq: omega alpha beta} for some vector $\ba \in \Z^n$ such that:
\begin{itemize}
\item $a_i-a_j\le m$ for all $i,j$
\item if $a_i-a_j=m$ then $i<j$.
\end{itemize}
\end{lemma}

\begin{proof}
Since $\omega\perm=\ta{a}g_{\ba}^{-1}=[x_1,\ldots,x_n]$, we get $x_1<\ldots<x_n$. We need to check the last condition $x_n<x_1+mn$ in terms of the vector $\ba$. 

Observe $x_i=na_{g_{\ba}^{-1}(i)}+g_{\ba}^{-1}(i)$,  so
 $x_n<x_1+mn$ if and only if either $a_{g_{\ba}^{-1}(1)}+m>a_{g_{\ba}^{-1}(n)}$ or  $a_{g_{\ba}^{-1}(1)}+m=a_{g_{\ba}^{-1}(n)}$ and $g_{\ba}^{-1}(1)>g_{\ba}^{-1}(n)$. 

Now $a_{g_{\ba}^{-1}(1)}=\min(\ba), a_{g_{\ba}^{-1}(n)}=\max(\ba)$, so either $\max(\ba)-\min(\ba)<m$ or
$\max(\ba)-\min(\ba)=m$ and all appearances of $\max(\ba)$ are to the left of all appearances of $\min(\ba)$ in $\ba$.
\end{proof}

\begin{remark}
\label{rem: balancing}
The above results were stated in \cite{GMV} for the affine symmetric group $\affSnCox$ (as opposed to extended affine $\affSn$), 
but are equivalent to them after imposing the balancing condition for all affine permutations. In particular,
$\perm$ should be replaced by the degree $0$ affine permutation
 $\widehat{\p}_m=[0-\kappa,m-\kappa,\ldots,(n-1)m-\kappa]$ where $\kappa=\frac{1}{2}(mn-m-n-1)$.
In particular, Lemma \ref{lem: omega omegam} can be rephrased by saying that $\widehat{\p}_m$ establishes a bijection between the set of $m$-stable affine permutations and the dilated fundamental alcove.
\end{remark}

\begin{example}
Let $n=5, m=3, \ba=(0,1,0,0,2)$.
Thus $ g_\ba^{-1} = [1,3,4,2,5]$, $\wa{a} = [1,3,4,7,15]$,
with inverses
$\wa{a}^{-1} = [1,-1,2,3,-5]$ and $g_\ba=[1,4,2,3,5]$ .
Note 
$$
\omega^{-1}=\perm\wa{a}^{-1}=[0,3,6,9,12] \circ [1,-1,2,3,-5] = [0,4,3,6,2]
$$
is $3$-restricted.
Using \eqref{eq: omega alpha beta} we can also check
$
\omega^{-1}(i)=-5a_i+3(g_{\ba}(i)-1)$ as
$$
(0,4,3,6,2)
=-5(0,1,0,0,2) + 3(0,3,1,2,4)
=$$ $$-5(0,1,0,0,2) + 3\left( (1,4,2,3,5)-(1,1,1,1,1)\right).
$$
\end{example}

\subsection{Lexicographic ordering and combinatorics of integer sequences}\label{sect:combinatorics}

Recall that for $\ba \in \Z_{\geq 0}^{n}$, we denote $||\ba|| := \sum_{i} a_{i}$. As in Section \ref{sect:affine sn}, we denote by $g_{\ba} \in \Sn$ the shortest element such that $g_{\ba}\cdot \ba = \sort(\ba)$.

\begin{lemma}\label{lemma:g vs pi}
For every $\ba \in \Z^{n}_{\geq 0}$ we have $g_{\pi\cdot \ba} = g_{\ba}(12\cdots n)^{-1}$. If $a_{i} \neq a_{i+1}$,  then we have $g_{s_{i}\cdot \ba} = g_{\ba}s_{i}$.
\end{lemma}
\begin{proof}
We use the explicit equation   \eqref{eq: def ga} for $g_{\ba}$.
Assume $i \neq 1$. Denote $X_{\pi} := \{j : (\pi\cdot \ba)_{j} < (\pi\cdot \ba)_{i}\}$ and $Y_{\pi} := \{j : (\pi\cdot \ba)_{j} = (\pi\cdot \ba)_{i} \; \text{and} \; j \leq i\}$. Similarly, denote $X := \{j : a_{j} < a_{i-1}\}$ and $Y := \{j : a_{j} = a_{i-1} \; \text{and} \; j \leq i-1\}$, so that $g_{\pi\cdot \ba}(i) = \sharp X_{\pi} + \sharp Y_{\pi}$ and $g_{\ba}(i-1) = \sharp X + \sharp Y$. Note that, if $j \neq 1$, then $j \in X_{\pi}$ (resp. $j \in Y_{\pi}$) if and only if $j-1 \in X$ (resp. $j-1 \in Y$). Note also that we cannot have $n \in Y$ because $i - 1 < n$. Moreover, we have that $1 \in X_{\pi} \cup Y_{\pi}$ if and only if $n \in X$ and, by the previous sentence, this happens if and only if $n \in X \cup Y$. This shows that $g_{\pi\cdot \ba}(i) = g_{\ba}(i-1)$. Note that this forces $g_{\pi\cdot \ba}(1) = g_{\ba}(n)$. So $g_{\pi\cdot \ba} = g_{\ba}(12\cdots n)^{-1}$, as needed. The other equality is clear.
\end{proof}

It is easy to see that the assignment $\ba \mapsto \omega_\ba$ gives a bijection between $\Z^n_{\geq 0}$ and the set $\Min$. More precisely, let us denote by $\mathcal{P}_{k}(n)$ the set $\{\ba \in \Z^n_{\geq 0} : ||\ba|| = k\}$. Inside, we have the sets

$$
\mathcal{P}_{k}^{\circ}(n) := \{\ba \in \mathcal{P}_{k}(n) : a_{1} \neq 0\}, \qquad \mathcal{P}'_{k}(n) := \{\ba \in \mathcal{P}_{k}(n) : a_{1} = 0\},
$$

\noindent so that $\cP_{k}(n) = \cP^{\circ}_{k}(n)\sqcup \cP'_{k}(n)$. Then, we have the following result (see also Lemma \ref{lemma:w is very injective}).

\begin{lemma}\label{lemma:w is injective}
The assignment $\ba \mapsto \wa{a}$ gives a bijection between $\Z^{n}_{\geq 0}$ and the set $\Min$. 
Moreover, 

\begin{itemize}
\item[(a)] the set $\mathcal{P}_{k}(n)$ gets identified with
$$
\Min_{k} := \{\omega \in \Min : \deg\omega = k\}.
$$
\item[(b)] $\mathcal{P}'_{k}$ gets identified with $\left\{\omega \in \Min_{k} : \omega(1) = 1\right\}$.
\item[(c)] $\mathcal{P}^{\circ}_{k}$ with $\left\{\omega \in \Min_{k} : \omega(1) > 1\right\}$. 
\end{itemize}
\end{lemma}

For affine permutations $\omega, \omega' \in \affSn$, we say that $\omega \lex \omega'$ if the window notation of $\omega$ is greater than that of $\omega'$ in lexicographic ordering. More explicitly, $\omega \lex \omega'$ if there exists $i \in \{1, \dots, n\}$ such that $\omega(j) = \omega'(j)$ for $j = 1, \dots, i -1$ and $\omega(i) > \omega'(i)$. 

We would like to study this partial order in more detail. In particular, we will see how it translates to $\Z^{n}_{\geq 0}$ under the bijection $\ba \mapsto \omega_{\ba}$. 
In order to do this, let us define a partial order on $\cP_{k}(n)$ inductively. For $n = 2$ and even $k = 2\ell$, we have
$$
(\ell, \ell) \prec (\ell+1, \ell-1) \prec (\ell-1, \ell+1) \prec \cdots \prec (2\ell, 0) \prec (0, 2\ell)
$$

\noindent and for $k = 2\ell + 1$ odd we have
$$
(\ell+1, \ell) \prec (\ell, \ell+1) \prec (\ell+2, \ell-1) \prec (\ell-1, \ell+2) \prec \cdots \prec (2\ell+1, 0) \prec (0, 2\ell+1).
$$

Now assume we have defined partial orders on $\cP_{k'}(n)$ for every $k'$. Let us define partial orders on $\cP_{k}(n+1)$. The set $\cP_{0}(n+1)$ is a singleton so there is nothing to do. On $\cP_{1}(n+1)$ we have $(1, 0, \dots, 0) \prec (0, 1, \dots, 0) \prec \cdots \prec (0, 0, \dots, 1)$. Assume that we have defined a partial order on $\cP_{k}(n+1)$. To define a partial order on $\cP_{k+1}(n+1)$, recall that we have a decomposition 
$$
\cP_{k+1}(n+1) := \cP_{k+1}^{\circ}(n+1) \sqcup \cP_{k+1}'(n+1).
$$

\noindent The map $\pi$ gives a bijection $\pi:
\cP_{k}(n+1) \to \cP^{\circ}_{k+1}(n+1)$, and this gives a partial
order on the set $\cP_{k+1}^{\circ}(n+1)$. By forgetting $a_{1} = 0$,
$\cP_{k+1}'(n+1)$  is identified with $\cP_{k+1}(n)$, and this gives a
partial order on the set $\cP_{k+1}'(n+1)$. Finally, we declare every
element in $\cP_{k+1}^{\circ}(n+1)$ to be smaller than every element
of $\cP_{k+1}'(n+1)$ This gives a partial order on $\cP_{k+1}(n+1)$.
Figure \ref{fig:order} below 
gives some examples on how these partial orders look when
$n = 3$. For each $\ba \in \Z_{\geq 0}^{3}$ listed,
we also include $\wa{a}$,
both in window notation and in its decomposition given by Lemma
\ref{claim:coset} for reference. 


\begin{figure}[ht!]
\begin{alignat*}{9}
        &\id \\
        & [1,2,3]\\
k = 0 \qquad & (0,0,0) \\
\\
& \pi && s_1 \pi && s_2 s_1 \pi \\
&   [2,3,4] && [1,3,5] && [1,2,6]            \\
k = 1 \qquad & \tikzmark{StartB1} (1,0,0) \tikzmark{EndB1}&
                \;  \prec \; & 
\tikzmark{StartC1}(0,1,0) 
& \prec &
(0,0,1) \tikzmark{EndC1}\\
\\
        & \pi^2 && \pi s_1 \pi &&  \pi s_2 s_1 \pi && 
s_1 \pi s_1 \pi &&  s_1 \pi  s_2 s_1 \pi && 
 s_2 s_1 \pi s_2 s_1 \pi \\
        & [3,4,5] && [2,4,6] && [2,3,7] && [1,5,6] && [1,3,8] && [1,2,9] \\
k = 2 \qquad & \tikzmark{StartB2} (1,1,0) &\prec &(1,0,1) &\prec &(2,0,0) \tikzmark{EndB2} & \;\;  \prec \;\; &
\tikzmark{StartC2} (0, 1, 1) &\prec &(0, 2, 0) &\prec & (0, 0, 2)
\tikzmark{EndC2} \\
\end{alignat*}
\begin{alignat*}{9}
        & \pi^3 && \pi^2 s_1 \pi &&  \pi^2 s_2 s_1 \pi && 
\pi s_1 \pi s_1 \pi && \pi s_1 \pi  s_2 s_1 \pi && 
\pi s_2 s_1 \pi s_2 s_1 \pi \\
        & [4,5,6] && [3,5,7] && [3,4,8] && [2,6,7] && [2,4,9] && [2,3,10] \\
k = 3 \qquad & \tikzmark{StartB3}(1,1,1) &\prec
        &(2, 1, 0) &\prec &(1, 2, 0) &\prec &(2, 0, 1) &\prec
&(1, 0, 2) &\prec &(3, 0, 0) \tikzmark{EndB3}   
\end{alignat*}
\begin{alignat*}{9}
\qquad &\qquad \qquad  && s_1 \pi s_1 \pi s_1 \pi && s_1 \pi s_1 \pi  s_2 s_1 \pi && 
 s_1 \pi s_2 s_1 \pi s_2 s_1 \pi &&  s_2s_1 \pi  s_2s_1 \pi  s_2 s_1 \pi \\
&&&[1,6,8] && [1,5,9] && [1,3,11] && [1,2,12] \\
&&\prec &\tikzmark{StartC3} (0,2,1) &\prec &(0, 1, 2) &\prec &(0, 3, 0) &\prec &(0, 0, 3) \tikzmark{EndC3}
\end{alignat*}
\InsertUnderBrace[draw=blue,text=blue]{StartB1}{EndB1}{$\cP^{\circ}$}
\InsertUnderBrace[draw=blue,text=blue]{StartB2}{EndB2}{$\cP^{\circ}$}
\InsertUnderBrace[draw=blue,text=blue]{StartB3}{EndB3}{$\cP^{\circ}$}
\InsertUnderBrace[draw=blue,text=blue]{StartC1}{EndC1}{$\cP'$}
\InsertUnderBrace[draw=blue,text=blue]{StartC2}{EndC2}{$\cP'$}
\InsertUnderBrace[draw=blue,text=blue]{StartC3}{EndC3}{$\cP'$}
\caption{The partial order on $\ba \in \cP_k(3)$
for degrees $k \le 3$. 
So that one may compare $\prec$ to $\lex$ and to Bruhat order,
above each $\ba$ is the corresponding $\wa{a}$ both in
window notation  and it expression from Lemma \ref{claim:coset}.
} \label{fig:order}
\end{figure}

 For another example, when $n = 4$, $k = 2$ we have $(1,1,0,0) \prec (1, 0, 1, 0) \prec (1, 0, 0, 1) \prec (2, 0, 0, 0) \prec$ $(0, 1, 1, 0) \prec (0, 1, 0, 1)$ $\prec (0, 2, 0, 0) \prec (0, 0, 1, 1)$ $\prec (0, 0, 2, 0) \prec (0, 0, 0, 2)$. 

 The following lemma gives properties of this partial order that will be important for us.

\begin{lemma}\label{lemma:po}
With the partial order defined above, $\cP_{k}(n)$ is linearly ordered. Moreover, the following properties are satisfied.
\begin{enumerate}
\item If $\ba \prec \boldb$ in $\cP_{k}(n)$, then
$\pi\cdot \ba \prec \pi\cdot \boldb$
in $\cP_{k+1}(n)$.
\item If $a_{n} \geq a_{1} > 0$, then $(a_{1}, a_{2}, \dots, a_{n-1}, a_{n}) \prec (a_{n}+1, a_{2}, \dots, a_{n-1}, a_{1}-1)$, 
\item If $a_{i} > a_{i+1}$, then $\ba \prec s_{i}\cdot \ba$.
\item If $a_{i} > a_{i+1}$ and $\boldb \prec \ba$, then $s_{i}\cdot \boldb
\prec s_{i}\cdot \ba$.  \end{enumerate}
\end{lemma}
\begin{proof}
By induction, it follows easily that $\cP_{k}(n)$ is linearly ordered, as it is defined to be the concatenation of two linearly ordered sets. Property (1) is obvious from the definition. It remains to show (2), (3) and (4). Note that when $n = 2$ or $k = 0, 1$, (2), (3) and (4) are easy to check from the explicit definition of the partial order on $\cP_{k}(2)$ or on $\cP_{1}(n)$. So we may use an inductive procedure. We assume that (2), (3) and (4) are valid for $\cP_{k'}(n)$ for every $k'$ and for $\cP_{0}(n+1), \dots, \cP_{k}(n+1)$, and we show that they are valid for $\cP_{k+1}(n+1)$. Recall that  $\cP^{\circ}_{k+1}(n+1) = \pi(\cP_{k}(n+1))$ and $\cP'_{k+1}(n+1) = \cP_{k+1}(n+1) \setminus \cP^{\circ}_{k+1}(n+1)$.

 We start with (2). Note that if $\ba$ is as in (2), then $\ba \in \cP_{k+1}^{\circ}(n+1)$. Then (2) happens if and only if in $\cP_{k}(n+1)$ we have $(a_{2}, \dots, a_{n}, a_{n+1}, a_{1} - 1) \prec (a_{2}, \dots, a_{n}, a_{1} - 1, a_{n+1})$. But this is clear because $\cP_{k}(n+1)$ satisfies (3).

  Now we move on to (3). Thanks to (1) and our inductive assumption, the only problem can arise with $s_{1}$: indeed, for $i > 1$ we can either go to $\cP_{k}(n)$, if $a_{1} > 0$; or to $\cP_{k+1}(n-1)$ if $a_{1} = 0$ and the result follows by induction. So we assume $i = 1$. If $\ba \in \cP'_{k+1}(n+1)$, then we can never have $a_{1} > a_{2}$, so we may assume that $\ba \in \cP^{\circ}_{k+1}(n+1)$.  If $a_{2} = 0$, then $s_{1}\cdot \ba \in \cP'_{k+1}(n+1)$, and we have $\ba \prec s_{1}\cdot \ba$ by definition. Otherwise, we may assume that $a_{1} > a_{2} \geq 1$. Then $s_{1}\cdot \ba \succ \ba$ is equivalent to, in $\cP_{k}(n+1)$, having $(a_{2},  \dots, a_{n+1}, a_{1}-1) \prec (a_{1}, a_{3}, \dots, a_{n+1}, a_{2} -1)$. This is clear because $\cP_{k}(n+1)$ satisfies (2). 

Finally, we check (4). Note that we canot have $i = 1$ and $\ba \in
\cP'_{k+1}(n+1)$ simultaneously. We also cannot have $\ba \in
\cP^{\circ}_{k+1}(n+1)$ and $\boldb \in \cP'_{k+1}(n+1)$
simultaneously. If both $\ba$ and $\boldb$ belong to $\cP'_{k+1}(n+1)$,
the result follows by forgetting the initial $0$ and using induction. If $\ba \in \cP'_{k+1}$ and $\boldb \in
\cP^{\circ}_{k+1}$
then, since $i \neq 1$, $s_{i}$ preserves both $\cP'_{k+1}(n+1)$ and
$\cP^{\circ}(n+1)$ so the result is also clear. The result is also
clear if both $\ba, \boldb \in \cP^{\circ}_{k+1}(n+1)$ and $i \neq 1$.
So it remains to check the case $\ba, \boldb \in \cP^{\circ}$, $i =
1$. If $a_{2} = 0$, $b_{2} \neq 0$, the result is clear. 

If $a_{2}, b_{2} = 0$, then we have that $\ba \succ \boldb$ if and only if $(0, a_{3}, \dots, a_{n+1}, a_{1} - 1) \succ (0, b_{3}, \dots, b_{n+1}, b_{1} - 1)$. This happens if and only if in $\cP_{k}(n)$ we have $(a_{3}, \dots, a_{n+1}, a_{1} - 1) \succ (b_{3}, \dots, b_{n+1}, b_{1} - 1)$. By (1), this implies 
$(a_{1}, a_{3}, \dots, a_{n}) \succ (b_{1}, b_{3}, \dots, b_{n})$ in $\cP_{k+1}(n)$. But then $(0, a_{1}, a_{3}, \dots, a_{n}) \succ (0, b_{1}, b_{3}, \dots, b_{n})$, which is what we wanted to show.

If $a_{2}, b_{2} \neq 0$, we need to show that $(a_{1}, \dots, a_{n+1}, a_{2} - 1) \succ (b_{1}, \dots, b_{n+1}, b_{2} - 1)$. This happens if and only if $(a_{3}, \dots, a_{n+1}, a_{2} - 1, a_{1} - 1) \succ (b_{3}, \dots, b_{2} - 1, b_{1} -1)$. But by assumption,  $(a_{3}, \dots, a_{n+1}, a_{1} - 1, a_{2} - 1) \succ (b_{3}, \dots, b_{1} - 1, b_{2} -1)$ and $a_{1} - 1 > a_{2} - 1$. Since $\cP_{k-1}(n+1)$ satisfies (4), the result follows. 
\end{proof}


We now compare the order $\prec$ to the lexicographic ordering via the bijection $\omega: \Z_{\geq 0}^{n} \rightarrow \Min$. First, we have the following result, compare with Lemma \ref{lemma:po}(1).


\begin{lemma}\label{lemma:pi and lex}
Let $\ba, \boldb \in \mathbb{Z}^{n}_{\geq 0}$ and assume $\wa{a} \lex \wa{b}$. Then, $\wb{\pi \cdot \ba} \lex \wb{\pi \cdot \boldb}$. 
\end{lemma}
\begin{proof}
If the window notation of $\ww$ is 
$[ \ww(1), \ww(2), \cdots \ww(n)]$ then the window notation
of $\pi \ww$ is 
$[ \ww(1)+1, \ww(2)+1, \cdots \ww(n)+1]$.
Hence it is obvious that
 $\ww \lex \ww' \implies \pi \ww \lex \pi \ww'$.
Next, observe $\wb{\pi \cdot \ba} =\pi\wa{a} $ 
from which the lemma follows.  Indeed the entries of
the window notation for $\wa{a}$ sort  those of 
$\ta{a}$, and these have the form $i+n a_i$.
On the other hand, the entries of $\tav{\pi \cdot \ba}$ are
$i + n a_{\pi^{-1}(i)}= i + n a_{i-1}$
which we may reindex as $i+1 + n a_i$ as well as $1+n(a_n+1) = 
n+1 +n a_n$.




\end{proof}

The next result tells us that, even though the partial order $\prec$ looks complicated, it is in fact very natural when transported via the map $\ba \mapsto \wa{a}$.

\begin{lemma}\label{lem:compare to lex}
Let $\ba, \boldb \in \Z_{\geq 0}^{n}$ be such that $||\ba|| = ||\boldb||$. Then, $\ba \prec \boldb$ if and only if $\wa{a} \lex \wa{b}$.
\end{lemma}
\begin{proof}
Since for fixed degree we are dealing with linear orderings,
by Lemma \ref{lemma:w is injective} we only need to check $\ba \prec \boldb$ implies $\wa{a} \lex \wa{b}$. Let us denote $k := ||\ba|| = ||\boldb||$. The case when $\ba \in \mathcal{P}^{\circ}_{k}$ and $\boldb \in \mathcal{P}'_{k}$ follows from Lemma \ref{lemma:w is injective} (b) and (c). The case when $\ba, \boldb \in \mathcal{P}^{\circ}_{k}$ follows from Lemma \ref{lemma:pi and lex} and an inductive argument on $k$. 

Finally, assume $\ba, \boldb \in \mathcal{P}'_{k}$ so that $a_{1} =
b_{1} = 0$ and therefore $\wa{a}(1) = \wa{b}(1) = 1$. We have to
consider $\overline{\ba} = (a_{2}, \dots, a_{n})$, $\overline{\boldb} =
(b_{2}, \dots, b_{n})$. By induction on $n$ we may assume
$\wa{\overline{a}} \lex \wa{\overline{b}}$. Note that we have
$g_{\ba}^{-1}(i+1) = g_{\overline{\ba}}^{-1}(i) +1$ and that
$\overline{a}_{i} = a_{i+1}$ for $i = 1, \dots, n-1$, and $$
\wa{a} = \left[1, \wa{\overline{a}}(1) + 1 + a_{g_{\ba}^{-1}(2)}, \dots, \wa{\overline{a}}(n-1) + 1 + a_{g_{\ba}^{-1}(n)}\right]
$$

\noindent and similarly for $g_{\boldb}$, $\wa{b}$. By assumption, $\wa{\overline{a}} \lex \wa{\overline{b}}$. So there exists $i_{0} \in \{1, \dots, n-1\}$ such that $\wa{\overline{a}}(i) = \wa{\overline{b}}(i)$ for $i < i_{0}$ and $\wa{\overline{a}}(i_{0}) > \wa{\overline{b}}(i_{0})$. If $i < i_{0}$ then by Lemma \ref{lemma:w is very injective} we get $\wa{a}(i+1) = \wa{b}(i+1)$. Now, $\wa{\overline{a}}(i_{0}) > \wa{\overline{b}}(i_{0})$ implies that $(n-1)(a_{g_{\ba}^{-1}(i_{0} + 1)} - b_{g_{\boldb}^{-1}(i_{0} + 1)}) + g_{\overline{\ba}}^{-1}(i_{0}) - g_{\overline{\boldb}}^{-1}(i_{0}) > 0$, from where we deduce that $a_{g_{\ba}^{-1}(i_{0} + 1)} \geq b_{g_{\boldb}^{-1}(i_{0} + 1)}$. Finally,
\begin{gather*}
\wa{a}(i_{0} + 1) = \wa{\overline{a}}(i_{0}) + 1 + a_{g_{\ba}^{-1}(i_{0} + 1)} > \wa{\overline{b}}(i_{0}) + 1 + b_{g_{\boldb}^{-1}(i_{0} + 1)} = \wa{b}(i_{0} + 1)
\end{gather*}

\noindent and we conclude that $\wa{a} \lex \wa{b}$. 
\end{proof}

One can check Figure \ref{fig:order} to see  examples of 
the structure described in  both
Lemmas \ref{lemma:pi and lex} and \ref{lem:compare to lex},
as well as Lemma \ref{lemma:lex vs bruhat} below.

Now let $\ww \in \Min$. Recall that by Lemma \ref{claim:coset}, that we may express $\ww$ in the form
$$
\ww = (s_{\nu_r}\cdots s_2s_1)\pi \cdots (s_{\nu_2}\cdots s_2s_1)\pi(s_{\nu_1}\cdots s_2s_1)\pi
$$

\noindent where $0 \leq \nu_r \leq \cdots \leq \nu_1 < n$. We select $\ell \leq r$ and $j \leq \nu_{\ell}$ and consider the affine permutation
$$
\widehat{\ww} := (s_{\nu_r}\cdots s_2s_1)\pi \cdots (s_{\nu_{\ell}}\cdots s_{j+1}\widehat{s_{j}}s_{j-1}\cdots s_1)\pi\cdots (s_{\nu_1}\cdots s_2s_1)\pi
$$

\noindent where a hat over $s_{j}$ means that we omit that transposition. Clearly, $\widehat{\ww}$ belongs to the monoid $\affSnplus$, and $\deg(\ww) = \deg(\widehat{\ww})$. Let us denote by $\omega' \in \Min$ the permutation whose window notation is the increasing arrangement of the window notation of $\widehat{\ww}$. 

\begin{lemma}\label{lemma:lex vs bruhat}
We have $\ww' \lex \ww$. 
\end{lemma}
\begin{proof}
Let us start with the easy observation that, in window notation:
\begin{equation}\label{eq:cycle+pi}
s_{k}\cdots s_{1}\pi = [1, 2, \dots, k, k+2, \dots, n, n+k+1].
\end{equation}

Now let us split $\ww$ as follows:
$$
\ww = \underbrace{(s_{\nu_r}\cdots s_1)\pi \cdots (s_{\nu_{\ell + 1}}\cdots s_{1})\pi}_{\alpha}\, \underbrace{(s_{\nu_{\ell}} \cdots s_1)\pi}_{\beta} \, \underbrace{(s_{\nu_{\ell - 1}}\cdots s_{1})\pi \cdots (s_{\nu_1}\cdots s_1)\pi}_{\gamma}
$$

\noindent Note that $\alpha, \beta, \gamma \in \Min$. Moreover,
letting $k:=\nu_{\ell}$ since $\nu_{\ell} \leq \nu_{\ell - 1} \leq
\cdots \leq \nu_{1}$ it follows from \eqref{eq:cycle+pi} that $\gamma
= [1, 2, \dots, k, \gamma({k+1}), \dots, \gamma({n})]$, where $k <
\gamma({k+1}) < \cdots < \gamma({n})$. If we write $\alpha =
[\alpha(1), \dots, \alpha(n)]$ we then have that
$$\ww = [\alpha(1), \dots, \stackbin{j}{\alpha(j)}, \dots,
 \alpha(k), \ww({k+1}), \dots, \stackbin[= \alpha(k+1) + zn]{p}{\ww(p),} \dots,  \ww({n})],$$ 
where we let $1 \le p \le n$ be such that $\gamma(p) = zn$. 
Observe $z  > 0$.

Let us now compute $\widehat{\ww}$. Let $\widehat{\beta} :=
s_{\nu_{\ell}}\cdots s_{j+1}\widehat{s_{j}}s_{j-1}\cdots s_{1}\pi$,
so that $\widehat{\ww} = \alpha\widehat{\beta}\gamma$.  A
straightforward computation shows that, in window notation,
$\widehat{\beta} = [1, 2, \dots, j-1, k+1, j+1, \dots, k, k+2, \dots,
n, n+j]$. Then, $\widehat{\ww} = $
\begin{gather*}
[\alpha(1) \dots \alpha({j\!-\!1}),
\stackbin[=\ww(p)-zn]{j}{\alpha({k\!+\!1})},
\alpha({j\!+\!1}) \dots, \alpha({k}), {\ww}({k+1}), \dots,
\stackbin[=\ww(j)+zn]{p}{\alpha(j) + zn} \dots
{\ww}({n})],
\end{gather*}  
and in particular the window notation for $\ww$ and $\widehat{\ww}$
agree except for in the $j$th and $p$th entries. 
Since $\ww$ is already sorted so its entries increase, to
show $\ww' \lex \ww$ it suffices to show
$\widehat{\ww}(j) > \ww(j)$
and 
$\widehat{\ww}(p) > \ww(j)$. This will ensure that
the first $j-1$ entries of $\ww$ and the sorted
$\ww'$ agree, but the $j$th
entry of $\ww'$ will be strictly larger than $\ww(j)$.
We compute
$\widehat{\ww}(j) = \alpha(k+1) > \alpha(j) = \ww(j)$ since $\alpha
\in \Min$, and 
$\widehat{\ww}(p) =  \alpha(j) + zn > \alpha(j) = \ww(j)$.

\MVcomment{ 
\noindent We claim that $\alpha({j-1}) < \widehat{\ww}({i})$ for $i =
j, \dots, n$. This is clear for $i = j, \dots, k$.
For $i > k$, let $\gamma({i}) = a_{i} + nb_{i}$,
\MV{can i rewrite this? in particular i might write
$g_i + n a_i$ or change it to be more like the 
argument in lex-bruhat-proof.pdf . ??}
 with $1 \leq a_{i} \leq n$. Note that
we must have $a_{i} \in \{k+1, \dots, n\}$. If $a_{i} = n$ then we
have $\widehat{\ww}({i}) = \alpha(j) + n > \alpha(j-1)$. Else, we
have $\widehat{\ww}({i}) = \ww({a_{i}}) > \ww({k}) > \alpha({j-1})$.
It follows that in window notation $$
\ww' = [\alpha({1}), \dots, \alpha({j-1}), \ww'({j}), \ww'({j+1}), \dots, \ww'({n})]
$$

\noindent with $\ww'({j}) \geq \alpha({k+1}) > \alpha({k}) = \ww({k}) \geq \ww({j})$.   Since $\ww'({i}) = \alpha({i}) = \ww({i})$ for $i < j$, the result follows.
} 

\end{proof}


\begin{remark}

From Lemma \ref{lemma:lex vs bruhat} we see that for
$\ww,\p \in \Min$ of the same degree
that
\begin{gather}
\label{eq: bruhat-lex}
\p \bruhat \ww  \text{ implies } \p \lex \ww
\end{gather}
in window notation.
We could also deduce this from the characterization of Bruhat
order for $\affSnCox$
given in Bj\"orner-Brenti \cite[Theorem 8.3.7]{BB}
(which one must extend appropriately to $\affSn$; this is easy if
only comparing permutations of the same degree).
In particular, they characterize
$\p \lebruhat \ww$ for $\p, \ww \in \affSnCox$
if $\p[i,j] \le \ww[i,j]$ for all $i, j \in \Z$
where
$$ \p[i,j] = \# \{a \le i \mid \p(a) \ge j \}.$$
We can prove \eqref{eq: bruhat-lex} as follows.

Let 
 $\p \neq \ww \in \Min$  be of the same degree, which means 
$\sum_{i=1}^n \p(i) = \sum_{i=1}^n \ww(i) $. 
(This sum is $n(n+1)/2$ in the case $\p, \ww \in \affSnCox$.)
Suppose
that $\p \not\lex \ww$
which means $\p \llex \ww$. Hence for some
$1 \le \ell \le n$ we have $\p(1)=\ww(1), \p(2) = \ww(2),
\cdots, \p(\ell-1) = \ww(\ell-1)$ but $\p(\ell) < \ww(\ell)$.
Since $\sum_{i=1}^n \p(i) = \sum_{i=1}^n \ww(i) $,
there must be some $\ell \le i \le n$ such that $\p(i) > \ww(i)$.
Let $i \le n$ be the largest such $i$.
In other words for $n \ge r > i$ we have $\p(r) \le \ww(r)$.
Let $j:=\p(i)$. Let us compare $\p[i,j]$ and $\ww[i,j]$.

Since $\ww \in \Min$ we have $\ww(1) < \ww(2) < \cdots < \ww(n)$
and so given $a$ such that $\ww(a) \ge j = \p(i) > \ww(i)$
and $a \le i$ then
$i+1 < a +kn < n$ for some $k> 0$.
In particular $\p(a) = \p(a+kn) - kn \ge \ww(a+kn)-kn = \ww(a) \ge j$.
Hence $\{a \le i \mid \ww(a) \ge j \} \subseteq
\{a \le i \mid \p(a) \ge j \}$. Further as $\ww(i) < \p(i) =j$
the element $i$ does not belong to the first set but does to the
second. This shows $\ww[i,j] < \p[i,j]$ and so $\p \not\bruhat \ww$.
This proves \eqref{eq: bruhat-lex}.


\end{remark}



We remark that even though $\lex$ is a total order on $\affSn$, 
we only relate it to Bruhat order for two affine permutations
of the same degree and  that are both in $\Min$.

\section{Rational Cherednik algebras}
Throughout the rest of the paper
we take $m, n \in \Z_{>0}$ and $\gcd(m,n)=1$ unless otherwise stated.

\subsection{Definition and basic properties}
\label{sec: RCA background}

We work with the rational Cherednik algebra $H_{t, c}:= H_{t,c}(\Sn,\C^n)$ of $\Sn$ acting on $\C^n$ by permuting the coordinates. Let us recall that this is the quotient of the semidirect product algebra $\C\langle x_{1}, \dots, x_{n}, y_{1}, \dots, y_{n}\rangle \rtimes \Sn$ by the relations
\begin{align*}
[x_{i}, x_{j}] = 0, \qquad & [y_{i}, y_{j}] = 0, \\
[y_{i}, x_{j}] = c(ij), \qquad  & [y_{i}, x_{i}] = t - c\sum_{j \neq i}(ij)
\end{align*}
Here $t$ and $c$ are complex parameters. Clearly, for a nonzero complex number $a \in \C^{*}$, $H_{at, ac} \cong H_{t,c}$. So we have the dichotomy $t = 0$ or $t = 1$. For most of the paper, we will assume that $t = 1$ and write $H_{c} := H_{1,c}$.

We recall some basic facts about $H_{t,c}$ following \cite{BEG}. The algebra $H_{t,c}$ is graded, with $x_i$ of degree $1$, $y_i$ of degree $(-1)$, and $\Sn$ in degree zero. When $t = 1$, the grading on $H_{c}$ is internal
 and defined by the Euler element
$$
h=\frac{1}{2}\sum_{i}(x_iy_i+y_ix_i)=\sum_{i}x_iy_i+ \frac{n}{2} -c\sum_{i<j}(ij).
$$

\noindent Let us emphasize that the grading on $H_{0, c}$ is not internal. The algebra $H_{t,c}$ is also filtered, with $x_i$ and $y_i$ of filtration level 1 and $\Sn$ of filtration level 0. An important {\em PBW theorem} states that
$$
\gr H_{t,c}=\C[x_1,\ldots,x_n,y_1,\ldots,y_n]\rtimes \Sn,
$$
where $\gr$ denotes associated graded with respect to this filtration.
This implies that a basis of $H_{t,c}$ as a $\C$-vector space is given by
${x}^{\ba}\ww {y}^{\boldb}$, where $\ww \in \Sn$,
${x}^{\ba} := x_{1}^{a_1}\cdots x_{n}^{a_n}$,
$\ba \in \Z_{\geq 0}^{n}$ and similarly for ${y}^{\boldb}$,
$\boldb \in \Z_{\geq 0}^{n}$. We will refer to this basis as the
\emph{PBW basis} of $H_{t,c}$. Another easy consequence of the PBW
theorem is that $H_{t,c}$ contains the following subalgebras:
$$
\Hnx := \C[x_1, \dots, x_n] \rtimes \Sn, \qquad \Hny := \C[y_1, \dots, y_n]\rtimes \Sn.
$$

Next we need to consider some modules for $H_{t,c}$. We have the {\em standard modules} 
$$
\Delta_{t,c}(\mu)=\Ind_{\Hny}^{H_{t,c}}V_{\mu}\simeq V_{\mu}\otimes_{\C}\C[x_1,\ldots,x_n],
$$
where $V_{\mu}$ is an irreducible representation of $\Sn$ corresponding to the  Young diagram $\mu$ of size $n$, and $y_i$ annihilate $V_{\mu}$.
In particular, for
$\mu=(n)$ the representation $V_{\mu}$ is trivial, and we get the {\em
polynomial representation} $\Delta_{t,c}(\triv)\simeq
\C[x_1,\ldots,x_n]$.

When $t = 1$, it is not hard to see using the Euler element $h$ that $\Delta_{c}(\mu) := \Delta_{1, c}(\mu)$ has a unique irreducible quotient that we denote by $L_{c}(\mu)$. In fact, these are the simple objects of the category $\cO_{c}$, which is defined as the category of $H_{c}$-modules which are finitely generated over $\C[x_1,\ldots,x_n]$ and where $y_i$ act locally nilpotently. For example, the standard modules $\Delta_{c}(\mu)$ belong to $\cO_{c}$. We have the following facts about the category $\cO_{c}$.



\begin{theorem}[\cite{BEG}] a) If $c\in \C\setminus \Q$ then the category $\cO_{c}$ is semisimple, and all standard modules
$\Delta_c(\mu)$ are irreducible. The same is true if $c$ is rational but its denominator is greater than $n$.

b) Suppose that $c=m/n$ where $m,n \in \Z_{>0}, \gcd(m,n)=1$. Then $\Delta_c(\mu)$ is irreducible, unless $\mu$ is a hook.

c) Suppose that $c=m/n$ where $m,n \in \Z_{>0}, \gcd(m,n)=1$,
and  let $\mu_{\ell}=(n-\ell,1^{\ell})$ be a hook partition. Then the morphisms between standard modules have the following form:
$$
\Hom_{H_c}(\Delta_c(\mu),\Delta_c(\mu'))=\begin{cases}
\C &   \mu = \mu', \\ \C &   \mu=\mu_{\ell},\mu'=\mu_{\ell-1}\ \text{for some}\ \ell,\\
0 & \text{otherwise}.
\end{cases}
$$
\end{theorem}

This theorem is proved in \cite{BEG} using the {\em Knizhnik-Zamolodchikov functor} which compares the representation theory of $H_{c}$ to that of the type $A$ finite Hecke algebra. In this paper we give an alternative combinatorial proof.  The ``otherwise" case of (b) is proved in Lemma \ref{lem:not hooks}, and 
Lemma \ref{lem:far hooks}, while the interesting morphism $\Delta_c(\mu_{\ell})\to \Delta_c(\mu_{\ell-1})$ is constructed in Proposition \ref{prop:bgg}. Part (b) easily follows from (c), see Corollary \ref{cor:not hook}.


The representation theory of the algebra $H_{0, 1}$ is very different from that of $H_{c}$. For example, it is no longer true that the standard module $\Delta_{0, 1}(\mu)$ has a unique irreducible quotient, moreover, the algebra $H_{0,1}$ is finite over its center so every irreducible $H_{0,1}$-module is finite-dimensional, \cite{EG}. Still, in Section \ref{sect:m to infinity} we will use our results in the $t = 1$ case and a limiting procedure to give a combinatorial basis of $\Delta_{0, 1}(\mu)$. 

We will also consider the spherical subalgebra of $H_{t, c}$. Let $e := \frac{1}{n!}\sum_{p \in \Sk{n}}p \in \C\Sk{n} \subseteq H_{t,c}$ be the trivial idempotent for $\Sk{n}$. The spherical rational Cherednik
algebra is the corner algebra $eH_{t,c}e$. We have an obvious functor $H_{t,c}\text{-mod} \to eH_{t,c}e\text{-mod}$, given by $M \mapsto eM = M^{\Sk{n}}$. When $t = 0$ or  $t=1$ and $c$ is not a negative real number, this functor is known to be an equivalence.



\subsection{An alternative presentation of \texorpdfstring{$H_{t,c}$}{Hc}}\label{sec: new presentation}

 It will be convenient to work with a trigonometric presentation of the algebra $H_{t, c}$ that has already appeared in the work of Griffeth and Webster in the cyclotomic setting, \cite{g2, W}. Since some relations become more transparent in the type $A$ setting, we recall this presentation in detail.
First, we have the Dunkl-Opdam elements in $H_{t, c}$:
$$
u_{i} := x_{i}y_{i} - c\sum_{j < i}(ij).
$$
\begin{lemma}
The Dunkl-Opdam elements generate a polynomial subalgebra of $H_{t, c}$.
\end{lemma}
\begin{proof}
It is straightforward to see that $u_iu_j=u_ju_i$. Since the leading term of $u_i$ is $x_iy_i$, 
the leading terms of $u_i$ in $\gr H_{t,c}$ are algebraically independent, and hence $u_i$ are algebraically independent in $H_{t,c}$.
\end{proof}
We will denote this polynomial subalgebra by $\A$.

We remark that we have the following relations where, as usual, $s_{i} = (i, i+1) \in \Sn$ is a simple transposition.
\begin{align}
\label{si ui} s_{i}u_{i}  & = u_{i+1}s_{i}  + c \\
s_{j}u_{i} & = u_{i}s_{j} \; \text{if} \; j \neq i, i -1
\end{align}
\begin{remark}
These equations imply that $u_i$ and $s_j$ form a subalgebra in $H_{t,c}$ isomorphic to the degenerate affine Hecke algebra. We will denote this algebra by $\Hnu$.
\end{remark}

We will also need the following shift 
operators. Let $\tau :=
x_{1}(12\cdots n)$, $\lambda := (12\cdots n)^{-1}y_{1}$. Note that,
for every $i$, we have that $\tau = (1\cdots i)x_{i}(i \cdots n)$ and  $\lambda
= (n\cdots i)y_{i}(i\cdots n)$. 
 The following relations are
straightforward to check. 
\begin{align}
\label{tau ui} \tau u_{i} & = u_{i+1}\tau, i \neq n \\
\label{u n+1} \tau u_{n} & = (u_{1} - t)\tau \\
\lambda u_{i} & = u_{i-1}\lambda, i \neq 1  \\
\label{u0} \lambda u_{1} &  = (u_{n} + t)\lambda \\
\label{si tau 1} s_{i}\tau & = \tau s_{i-1}, i \neq 1 \\
\label{si tau 2} s_{1}\tau^{2} & = \tau^{2} s_{n-1} \\
\label{si lambda 1} s_{i}\lambda & = \lambda s_{i+1}, i \neq n-1 \\
\label{si lambda 2} s_{n-1}\lambda^{2} & = \lambda^{2}s_{1} \\
\label{tau lambda} \tau\lambda & = u_{1} \\
\label{lambda tau} \lambda\tau & = u_{n} + t \\
\label{s1 tau lambda} \lambda s_{1}\tau & = \tau s_{n-1}\lambda + c
\end{align}
It is clear that the elements $s_{1}, \dots, s_{n-1}, \tau$ and $\lambda$ generate the algebra $H_{t,c}$. It turns out that, together with the $u_{i}$, they give a presentation of this algebra.

\begin{remark}
\label{remark-ui for i in Z}
Given relations \eqref{u0} and \eqref{u n+1}, it is convenient to
define $u_i$ for $i \in \Z$ by setting $u_{i+nk} = u_i -kt$
for $1 \le i \le n$.
\end{remark}

The following theorem appears in work of Griffeth \cite[Theorem 3.1]{g2} and Webster \cite[Theorem 2.3]{W} in the more complicated cyclotomic setting. 

\begin{theorem}\label{thm:new presentation}
Let ${\sf H}_{t,c}$ be the algebra generated by elements $u_{1}, \dots, u_{n}, \tau, \lambda$ and the symmetric group $\Sk{n}$, subject to the relations that $[u_{i}, u_{j}] = 0$ and \eqref{si ui}--\eqref{s1 tau lambda}. Then, ${\sf H}_{t,c} \cong H_{t,c}$. 
\end{theorem}

We refer to the aforementioned \cite{g2, W} for a proof of this result. Here, we just remark that to recover $x_{i}$ and $y_{i}$ from $\tau, \lambda$ and $\Sk{n}$ we set

$$
x_i = s_{i-1}\cdots s_{1}\tau s_{n-1}\cdots s_{i}, \quad y_{i} = s_{i}\cdots s_{n-1}\lambda s_{1}\cdots s_{i-1}.
$$

Note that we can also eliminate the Dunkl-Opdam elements $u_i$ from this presentation:
\begin{proposition}
\label{prob: eliminate u}
The algebra $H_{t,c}$ is generated by $s_i$, $\lambda$ and $\tau$ subject to the  equations \eqref{si tau 1}, \eqref{si tau 2},         \eqref{si lambda 1}, \eqref{si lambda 2}, \eqref{s1 tau lambda} and one more equation
\begin{equation}
 \lambda\tau   = t+ s_1\cdots s_{n-1}\tau\lambda s_{n-1}\cdots s_1  -c \sum_{i=1}^{n-1}s_1\cdots s_i\cdots s_1 
\end{equation}
\end{proposition}


The following lemma relates the nonnegative part of $H_{t,c}$ to affine permutations.

\begin{lemma}
\label{lem:monoids x}
Let $\cX$ denote the monoid of monomials in $s_i$ and $\tau$ (or, equivalently, in $s_i$ and $x_j$).
Then there is an isomorphism of monoids $F_{\cX}:\cX\to \affSnplus$ such that 
$$
F_{\cX}(s_i)=s_i,\ F_{\cX}(\tau)=\pi,\ F_{\cX}(x_1^{a_1}\cdots x_n^{a_n})=\ta{a}
$$
for $a_i \ge 0$.
\end{lemma}

\begin{proof}
We can define $F_{\cX}$ by $F_{\cX}(s_i)=s_i,\ F_{\cX}(\tau)=\pi$. Since the relations $s_i\pi=\pi s_{i-1}$
and $s_1\pi^2=\pi^2s_{n-1}$ hold in $\affSn$, $F_{\cX}$ is a homomorphism. 
Considering the window notation for $s_i$ and $\pi$, it is easy
to see the image is $\affSnplus$.

Now 
$$
F_{\cX}(x_i)=s_{i-1}\cdots s_{1}\pi s_{n-1}\cdots s_{i}=[1,\ldots,i-1,i+n,i+1,\ldots,n]=\tav{(0,\ldots,0,1,0,\ldots,0)}
$$
so $F_{\cX}(x_1^{a_1},\cdots,x_n^{a_n})=\ta{a}$ for all $\ba
\in \Z_{\ge 0}^n$.

Finally, by Lemma \ref{lem:coset positive} any element of $\affSnplus$
can be uniquely written as $\ww=\ta{a}g$ for $g\in \Sn$,
$\ba \in \Z_{\ge 0}^n$, while any
element of $\cX$ can be uniquely written as $x_1^{a_1}\cdots
x_n^{a_n}g$ for $g\in \Sn$.  Therefore $F_{\cX}$ is a bijection.
\end{proof}

\begin{corollary}
The monoid $\affSnplus$ is generated by $s_i,\pi$ modulo relations in $\Sn$ and
$$
s_i\pi=\pi s_{i-1},\  s_1\pi^2=\pi^2s_{n-1}.
$$
\end{corollary}

Similarly, we have the following.

\begin{lemma}
\label{lem:monoids y}
Let $\cY$ denote the monoid of monomials in $s_i$ and $\lambda$ (or, equivalently, in $s_i$ and $y_j$).
Let $\affSnminus$ be the monoid generated by inverses of elements in $\affSnplus$.
Then there is an isomorphism of monoids $F_{\cY}:\cY\to \affSnminus$ such that 
$$
F_{\cY}(s_i)=s_i,\ F_{\cY}(\lambda)=\pi^{-1},\ F_{\cY}(y_1^{a_1}\cdots y_n^{a_n})=\ta{-a}.
$$
\end{lemma}
\begin{remark} \label{remark:FX vs FY}
The two isomorphisms $\FX$ and $\FY$  are not compatible 
in the group $\affSn$ in the sense that 
relations between elements in the two monoids may not 
hold for their preimages in $\RCA$.
For instance
$\pi \pi^{-1} = \id = \pi^{-1} \pi$ in $\affSn$
whereas  $\tau \lambda \neq \lambda \tau$.
See equations \eqref{tau lambda} and \eqref{lambda tau}.
\end{remark}

\section{A Mackey formula for \texorpdfstring{$H_{t,c}$}{Ht,c}}
\label{sec-mackey}

\subsection{Basis in \texorpdfstring{$H_{t,c}$}{Htc}}

In this section we present a basis in the algebra $H_{t,c}$ using the generators from Section \ref{sec: new presentation}. It is an analogue of the PBW basis
from Section \ref{sec: RCA background}. Recall that $\Hny$ is the subalgebra generated by $\Sn$ and $y_i$
(or, equivalently, by $\Sn$ and $\lambda$), and that $\Hnu$ denotes the subalgebra generated by $\Sn$ and $u_i$.

\begin{lemma}
\label{lemma-algebra-basis}
(a) The algebra $\Hny$ has a basis 
$$
g\lambda (s_1s_2\cdots s_{\mu_1}) \cdot \lambda (s_1s_2\cdots s_{\mu_2}) \cdots
 \lambda (s_1s_2\cdots s_{\mu_{r'}}) 
$$
for $g\in \Sn$ and $0\le \mu_{r'}\le \ldots \le \mu_1$.

(b) The algebra $H_{t,c}$ has a basis 
$$
(s_{\nu_r} \cdots s_2 s_1) \tau \cdots
 (s_{\nu_1} \cdots s_2 s_1) \tau \cdot g \cdot \lambda (s_1s_2\cdots s_{\mu_1}) \cdot \lambda (s_1s_2\cdots s_{\mu_2}) \cdots
 \lambda (s_1s_2\cdots s_{\mu_{r'}}),
$$
for $g\in \Sn$ and $0\le \mu_{r'}\le \ldots \le\mu_1, 0\le \nu_{r}\le\cdots \le \nu_1$.


(c) The algebra $H_{t,c}$ is free as a right $\Hny$-module
 with the basis
$$
(s_{\nu_r} \cdots s_2 s_1) \tau \cdots
 (s_{\nu_1} \cdots s_2 s_1) \tau.
$$ 
\end{lemma}

\begin{proof}
By Lemma \ref{lem:monoids x} any monomial in $x_i$
of degree $r$
 can be written as $F_{\cX}^{-1}(\ww)$ for  $\ww\in \affSnplus$
also of degree $r$.
By Lemma \ref{claim:coset} we can write 
$$
\ww=(s_{\nu_r} \cdots s_2 s_1) \pi \cdots (s_{\nu_1} \cdots s_2 s_1) \pi \cdot g,\ 0\le \nu_r\le\cdots \le \nu_1.
$$
so that
$$
F_{\cX}^{-1}(\ww)=(s_{\nu_r} \cdots s_2 s_1) \tau \cdots (s_{\nu_1} \cdots s_2 s_1) \tau \cdot g.
$$
Similarly we can write
 $$
F_{\cY}^{-1}(\ww^{-1})=g^{-1}\lambda (s_1s_2\cdots s_{\nu_1}) \cdot \lambda (s_1s_2\cdots s_{\nu_2}) \cdots
 \lambda (s_1s_2\cdots s_{\nu_r}).
$$
The algebra $\Htc$ has PBW basis $x_1^{a_1}\cdots x_n^{a_n}g'y_1^{b_1}\cdots y_n^{b_n}$,
and we can rewrite the monomials  $x_1^{a_1}\cdots x_n^{a_n}$ and $y_1^{b_1}\cdots y_n^{b_n}$ as above independently.
Finally, (c) is obvious from (a) and (b).
\end{proof}

\begin{remark}
We can  write this basis in more compact form  $F_{\cX}^{-1}(\ww_1)gF_{\cY}^{-1}(\ww_2^{-1}),$
where $\ww_1$ and $\ww_2$ are minimal length coset representatives in
$\Min$. 
\end{remark}

\begin{corollary}
\label{cor: basis in induced module}
Let $V$ be a finite dimensional representation of $\Sn$ with basis
$v_T$, $T \in \cT$. We can regard it as a representation 
of $\Hny$ where $y_i$ 
act by 0. Then the induced representation $\Delta_{t,c}(V) := \Ind_{\Hny}^{H_{t,c}}(V)$
has the basis
$$
v_T(\ww):= F_{\cX}^{-1}(\ww) v_T,
$$
where $\ww$ is a minimal length coset representative in
$\Min$ and $T\in \cT$. 
\end{corollary} 

We define a partial order on the basis elements of $\Delta_{t,c}(V)$ in Corollary \ref{cor: basis in induced module} as follows: $v_{T}(\ww) < v_{T'}(\ww')$ if $\deg\ww = \deg \ww'$ and $\ww \lex \ww'$.



\begin{example}
If $V$ is the trivial representation of $\Sn$ then $\Delta_{t,c}(V)$ is just the polynomial representation. By Lemma \ref{lem:monoids x} the basis $v_T(\ww)$ matches the monomial basis in $\C[x_1,\ldots,x_n]$ and by Lemma \ref{lem:compare to lex} the partial order we have defined coincides with the partial order $\prec$ defined in Section \ref{sect:combinatorics}. 
\end{example}

Next, we want to understand the action of the degenerate affine Hecke algeba $\Hnu$ in this basis.
Via the homomorphism $\ev: \Hnu \to \Sn$ 
that sends $u_1 \mapsto 0$ and $s_i \mapsto s_i$,
the $u_i$ act on $V$ as Jucys-Murphy elements,  and they
 can be simultaneously diagonalized. 

\begin{lemma}\label{lemma:dAHA}
Suppose that $v_{T} \in V$ has weight $\wt$, i.e., it
 is a common eigenvector for the $u_i$  with eigenvalues $\wt_i$. 
Then 
\begin{align*}
u_i(v_{T}(\ww))&=(\ww\cdot\wt)_{i} v_{T}(\ww)+ \mathrm{\ell.o.t}\\
	&= \wt_{\ww^{-1}(i)} v_{T}(\ww) + \mathrm{\ell.o.t}
\end{align*}
where $\ww$ is a a minimal length coset representative in
$\Min$, 
$\ww\cdot\wt$ is defined using the action \eqref{AffSymmOnCn} and
$\mathrm{\ell.o.t}$ denotes lower order terms.  
\end{lemma}

\begin{proof}
As in Remark \ref{remark-ui for i in Z},
to simplify notation, we define $u_i$ for all $i\in \Z$ by
$u_{i+n}=u_i-t$, and likewise for $\wt_{i+n}$.
 Now the relations between $u_i, s_j$ and $\tau$ get
the following form: $$
s_{j}u_{i}=u_{s_j(i)}s_j +
\begin{cases}
    c &  i\equiv j\bmod n, \\
    -c  &  i\equiv j+1\bmod n, \\
    0 & \text{otherwise},
    \end{cases}
$$
and
$$
\tau u_i= u_{\pi(i)}\tau.
$$
Overall, we can write 
$$
F_{\cX}^{-1}(\ww) u_i=u_{\ww(i)} F_{\cX}^{-1}(\ww) +\ldots
$$
so
$$
u_iF_{\cX}^{-1}(\ww) v_{T}=F_{\cX}^{-1}(\ww) u_{\ww^{-1}(i)}v_{T}+\ldots=
(\ww \cdot \wt)_{i}F_{\cX}^{-1}(\ww)v_{T}+\ldots,
$$
and by Lemma \ref{lemma:lex vs bruhat} all extra terms are less than  $F_{\cX}^{-1}(\ww) v_{T}$ in our order.  


\end{proof}

\begin{corollary}
\label{cor: jm eigenvalues}
The generalized eigenvalues of $u_i$ on $\Delta_{t,c}(V)$ are expressed as  $\wt=\ww \cdot \kappa_{T}$
where $\kappa_{T}$ are eigenvalues of Jucys-Murphy operators for the basis $v_{T}$ and $\ww \in \Min$.
\end{corollary}

\begin{remark}
Note that Lemma \ref{lemma:dAHA} relates the standard modules to the \emph{integrable} modules of the trigonometric (aka degenerate) double affine Hecke algebra (\cite{CherednikBook}) via Suzuki's localization, cf. \cite{Suzuki, Suzuki2, Vasserot}
\end{remark}

\subsection{Decomposition into \texorpdfstring{$\Hnu$}{Hn(u)}-modules}

In this section we give a more precise presentation of induced modules.  First, we recall a useful construction of 
$\Hnu$-modules which are induced from parabolic subgroups.

Let
\begin{align*}
\Pinc &= \{ \ba \in \Z_{\ge 0}^n \mid  a_1 \le a_2\le \cdots \le a_n \}
\\
\Pdec &= \{ \bd \in \Z_{\ge 0}^n \mid  d_1 \ge d_2\ge \cdots \ge d_n \}.
\end{align*}
Let $\ba \in \Z^n$ and let $\Sa$ be its stabilizer in $\Sn$.
In the special case $\bd \in \Pdec $, the stabilizer
$\Sd$ is a standard parabolic subgroup.
If $(k_1,\ldots,k_s)$ is the  composition of $n$ that
gives the multiplicities of the entries of $\bd$, then
$\Sd = \langle s_i \mid d_i = d_{i+1}\rangle \simeq 
 \Sk{k_1}\times \cdots \times \Sk{k_n}$.
Note that this subgroup is conjugate to any such parabolic
with the $k_i$ reordered.
Recall $\ww_0= [n,\ldots, 2,1]$ is the longest element of $\Sn$.
Let
$$ \revd = \ww_0 \cdot \bd = (d_n, \ldots, d_2, d_1)$$
so in particular $\Sd$ and $\Srevd$ are conjugate;
and $\bd \in \Pdec \iff \revd \in \Pinc$.
Let $\wdlong$ be the longest element of $\Sd$.
Conjugation by $\ww_0 \wdlong $ induces an isomorphism 
$ \Sd \to  \Srevd$ we will denote $\revdtwist$.
Observe $\ww_0 \wdlong = g_{\bd}$ as it sorts $\bd$ to $\revd
= \sort(\bd)$.
In fact, we would have produced the same isomorphism
conjugating by $\wa{\bd}^{-1}$, where we recall $\wa{\bd} \in \Min
\subseteq \affSn$. 
Similarly, we have a corresponding parabolic subalgebra of $\Hnu$
we will denote 
$$\Hd := \langle s_i, u_j \mid d_i=d_{i+1}, 1\le j
\le n \rangle = \C[u_1,\ldots,u_n]\rtimes \Sd.$$
Just as we have an algebra automorphism $\shift:\Hnu \to \Hnu$ that
sends $s_i \mapsto s_i$ but does a constant shift 
$u_i \mapsto u_i + t$, $\Hd$ has a ``finer" automorphism that 
is the identity on $\Sd$ and
shifts $u_i \mapsto u_i + d_i t$. Using this $\shift$ map,
we can extend 
\begin{align*}
\revdtwist: \Hd  &\to \Hrevd \\
	s_i &\mapsto s_{ g_{\bd}(i)} = g_{\bd} s_i  g_{\bd}^{-1} \\
	u_j &\mapsto u_{g_{\bd}(j)} + d_{j} t.
\end{align*}
Given an $\Hrevd$-module $M$, via the above algebra isomorphism
we can turn it into the ``twisted" 
$\Hd$-module we denote $M^{\revdtwist}$.
Note that $\revdtwist(s_i) = \wa{d}^{-1} s_i \wa{d}$
and $\revdtwist(u_j) = u_{\wa{d}^{-1}(j)}$, when we extend
the notion of the $u_j = u_{j+kn} + kt$ to be indexed by $j \in \Z$
 as in Remark \ref{remark-ui for i in Z}.
However, the map $\revdtwist$ does not agree with 
conjugation by $\wa{d}^{-1}$, 
but under some lens it does up to ``lower
order terms" in a sense that will be made more precise 
in Remark \ref{rem-classical Mackey} below. 
\MVcomment{
staying blue while still a draft; still working on it: 
planning to put it in the remark.... (right now in monicastuff.tex)
Conjugation by $\wa{\bd}^{-1}$ or by $\wa{d} \in \Min$
 does not make sense in $\Hnu$ 
as $\pi \not\in \Hnu$, nor does conjugation by $\FXi(\wa{d})$
as $\tau$ is not invertible. However, $\Sn$ acts on $\Hnu$
by inner automorphisms, and
...  We let $\pi$ act by the outer
automorphism  on $\A$ ... 
} 
This is consistent with the observation
that in $\Hnu$ we have
$ u_j \FXi(\ww) = \FXi(\ww) u_{\ww^{-1}(j)} + \mathrm{\ell.o.t}$,
where the latter are $\FXi$ applied to
 terms lower than $\ww \in \affSn$ in Bruhat order.

\begin{example}
\label{example-twist}
Let $n=5$, $\bd = (2,2,0,0,0)$, so $\revd = (0,0,0,2,2)$.
Note $\wa{\bd} = [3,4,5,11,12]$,
$\wa{\bd}^{-1} = [-6,-5,1,2,3]$,
$\Sd \simeq  \Sk{2} \times \Sk{3}$ and 
$\Srevd \simeq  \Sk{3} \times \Sk{2}$.

The permutations $\ww_0 = [5,4,3,2,1], \,
\wdlong = [2,1,5,4,3], \,
\ww_0 \wdlong = [4,5,1,2,3] = g_{\bd}$. 
The restricted module
$\Res_{\Hrevd} \triv = M \boxtimes N$ where $M, N$ are
one-dimensional spanned by weight vectors with $u$-weight
$(0,-c,-2c)$ and $(-3c,-4c)$ respectively. 
The twisted $\Hd$-module  $[\Res_{\Hrevd} \triv]^{\revdtwist}$
is one-dimensional, now spanned by weight vectors with $u$-weight
$(-3c +2t,-4c+2t)$ and $(0,-c,-2c)$ repectively. 
It has the form $N^{{\shift \times 2}} \boxtimes M$.
The map $\revdtwist$ sends
\begin{align*}
s_1 &\mapsto s_4 &	u_1 &\mapsto u_{-6} = u_4 + 2t \\
s_3 &\mapsto s_1&	u_2 &\mapsto u_{-5} = u_5 + 2t \\
s_4 &\mapsto s_2&	u_3 &\mapsto u_{1}  \\
		&&u_4 &\mapsto u_{2} \\
		&&u_5 &\mapsto u_{3} 
\end{align*}
\end{example}

Just as the minimal length left coset representatives
 $\{\wa{a} \mid \ba \in \Pn\} = \Min \subseteq \affSnplus$ are those affine permutations
whose window
notation have positive increasing entries, 
the minimal length {\em double} coset representatives
with respect to $\Sn$
are those $\wa{a}$ whose inverses' window notation have
increasing entries. These are exactly the $\wa{d}$ for $\bd \in
\Pdec$.  See Example \ref{example-twist}. 

\begin{example}
\label{example-double coset}
Let $n=3$, $\bd = (4,1,1)$, so $\wa{d} = [5,6,13]$ 
is a minimal length double coset representative as
$\bd \in \Pdec$.
Note the double coset decomposes into left cosets as
\begin{alignat*}{4}
\Sk{3} [5,6,13] \Sk{3} &= &[5,6,13] \Sk{3} &\; \sqcup \;&
 [4,6,14] \Sk{3} & \;\sqcup \;& [4,5,15] \Sk{3} \\
\text{i.e., } \Sk{3} \wa{d} \Sk{3} &=  &\wa{d} \Sk{3} & \;\sqcup \;&  \wb{s_1 \bd} \Sk{3}
        & \;\sqcup \;& \wb{s_2 s_1 \bd} \Sk{3} \\
&= &\wb{(4,1,1)} \Sk{3} & \;\sqcup \;&  \wb{(1,4,1)} \Sk{3} & \;\sqcup \;&
        \wb{(1,1,4)} \Sk{3} .
\end{alignat*}
$\Sk{3}/\Sd$ has minimal length left coset representatives
$\{\id, s_1, s_2 s_1\}$.

\end{example}

It is well known that $\C \Sn$ is a free right module over $\C \Sd$ of rank $\frac{n!}{k_1!\cdots k_s!}$.
Given a representation $M$ of $\Sd$, we can consider the induced
representation $\Ind_{\Sd}^{\Sn}M$ which has dimension
$\frac{n!}{k_1!\cdots k_s!}\dim M$. Note that if $M$ is a
$\Hd$-module, then 
$\Ind_{\Sd}^{\Sn}M$ naturally has a structure of $\Hnu$-module,
which agrees with  
$\Ind_{\Hd}^{\Hnu}M$.


\begin{theorem}
\label{theorem-mackey}
Let $V$ be a representation of $\Sn$, inflated to a representation
of $\Hny$ by setting $y_i$ to act as $0$.
The induced module $\Delta_{t,c}(V)
=\Ind_{\Hny}^{H_{t,c}}(V)$ 
has a filtration such that subquotients are
isomorphic  as $\Hnu$-modules to the induced representations 
$$
V_{\bd}:=\Ind_{\Hd}^{\Hnu}\left[\Res^{\Hnu}_{\Hrevd}V\right ]^{\revdtwist},\ \bd \in \Pdec,
$$
where here we inflate $V$  along the homomorphism $\ev$.
\end{theorem}

\begin{proof}
\MVcomment{
By definition, $\Delta(V)$ is spanned by $x^a\otimes V$, for all $x^a=x_1^{a_1}\cdots x_n^{a_n}$. For fixed 
$b=(b_1,\ldots,b_n),b_1\le \ldots\le b_n$, we can consider the subspace
$$
V_b=\bigoplus_{\sort(a)=b} x^a\otimes V.
$$
Let $V_{\le b}=\bigoplus_{b'\preceq b}V_b$ \EG{in which order?}, and $V_{<b}=\bigoplus_{b'\prec b}V_b$.
To prove the theorem, we need to establish the following claims:
\begin{itemize}
\item[(a)] $V_b$ is preserved by $\Sn$ action
\item[(b)] As $\Sn$ representation, $V_b\simeq \Ind_{G_b}^{\Sn}\left[\Res^{\Sn}_{G_b}V\right]$
\item[(c)] The action of $u_i$ sends $V_b$ to $V_{\le b}$.
\item[(d)] The associated graded action  $u_i: V_{b}\to V_{\le b}/V_{<b}\simeq V_{b}$ matches the action 
of $u_i$ on $M_b$.
\end{itemize}

Claim (a) is obvious. Claim (b) is also clear since the vectors $a$ such that $\sort(a)=b$ match the cosets in 
$\Sn/G_b$. \EG{To prove claims (c) and (d) we need to understand better the combinatorics of the action of $u_i$ above....}
} 

Let $\bd \in \Pdec$.
By Lemma \ref{lemma-algebra-basis} 
and Lemma \ref{claim:coset}
$\RCA$ has  filtrations
\begin{align*}
&\cB_{\le \bd} = \bigoplus_{\substack{\ba \in \Pdec \\
||\ba|| \le ||\bd|| \\
||\ba|| = ||\bd|| \implies \ba \glex \bd 
 }}  \C \Sn F_{\cX}^{-1}(\wa{a}) \Hny \\
&\cB_{< \bd} = \bigoplus_{\substack{\ba \in \Pdec \\
||\ba|| \le ||\bd|| \\
||\ba|| = ||\bd|| \implies \ba \lex \bd 
 }}  \C \Sn F_{\cX}^{-1}(\wa{a}) \Hny
\end{align*}
clearly preserved by $\C \Sn$. 
By Lemma \ref{lemma:lex vs bruhat}
the filtrations are
also preserved by $\Hnu$.
These induce filtrations on $\Delta(V)$
with subquotients
\MVcomment{
\begin{gather*}
V_{\bd} = \bigoplus_{\substack{\ba \in \Pdec \\
||\ba|| = ||\bd|| \text{ and } \ba \glex \bd 
 }}  \C \Sn \wa{a} \otimes_{\C \Sn} V
\quad /
\bigoplus_{\substack{\ba \in \Pdec \\
||\ba|| = ||\bd|| \text{ and } \ba \lex \bd 
 }}  \C \Sn \wa{a} \otimes_{\C \Sn} V.
\end{gather*}
} 
\begin{gather*}
V_{\bd} =  \cB_{\le \bd} \Delta_{t,c}(V)
\; /  \;  
 \cB_{< \bd} \Delta_{t,c}(V).
\end{gather*}
\MVcomment{  
Because
$\Sn \wa{d} \Sn
= \bigsqcup\limits_{g \in \Sn/\Sd} g\wb{ \bd} \Sn
= \bigsqcup\limits_{g \in \Sn/\Sd} \wb{g\cdot \bd} \Sn$,
the following spaces are isomorphic
not just as  vector spaces, but as $\C \Sn$-modules,
$V_{\bd} \simeq \C \Sn F_{\cX}^{-1}(\wa{d}) \otimes_{\C \Sn} V$.
In particular, as a $\C \Sn$-module, $V_{\bd}$ is generated by 
$F_{\cX}^{-1}(\wa{d}) \otimes V$, and
is spanned by the independent spaces
$F_{\cX}^{-1}(\wb{g \cdot \bd}) \otimes_{\C} V$, for $g \in \Sn/\Sd$.
Note that if $s_i \in \Sd$ then
$\FXi(s_i \wa{d}) \otimes V =
\FXi(\wa{d} (\wa{d}^{-1} s_i \wa{d} ))\otimes V
=\FXi(\wa{d}  \revdtwist(s_i) ) \otimes V
=\FXi(\wa{d} ) \otimes  \revdtwist(s_i)V.$
Further as 
$u_j \wa{d} =
=(\wa{d}  \revdtwist(u_j) + \mathrm{\ell.o.t})$ , we have 
$u_j \FXi(\wa{d}) \otimes V =
(\FXi(\wa{d})  \revdtwist(u_j) + \mathrm{\ell.o.t})  \otimes V
\equiv \FXi(\wa{d})   \otimes  \revdtwist(u_j) V$
since the lower order terms here involve $\wa{a}$ with
$\ba \lex \bd$
by Lemma \ref{lemma:lex vs bruhat},
 and these are killed in $V_{\bd}$.
Thus we see that as an $\Hnu$-module
$V_{\bd}\simeq \Ind_{\Hd}^{\Hnu}\left[\Res^{\Hnu}_{\Hrevd}V
\right]^{\revdtwist}.$
} 
In the following argument we lighten notation, 
writing $\fxi{p}$ for $\FXi(p)$, so for instance
the above expressions would become
 $ \C \Sn \fxi{\wa{a}} \Hny$.

Because
$\Sn \wa{d} \Sn
= \bigsqcup\limits_{g \in \Sn/\Sd} g\wb{ \bd} \Sn
= \bigsqcup\limits_{g \in \Sn/\Sd} \wb{g\cdot \bd} \Sn$,
the following spaces are isomorphic
not just as  vector spaces, but as $\C \Sn$-modules,
$V_{\bd} \simeq \C \Sn \fxi{\wa{d}} \otimes_{\C \Sn} V$.
In particular, as a $\C \Sn$-module, $V_{\bd}$ is generated by 
$\fxi{\wa{d}} \otimes V$, and
is spanned by the independent spaces
$\fxi{\wb{g \cdot \bd }}\otimes_{\C} V$, for $g \in \Sn/\Sd$.
Note that if $s_i \in \Sd$ then
$\fxi{s_i \wa{d}} \otimes V =
\fxi{\wa{d} (\wa{d}^{-1} s_i \wa{d} )}\otimes V
=\fxi{\wa{d}  \revdtwist(s_i) } \otimes V
=\fxi{\wa{d} } \otimes  \revdtwist(s_i)V.$
Further as $u_j \fxi{\wa{d}} 
=(\fxi{\wa{d}}  \revdtwist(u_j) + \mathrm{\ell.o.t})$ , we have 
$u_j \fxi{\wa{d}} \otimes V =
(\fxi{\wa{d}}  \revdtwist(u_j) + \mathrm{\ell.o.t})  \otimes V
\equiv \fxi{\wa{d}}   \otimes  \revdtwist(u_j) V$
since the lower order terms here involve $\fxi{\wa{a}}$ with
$\ba \lex \bd$
by Lemma \ref{lemma:lex vs bruhat},
 and these are killed in $V_{\bd}$.
Thus we see that as an $\Hnu$-module
$V_{\bd}\simeq \Ind_{\Hd}^{\Hnu}\left[\Res^{\Hnu}_{\Hrevd}V
\right]^{\revdtwist}.$

\end{proof}

\begin{remark}
\label{rem-classical Mackey}
One can regard this theorem as a version of the classical Mackey formula:
\begin{multline*}
\Res_{K}\Ind_{H}^{G}(\rho)
=\bigoplus_{\ww\in K\backslash G/H}\Ind^{K}_{\ww H\ww ^{-1}\cap K}(\rho^\ww) 
\\
=\bigoplus_{\ww\in K\backslash G/H}\Ind^{K}_{\ww H\ww^{-1}\cap K}
(\Res^{H}_{H\cap \ww^{-1}K\ww}\rho)^\ww,
\end{multline*}
where $G$ is a finite group, $H,K$ are its subgroups, $\rho$  is a representation of $H$ and 
$\rho^\ww(x)=\rho(\ww^{-1}x\ww)$.

Our setting shares many features with classical Mackey
for the case $G=\affSn$, $H=K=\Sn$, where the minimal
length double coset representatives are
$\{ \wa{d} \mid \bd \in \Z^n, d_1 \ge \cdots \ge d_n\}$.
In that case $\Sd = \wa{d} \Sn \wa{d}^{-1} \cap \Sn$,
$\Srevd =  \Sn \cap \wa{d}^{-1} \Sn \wa{d}$ and one computes
the action on an induced module via
$p \left( \wa{d} \otimes V\right)
= \wa{d} \otimes (\wa{d}^{-1} p \wa{d}) V
= \wa{d} \otimes p V^{\wa{d}}$, which is also equal to
$ \wa{d} \otimes p V^{\revdtwist} 
 = \wa{d} \otimes \revdtwist(p) V $
for $p \in \Sd$.

In our setting
$\Hd$ plays the role
of $\wa{d} H \wa{d}^{-1} \cap K$;
$\Hrevd$ the role of $H \cap \wa{d}^{-1} K \wa{d}$.
$\FXi(\wa{d}^{-1} p \wa{d})=:\fxi{\wa{d}^{-1} p \wa{d}}$
makes sense for $p \in \Sd$.
On the other hand, $\wa{d}^{-1} u_i \wa{d}$ is problematic
on many levels. 
This is in part why we must work with the isomorphism
$\revdtwist$ above.

While conjugation by $\wa{d}^{-1}$ or by $g_{\bd}$ gives
us in isomorphism from $\Sd$ to $\Srevd$ when working
inside of $\affSn$, 
the most  natural way to extend the notion of 
conjugation by $\wa{d}^{-1}$ to $\RCA$ does {\em not}
send $\Hd$ to $\Hrevd$. 
While $\FXi(\wa{d}) =:\fxi{\wa{d}}$ is not invertible, 
this is not the main obstruction;
one can localize and invert the $x_i$
(as one does with the trigonometric Cherednik
algebra \cite{Suzuki}).
This essentially replaces $\tau$ with $\pi$ and adjoins $\pi^{-1}$,
so would enlarge our algebra 
and embed a copy of $\affSn$.
We can define $\pi u_i \pi^{-1} = u_{\pi(i)}= u_{i+1}$ and
$\pi^{-1} u_i \pi = u_{i-1}$ using the convention in Remark
\ref{remark-ui for i in Z}, and this is compatible with
relation \eqref{tau ui}.
This allows us to define conjugation by $\wa{d}^{-1}$.
 It will still send $\Sd \to \Srevd$
but will not send $\Hd \to \Hrevd$ as conjugation
by $\Sn$ does not preserve $\A$ (even though conjugation
by $\pi$ does preserve $\A$).
For $g\in \Sn$ recall that $\Hnu \ni g^{-1} u_i g
= u_{g^{-1}(i)} + \mathrm{\ell.o.t}$,
where here lower is with respect to $u$ degree. 
More specifically, $u_i g = g u_{g^{-1}(i)} + $ terms
$ \bruhat g$ in Bruhat order. 
These are the lower order terms we throw away when considering
$V_{\bd}$ or $  \cB_{\le \bd} / \cB_{< \bd}$.  
Throwing away these lower order terms agrees with
replacing conjugation by $\wa{d}^{-1}$ with the
isomorphism $\revdtwist : \Hd \to \Hrevd$
when describing the Mackey filtration.

\end{remark}

\MVcomment{
\begin{remark}
One can regard this theorem as a version of the classical Mackey formula:
$$
\Res_{K}\Ind_{H}^{G}(\rho)
=\bigoplus_{s\in K\backslash G/H}\Ind^{K}_{sHs^{-1}\cap K}(\rho^s)
=\bigoplus_{s\in K\backslash G/H}\Ind^{K}_{sHs^{-1}\cap K}
(\Res^{H}_{H\cap s^{-1}Ks}\rho)^s,
$$
where $G$ is a finite group, $H,K$ are its subgroups, $\rho$  is a representation of $H$ and 
$\rho^s(x)=\rho(s^{-1}xs)$.
\end{remark}
} 

As a corollary to Theorem \ref{theorem-mackey} we have the following.
\begin{corollary}
\label{cor-mackey}
Let $V$ be a $\C \Sn$-module such that when inflated along
$\ev$ to be an $\Hnu$-module it has $u$-weight basis
$\{ v_T \mid T \in \cT \}.$ Let $\wt_T$ denote the
weight of $v_T$.  If we assume $t \neq 0$, then the $\RCA$-module
$\Ind_{\Hny}^{\RCA} V$ has finite dimensional generalized
 $u$-weight spaces 
 and a generalized $u$-weight basis
indexed by $\Pn \times \cT$. Its generalized weights are
\begin{gather*}
\{ \wa{a}\cdot \wt_T \mid \ba \in \Pn, T \in \cT\}
= \{ g \wa{d}\cdot \wt_T \mid \bd \in \Pdec, g \in \Sn/\Sd, T \in \cT\}.
\end{gather*}

When $t = 0$, the weights are still given by the formula above but the $u$-weight spaces are no longer finite-dimensional. We study this case in detail in Section \ref{sect:m to infinity}. 
\end{corollary}
It is worth noting that given fixed $\wt = (\wt_1,  \ldots,
\wt_n) \in \C^n$, $\bd = (d_1, d_2, \ldots, d_n) \in \Pdec$ the set
\begin{gather*}
 \{ g \wa{d}\!\cdot \!\wt \mid  g \in \Sn/\Sd \} =
	\{f\!\cdot\! (\wt_1 \!+ \!d_n t, \wt_2 \!+ \!d_{n-1} t, \ldots, \wt_n \! +\! d_1 t)
\mid f \in \Sn/\Srevd \}.
\end{gather*}


\begin{example}
\label{example-22}
Let us consider the Mackey formula for $M=\Delta_c(\triv)$
in the case $n=2, c=2, t=1$.
As we shall see, it is not $\A$-semisimple.
For weights of the form
$\wt = (d,d)$, $\dim M_{(d,d)}^{\gen} = 2 > 1 = \dim M_{(d,d)}$.
For all other weights $\wt = (\wt_1, \wt_2)$
with $ \wt_1 \neq \wt_2$
that occur,
 its generalized weight spaces $\Mw^{\gen} = \Mw$
are $1$-dimensional.

We have a single $T =$ 
\begin{tikzpicture}[scale=0.5]
\draw[step=1] (0,0) grid (2,1);
\draw (0.5,0.5) node[black] {$1$};
\draw (1.5,0.5) node[black] {$2$};
\end{tikzpicture}
and $v_T$ has weight $\wt = (0,-c) = (0,-2)$.
For $\bd \in \Pdec$ we split into two cases according to $\Sd$.

{\it Case 1.}
$\bd = (d,d) = \revd$, $\Sd = \Sk{2}$. 
Thus our induction and restriction functors are trivial
and
$\Ind_{\Hd}^{H_2(\mathbf{u})}[\Res_{\Hrevd}^{H_2(\mathbf{u})}
\triv]^{\revdtwist} = \triv^{\revdtwist}$.
The module 
$ \triv^{\revdtwist}$ still carries the trivial action of  $\Sk{2}$,
but $u_1 = 0+d, u_2 = -c + d$.
In other words it corresponds to a weight vector
$ v_{\bd} = v_{(d, d)} \in \Delta_c(\triv)$ of weight
$\wt = (d, d-2)$.
Recall we require $d \ge 0$.

{\it Case 2.}
$\bd = (d_1 > d_2)$,   $\revd= (d_2, d_1)$, $g_{\bd} = s_1$,
and $\Sd = \{ \id \} = \Srevd$.
Now
$\Res_{\Hrevd}^{H_2(\mathbf{u})} \triv
=\Res_{\A}^{H_2(\mathbf{u})} \triv = (0)\boxtimes (-c)$
where we write the one-dimensional
 $\A = \C[u_1]\otimes \C[u_2]$-module on which
$u_1-\alpha$ and $u_2-\beta$ vanish as $(\alpha) \boxtimes (\beta)$.
The twisted module is
$$[(0)\boxtimes (-c)]^{\revdtwist}
= (-c+d_1) \boxtimes (0+d_2) = (d_1-2) \boxtimes (d_2).$$
Finally
\begin{gather*}
\Ind_{\Hd}^{H_2(\mathbf{u})}[\Res_{\Hrevd}^{H_2(\mathbf{u})}
\triv]^{\revdtwist} = 
\Ind_{\A}^{H_2(\mathbf{u})} (d_1-2) \boxtimes (d_2).
\end{gather*}
This is an irreducible $2$-dimensional $H_2(\mathbf{u})$-module.
\\
In the special case $d_2 = d_1-2$ it is not $A$-semisimple.
In other words the $u_i$ act with Jordan blocks of size $2$.
The generalized $\wt =(d_1-2,d_2) = (d_2, d_2)$-weight space
is $2$-dimensional and
corresponds to the basis vectors in $\Delta_c(\triv) $ which 
by abuse of notation we can still  call
$ v_{\bd} = v_{(d_2+2,d_2)}$ and $v_{s_1 \bd} = v_{(d_2, d_2+2)}$.
\\
When $d_2 \neq d_1-2$ we get one-dimensional weight spaces
spanned by
$ v_{\bd} = v_{(d_1, d_2)}$ of weight $\wt = (d_1-2,d_2)$
 and $v_{s_1 \bd} = v_{(d_2, d_1)} $ of weight $s_1 \cdot \wt =
(d_2, d_1-2)$. 
Because these (generalized) weights occur with multiplicity one,
the Mackey filtration tells us these generalized weight spaces
are true weight spaces.

\end{example}

\section{Representation theory of \texorpdfstring{$H_{t,c}$}{Htc}}\label{sec:rep theory}

\subsection{Generalized eigenspaces and intertwining operators}

As above, we will denote by $\A \subseteq H_{t,c}$ the polynomial subalgebra generated by the Dunkl-Opdam elements $u_{1}, \dots, u_{n}$. 

For an $H_{t,c}$-module $M$ and $\tw \in \C^n$, let $M_{\tw}^{\gen}$ denote the generalized eigenspace with weight $\wt$, that is, $(u_{i} - \wt_{i})$ acts locally nilpotently on $M_{\tw}^{\gen}$ for every $i$. We also denote by $M_{\tw} \subseteq M_{\tw}^{\gen}$ the subspace of honest simultaneous
eigenvectors. At $t = 1$ the Euler element is $h = \sum u_{i} + n/2$ and it is therefore easy to see that every module in category $\cO_{c}$ is locally finite for the $\A$-action, so that it decomposes as the direct sum of its generalized weight spaces, and each such space is finite-dimensional. Note that this also follows easily from Theorem \ref{theorem-mackey}.

We are interested in the spectrum of $\A$ on standard modules.
To study it, we will make use of the following intertwining operators, cf. \cite{CherednikIntertwining, griffeth, KnopSahi}
$$
\sigma_{i} := s_i - \frac{c}{u_{i} - u_{i+1}}, i = 1, \dots, n-1,\qquad \tau = x_{1}(12\cdots n).
$$

Note that $\tau \in H_{t,c}$, while the $\sigma_{i}$ are elements of the localization $H_{t,c}[(u_{i} - u_{j})^{-1} : i \neq j]$. Alternatively, given a representation $M$, we may think of $\tau$ as an operator which is defined globally on $M$, while $\sigma_{i}$ is only defined on those generalized eigenspaces $M^{\gen}_{\tw}$ for which $\tw_{i} - \tw_{i+1} \neq 0$, i.e., 
$s_i \cdot \wt \neq \wt$.

\begin{lemma}\label{lem:braid}\cite[(4.13)]{griffeth}
We have $\sigma_{i}\sigma_{i+1}\sigma_{i} = \sigma_{i+1}\sigma_{i}\sigma_{i+1}$ and, if $i < n$, $\tau\sigma_{i} = \sigma_{i+1}\tau$. Furthermore, 
\begin{equation}
\label{eq:sigma square}
\sigma_{i}^{2} = \frac{(u_{i} - u_{i+1} - c)(u_{i} - u_{i+1} + c)}{(u_{i} - u_{i+1})^{2}}.
\end{equation}
and $\lambda \sigma_{i} = \sigma_{i-1} \lambda$ if $i>1$,
also $\lambda \sigma_1 \tau =  \tau \sigma_{n-1} \lambda - 
\frac{u_0}{u_0-u_1}c = \tau \sigma_{n-1} \lambda - 
\frac{u_n +t}{u_n+t-u_1}c $.
\end{lemma}


It is not hard to see that we have
$$
\sigma_{i}: M^{\gen}_{\tw} \to M^{\gen}_{s_{i}\cdot \tw}, \qquad \tau: M^{\gen}_{\tw} \to M^{\gen}_{\pi\cdot \tw}
$$

\noindent where the symmetric group $\Sk{n}$ acts on $\C^n$ by
permuting the coordinates, and
$\pi\cdot (\wt_{1}, \dots, \wt_{n})
= (\wt_{n}+t, \wt_{1}, \dots, \wt_{n-1})$
as in Equation \eqref{AffSymmOnCn}.
In the first case we assume $s_i \cdot \wt \neq \wt$ so $\sigma_i$
is well-defined.
\begin{remark}
\label{remark: intertwiners vanish}
Note that, if $\sigma_{i}|_{M_{\tw}^{\gen}} = 0$, then $\tw_{i} -
\tw_{i+1} = \pm c$.   If $M$ is free as a $\C[x_{1}, \dots,
x_{n}]$-module (for example, a standard module) then
$\tau|_{M_{\tw}^{\gen}} \neq 0$ provided $M_{\tw}^{\gen} \neq 0$.
\end{remark}
\begin{remark}
\label{remark: renormalize intertwiners}
Note that $\sigma_i = (s_i u_i - u_i s_i)/(u_i - u_{i+1})$. 
These operators are well-defined on any simple $\RCA$-module
on which $\A$ acts semisimply. 
It is sometimes convenient to instead consider
$$\tilde \sigma_i :=  (s_i u_i - u_i s_i)/(u_i - u_{i+1} + c).$$ 
These also satisy the braid relations and their 
quadratic relation becomes $\tilde \sigma_i^2 = 1$.
We will see below (see Section \ref{sect:renormalization}) that these operators are well-defined
on $\Lctriv$.
\end{remark}

Because the intertwiners satisfy the braid relations, $\sigma_\omega$ and $\widetilde{\sigma}_\omega$ are well-defined for $\omega \in \affSn$, if one takes the convention $\sigma_\pi = \widetilde{\sigma}_\pi = \pi$. 

Using intertwiners to construct and parameterize an
 $\mathcal{A}$-weight basis
for an $\mathcal{A}$-semisimple (or calibrated) module,
as well as giving the action of generators on that basis,
follows ideas developed by  Ram in \cite{Ram} or Cherednik
in \cite{Cherednik}.
In \cite{Ram} the role of $\mathcal{A}$ was instead
played by  an appropriate commutative subalgebra of
the affine Hecke algebra, but the constructions apply in our
context as well.

\subsection{The standard modules}\label{sect:other standards}
From now on, unless otherwise explicitly stated, we will assume that the parameter $c$ has the form $c = m/n > 0$ with $\gcd(m,n) = 1$ and $t = 1$. In this section, we will analyze the action of the Dunkl-Opdam subalgebra on a standard module $\Delta_{c}(\mu)$. We will denote by $\SYT(\mu)$ the set of standard tableaux on $\mu$. For $T \in \SYT(\mu)$, $T_{i}$ denotes the box of $\mu$ labeled by $i$ under $T$, and 
$\ct{T}{i}$ is the content of this box. 

\begin{definition}\label{def: weight labels}
Let $(\ba, T) \in \Z_{\geq 0}^{n} \times \SYT(\mu)$. Denote by $\tw(\ba, T) \in \C^n$ the weight whose $i$-th component is $\tw_i(\ba, T) = a_{i} - \ct{T}{g_{\ba}(i)}c$ where, recall, $g_{\ba} \in \Sk{n}$ is the minimal permutation that sorts $\ba$ increasingly. 
\end{definition}

From Lemma \ref{lemma:g vs pi}, it is clear that we have that the intertwining operators send $\tau: \Delta_c(\mu)_{\wt(\ba, T)} \to \Delta_c(\mu)_{\wt(\pi\cdot \ba, T)}$ and, if $a_{i} \neq a_{i+1}$, $\sigma_{i}: \Delta_c(\mu)_{\wt(\ba, T)} \to \Delta_c(\mu)_{\wt(s_{i}\ba, T)}$. 

The following result was proved in \cite[Theorem 5.1]{griffeth}.

\begin{theorem}\label{thm:vermas}
Let $c = m/n > 0$ with $\gcd(m,n) = 1$ and $M = \Delta_c(\mu)$. Then, for any $(\ba, T) \in \Z_{\geq 0}^{n} \times \SYT(\mu)$ we have $M_{\wt(\ba, T)} = M_{\wt(\ba, T)}^{\gen}$ is $1$-dimensional. Moreover, if $a_{i} > a_{i+1}$, then $\sigma_{i}|_{M_{\wt(\ba, T)}} \neq 0$, and the action of the Dunkl-Opdam subalgebra on $M$ is diagonalizable with eigenvalues given by $\wt(\ba, T)$.
\end{theorem}


\begin{proof}
The operators $u_i$ act on $V_{\mu}$ as classical Jucys-Murphy
operators, and have spectrum $-\ct{T}{i}c$ for $T\in \SYT(\mu)$.
In other words the vector $v_T$ has weight $\wt_T =
(-\ct{T}{1}c, \ldots, -\ct{T}{n}c)$.
By Corollary \ref{cor-mackey} the generalized $u$-weights of $u_i$ on
$\Delta_c(\mu)$ are given by $ \wa{a} \wt_T =\wt(\ba,T)$. Now the
theorem follows from Lemma \ref{lemma:injective} below. 
\end{proof}

\begin{lemma}\label{lemma:injective}
Let $(\ba, T), (\bb, S) \in \Z_{\geq 0}^{n} \times \SYT(\mu)$. If $\wt(\ba, T) = \wt(\bb, S)$ then $\ba = \bb$ and $T = S$. 
\end{lemma}
\begin{proof}
Assume $\wt(\ba, T) = \wt(\bb, S)$. Then, for every $i = 1, \dots, n$, 
$$
a_{i} - b_{i} = c(\ct{T}{g_{\ba}(i)} - \ct{S}{g_{\bb}(i)}).
$$ 
But $T$ and $S$ have the same shape $\mu$, so
$$
\ct{T}{g_{\ba}(i)} - \ct{S}{g_{\bb}(i)} \in \{-n+1, -n+1, \dots, 0, \dots, n-2, n-1\}
$$ 
so, by our assumption on $c = m/n$, we must have $a_{i} - b_{i} = 0$. From here, we have $\ba = \bb$ and $\ct{S}{i} = \ct{T}{i}$ for every $i = 1, \dots, n$, which implies $S = T$. 
\end{proof}

\begin{remark}
\label{rem:generalized eigenvalues}
For arbitrary $t, c$, Corollary \ref{cor-mackey} still applies and the same proof shows that the generalized eigenvalues of $u_i$ on $\Delta_{t,c}(\mu)$
are given by $\wt(\ba,T)$.
\end{remark}




\subsection{Recovering the action of \texorpdfstring{$H_{c}$}{Hc}}\label{subsect:action} Let us now see that for $c=m/n,\gcd(m,n)=1$ the action of $H_{c}$ on the standard module $\Delta_c(\mu)$ is completely determined by the spectrum of the Dunkl-Opdam subalgebra. Fix an eigenbasis $\{v_{T} : T \in \SYT(\mu)\}$ of $V_{\mu}$ for the Jucys-Murphy operators. Note that we have the basis $v_{T}(\omega)$, $\omega \in \Min$ of $\Delta_c(\mu)$, cf. Corollary \ref{cor: basis in induced module}.  For each $\ba \in \mathbb{Z}^{n}_{\geq 0}$ and every standard Young tableau $T$ on $\mu$, denote by $v(\ba, T) \in \Delta_c(\mu)_{\wt(\ba, T)}$ a nonzero vector, normalized so that $v_{T}(\wa{a})$ appears with coefficient $1$ in $v(\ba, T)$. 
The action of $H_{c}$ on $\Delta_{c}(\mu)$ is given by the following formulas:
\begin{theorem}\label{thm:action standard}
The module $\Delta_c(\mu)$ has a basis given by $\{v(\ba, T) : \ba \in \Z^n_{\geq 0}, T \in \SYT(\mu)\}$, and the action of the algebra $H_{c}$ on $\Delta_c(\mu)$ is given by the following operators:

\begin{align*}
u_{i} \, v(\ba, T) & =\wt_{i}v(\ba, T) \notag \\
\tau \, v(\ba, T) &  = v(\pi\cdot \ba, T) \notag \\
\lambda \, v(\ba, T) & = \wt_{1}v(\pi^{-1}\cdot \ba, T) \\
\end{align*}
where $\wt=\wt(\ba,T)$. The action of $s_i$ with respect to this basis falls 
into the following five cases.
Let $j = g_{\ba}(i)$. Then
\newline
$s_{i}v(\ba, T) =$
\begin{gather*}
\begin{cases} 
v(s_{i}\cdot\ba, T) + \frac{c}{\wt_{i} - \wt_{i+1}}v(\ba, T) 
     & a_{i} > a_{i+1}, \\ 
\frac{(\wt_{i+1} - \wt_{i} - c)(\wt_{i+1} - \wt_{i} + c)}{(\wt_{i} - \wt_{i+1})^{2}}v(s_{i}\cdot \ba, T) +
\\ \qquad \qquad \frac{c}{\wt_{i+1}- \wt_{i}}v(\ba, T) 
     & a_{i} < a_{i+1}, \\
 v(\ba, T) 
     & a_{i} = a_{i+1} \text{ and $j, j\!+\!1$ in same row of $T$}, \\
 -v(\ba, T) 
     & a_{i} = a_{i+1}  \text{ and $j, j\!+\!1 $ in same column of $T$},
 \\
  v\left(\ba, s_{j}(T)\right) + \frac{c}{\wt_{i} - \wt_{i+1}}v(\ba, T)
      & a_{i} = a_{i+1}, \;  \; s_{j}T  \in \SYT(\mu)
\end{cases}
\end{gather*}
\end{theorem}
\MVcomment{
 \begin{gather*}  
s_{i} \, v(\ba, T) =
\begin{cases} 
v(s_{i}\cdot\ba, T) + \frac{c}{\wt_{i} - \wt_{i+1}}v(\ba, T) 
     & a_{i} > a_{i+1} \\ 
\frac{(\wt_{i+1} - \wt_{i} - c)(\wt_{i+1} - \wt_{i} + c)}{(\wt_{i} - \wt_{i+1})^{2}}v(s_{i}\cdot \ba, T) +
\\ \qquad \qquad \frac{c}{\wt_{i+1}- \wt_{i}}v(\ba, T) 
     & a_{i} < a_{i+1}  \\
 v(\ba, T) 
     & a_{i} = a_{i+1},  \qquad
     \begin{tikzpicture}[scale=0.6]
\draw (0,0) grid  (2,1);
\draw (0.5,0.5) node {\tiny $j$};
\draw (1.5,0.5) node {\tiny $j\!+\!1$};
\draw (3,0.5) node {$ \in T$};
\end{tikzpicture}  \\
 -v(\ba, T) 
     & a_{i} = a_{i+1},  \qquad
     \begin{tikzpicture}[scale=0.6]
\draw (0,0) grid  (1,2);
\draw (0.5,1.5) node {\tiny $j$};
\draw (0.5,0.5) node {\tiny $j\!+\!1$};
\draw (2,1) node {$ \in T$};
\end{tikzpicture}
 \\
  v\left(\ba, s_{j}(T)\right) + \frac{c}{\wt_{i} - \wt_{i+1}}v(\ba, T)
      & a_{i} = a_{i+1}, \qquad s_{j}T  \in \SYT(\mu)
\end{cases}
\end{gather*}
\MV{since  $g_{\ba}(i+1) = g_{\ba}(i)+1$ do  $j, j+1$  thought it looked cleaner to do that subst}
 } 


\begin{proof}
Thanks to our normalization, we have that if $a_{i} > a_{i+1}$, then $\sigma_{i}v(\ba, T) = v(s_{i}\ba, T)$. From here and \eqref{eq:sigma square} one can deduce the first two cases. 

Finding $s_{i}v(\ba, T)$ when $a_{i} = a_{i+1}$ is subtler. The weight of $s_{i}v(\ba, T)$ is $s_{i}\cdot\wt$. Note that, in this case, $g_{\ba}(i+1) = g_{\ba}(i) + 1 = j+1$. 
Let us denote by $s_{j}(T)$ the tableau that is obtained from $T$ by permuting the entries $j$ and $j+1$. Note that $s_{j}(T)$ may not be standard, and this is the case precisely when in the tableau $T$ we have
\begin{itemize}
\item[(1)] $T_{j} = (R, C)$ and $T_{j + 1} = (R, C+1)$ for some box $(R, C)$ in $\mu$, i.e., we see \begin{tikzpicture}[scale=0.6]
\draw (0,0) grid  (2,1);
\draw (0.5,0.5) node {\tiny $j$};
\draw (1.5,0.5) node {\tiny $j\!+\!1$};
\draw (3,0.5) node {$ \in T$ or};
\end{tikzpicture}  
\item[(2)] $T_{j} = (R, C)$ and $T_{j + 1} = (R + 1, C)$ for some box $(R, C)$ in $\mu$, i.e., we see \begin{tikzpicture}[scale=0.6]
\draw (0,0) grid  (1,2);
\draw (0.5,1.5) node {\tiny $j$};
\draw (0.5,0.5) node {\tiny $j\!+\!1$};
\draw (2,1) node {$ \in T$.};
\end{tikzpicture} 
\end{itemize}

In these cases, $s_{i}\cdot\wt(a,T)$ is not of the form $\wt(\ba', T')$ for a standard tableau $T'$, so we must have $\sigma_{i}v(\ba, T) = 0$ and therefore $s_{i}v(\ba, T) = \pm v(\ba, T)$. 
Using the explicit formula $\sigma_{i} = s_{i} - \frac{c}{u_{i} - u_{i+1}}$ we 
recover the third and fourth cases.
Finally, if $s_{j}T$ 
is a standard tableau, then $\sigma_{i}v(\ba, T) = v(\ba, s_{j}(T))$ and we  have the fifth and final case. 

Note that for the original weight basis $v_T$ of the $\Sn$-module indexed by $\mu$, we have normalized them to agree with the third, fouth, and fifth cases above at $\ba = {\mathbf 0}$.
Thus we have recovered the action of $H_{c}$ on $\Delta_c(\mu)$. 
\end{proof}
\subsection{Maps between standards} In this section and the next one, we study maps between standard modules. Note that similar results were obtained in \cite{balagovic} for the case of trigonometric Cherednik algebras and in \cite{Ram, GNP} for the case of affine Hecke algebras.

\begin{lemma}
\label{lem:weights mod n}
Suppose that $c=m/n,\gcd(m,n)=1$.
Let $d_i(\mu)$ be the number of boxes in $\mu$ with content $i\bmod n$. Then for all
 $(\ba,T)\in \Z^n_{\ge 0}\times \SYT(\mu)$ one has  
 $$
 \sharp \{j:\ n\wt_j(\ba,T)\equiv -mi\bmod n \}=d_i(\mu).
 $$
\end{lemma}

\begin{proof}
We have 
$$
n\wt_j(\ba,T)=na_j-m\ct{T}{g_{\ba}(j)}\equiv
-m\ct{T}{g_{\ba}(j)}\bmod n.
$$ 
Since $g_{\ba}$ is a permutation, $T_{g_{\ba}(j)}$ runs over all boxes in $\mu$ and 
the vector $\ct{T}{g_{\ba}(j)}$ has exactly $d_i(\mu)$ entries equal to $i\bmod n$.
\end{proof}

\begin{remark}
A similar argument and Remark \ref{rem:generalized eigenvalues} show that for $c=m/\ell,$ $ \gcd(m,\ell)=1$ one has
$$
\sharp \{j:\ \ell \wt_j(\ba,T)\equiv -mi\bmod \ell \}=d_i^{(\ell)}(\mu),
$$
where $d_i^{(\ell)}(\mu)$ is the number of boxes in $\mu$ with content $i\bmod \ell$.
\end{remark}

\begin{lemma}
\label{lem:not hooks}
Suppose that $c=m/n,\gcd(m,n)=1$.
Let $\mu\neq \mu'$ be two partitions of $n$. Then 
$\Hom_{H_c}(\Delta_c(\mu),\Delta_c(\mu'))=0$ unless both $\mu$ and $\mu'$ are hook partitions.
\end{lemma}

\begin{proof}
Suppose that $\Hom_{H_c}(\Delta_c(\mu),\Delta_c(\mu'))\neq 0$, then $\wt(\ba,T)=\wt(\ba',T')$ for some $(\ba,T)\in \Z_{\ge 0}\times \SYT(\mu)$ and $(\ba',T')\in \Z^n_{\ge 0}\times \SYT(\mu')$. By Lemma \ref{lem:weights mod n} we get
$d_i(\mu)=d_i(\mu')$ for all $i$, which implies that $\mu$ and $\mu'$ have the same $n$-core \cite{JK}. 

Since $\mu$ has size $n$, either its $n$-core is empty and $\mu$ is a hook, or $\mu$ is an $n$-core itself.
The same applies to $\mu'$, so they can share an $n$-core only if both partitions are hooks.
\end{proof}

\begin{remark}
A similar argument shows that for $c=m/\ell,\ \gcd(m,\ell)=1$ one could possibly have
$\Hom_{H_c}(\Delta_c(\mu),\Delta_c(\mu'))\neq 0$ only if $\mu,\mu'$ have the same $\ell$-core.
 This is known via the KZ functor, cf. \cite{BEG} and we have obtained a purely combinatorial proof. 
\end{remark}

\begin{corollary}
\label{cor:not hook}
Let $c=m/n,\gcd(m,n)=1$.
If $\mu$ is not a hook partition then $\Delta_c(\mu)$ is irreducible. 
\end{corollary}

\begin{proof}
The proof is standard but we include it here for completeness. If $R$ is a submodule of $\Delta(\mu)$, then (since the action of $y_1, \dots, y_n$ is locally nilpotent) there is a vector $v\in R$ such that $y_1v = \cdots = y_nv =0$. It spans a finite-dimensional subspace $U$ under the action of $\Sn$, and $\lambda(U)=0$, it contains an irreducible representation of $\Sn$ isomorphic to $V_{\mu'}$. Then there is a notrivial morphism $H_c$-modules
$\Delta(\mu')\to \Delta(\mu)$ which sends  $V_{\mu'}$ to this subspace.
\end{proof}

We determine the morphisms between $\Delta_c(\mu)$ for hook partitions $\mu$ in the next subsection.

\subsection{The BGG resolution}\label{sect:BGG} Throughout this section we assume $c=m/n,\gcd(m,n)=1$.

 Let us denote by $V_{\mu_{\ell}} := \wedge^{\ell}\C^{n-1}$ the hook
representation of $\Sk{n}$, so that $\mu_{\ell}$ is the partition $(n-
\ell, 1^{\ell})$, $\ell = 0, \dots, n-1$.    In particular, $V_{\mu_{0}}$ is the trivial representation and $V_{\mu_{n-1}}$ the sign representation. 
 It is known
\cite{BEG} that the representation $L_{m/n} := L_{m/n}(\triv)$ admits
a resolution
\begin{equation}\label{eqn:resolution}
0 \to \Delta_c(\mu_{n-1}) \to \cdots \to \Delta_c(\mu_{1}) \to \Delta_c(\mu_{0}) \to 0
\end{equation}

\noindent that in fact coincides with the Koszul resolution of $L_{m/n}$ when considering a standard module as a $\C[x_{1}, \dots, x_{n}]$-module. In this section, we will construct the resolution \eqref{eqn:resolution} in a purely combinatorial manner. 
We remark that this has been recently generalized in \cite{GFM} to some other BGG resolutions.

Let us set up some notation. For each collection $1 < i_{1} < \dots < i_{\ell} \leq n$, let $T_{i_{1} < i_{2} < \cdots < i_{\ell}}$ be the tableau on $\mu_{\ell}$ that has the numbers $1, i_{1}, \dots, i_{\ell}$ on its leg. Clearly, every tableau on $\mu_{\ell}$ is of this form. 

Recall that for each element $(\ba, T) \in \Z_{\geq 0}^{n} \times \SYT(\mu)$ we have a nonzero vector $v(\ba, T) \in \Delta_c(\mu)_{\wt(\ba, T)}$. Clearly, every map on $\Delta_c(\mu)$ is completely determined by the image of the vectors $v_{T} = v(0, T)$, $T \in \SYT(\mu)$. 

\begin{lemma}
\label{lem:far hooks}
Suppose that $\ell\neq j,j+1$. Then $\Hom_{H_c}(\Delta_c(\mu_{\ell}),\Delta_c(\mu_j))=0$.
\end{lemma}

\begin{proof}
Let $T=T_{i_{1} < i_{2} < \cdots < i_{\ell}}$  be a  standard tableau of shape $\mu_{\ell}$, we have
$$
n\wt_{i_1}(0,T)=m,\ldots, n\wt_{i_{\ell}}(0,T)=m\ell,\ n\wt_i(0,T)<0\ \text{for}\ i\neq 1,i_1,\ldots,i_{\ell}.
$$
Now suppose that $\wt(0,T)=\wt(\ba,T')$ for some $(\ba,T')\in \Z^n_{\ge 0}\times \SYT(\mu_j)$. One has
$$
0>n\wt_i(\ba,T')=na_i-m\cdot \ct{T'}{g_{\ba}(i)}\ge -m\cdot \ct{T'}{g_{\ba}(i)} \ \text{for}\ i\neq 1,i_1,\ldots,i_{\ell},
$$
so $\mu_j$ has at least $n-\ell-1$ boxes with positive contents, and $\ell\ge j$.

Suppose that $\ell\ge j+2$. It is easy to see that the equation 
$$
-m\ct Ti=na_{i}-m\ct{T'}{g_{\ba}(i)},
$$ 
implies 
$$
\begin{cases}
a_i=m \text{ and } \ct{T}{i}+n=\ct{T'}{g_{\ba}(i)} &  i=i_{j+1},\ldots,i_{\ell},\\
a_i=0 \text{ and }   \ct{T}{i}=\ct{T'}{g_{\ba}(i)} & \text{otherwise}.\\
\end{cases}
$$
By definition, this implies $g_{\ba}(i_{j+1})=n-\ell+1,\ldots,g_{\ba}(i_{\ell})=n,$
so
$$
\ct{T'}{n-\ell+1}=n-(j+1),\ldots, \ct{T'}{n}=n-\ell.
$$ 
But this means that the first row of $T'$ contains the numbers $n,n-1,\ldots,n-\ell+1$ in decreasing order, contradiction. 
\end{proof}

\begin{remark}
\label{rmk:comp factors}
Note that if a simple $L_c(\mu)$ appears as a composition factor inside a standard module $\Delta_c(\mu')$ then all weights $\wt(0, T)$ have to appear as weights of $\Delta_c(\mu')$, where $T$ is a standard Young tableau on $\mu$. Thus, the proof of Lemma \ref{lem:far hooks} shows that the only composition factors of $\Delta_c(\mu_{j})$ can be $L_c(\mu_{j})$ and $L_c(\mu_{j+1})$. 

Moreover, the multiplicity of $L_c(\mu)$ as a composition factor of $\Delta_c(\mu')$ is bounded above by the dimension of the (generalized) weight space $\Delta_c(\mu')_{\wt(0, T)}$ where $T$ is any standard Young tableau on $\mu$. Thus, $[\Delta_c(\mu_{j}):L_c(\mu_{j+1})] \leq 1$. We will see in the next proposition that this multiplicity is always equal to $1$.
\end{remark}

\begin{proposition}
\label{prop:bgg}
For $\ell = 1, \dots, n-1$, the 
space $\Hom_{H_{c}}(\Delta_c(\mu_{\ell}), \Delta_c(\mu_{\ell - 1}))$ is $1$-dimensional. Up to a nonzero scalar, the unique homomorphism $\phi_{\ell}: \Delta_c(\mu_{\ell}) \to \Delta_c(\mu_{\ell - 1})$ is determined by $\phi_{\ell}(v(0, T_{i_{1} < \dots < i_{\ell}})) = v(me_{i_{\ell}}, T_{i_{1} < \cdots < i_{\ell - 1}})$. 
\end{proposition}

To prove Proposition \ref{prop:bgg} we will construct the maps $\phi_{\ell}$ using the results in Section \ref{sec-mackey}.


\begin{lemma}
\label{lemma- hook one}
Let $\Wm = \tau (\sigma_{n-1} \cdots \sigma_2 \sigma_1 \tau)^{m-1}$.
Then for $1<i<n$ $\sigma_i \Wm = \Wm \sigma_{i-1}$ 
and $u_i \Wm = \Wm u_{i-1}$, but 
$u_1 \Wm = \Wm (u_{n} +mt)$.
\end{lemma}
The proof is an easy computation we leave to the reader.
Recall for $\mathbf{e}_1 = (1,0,\ldots,0)$ that 
$\Hu{e_1} = H_1(\mathbf{u}) \otimes H_{n-1}(\mathbf{u})$.


\begin{lemma}
\label{lemma- hook two}
Let $U\subseteq \Vnlo$
be the $\Sk{n-1} \times \Sk{1}$-submodule
spanned by all $v_T$ where $T = T_{i_1< i_2 < \cdots < i_{\ell-1}}$
with $i_{\ell-1} \neq n$. In particular $U \simeq \Vnolo$
as an $\Sk{n-1}$-module.
\begin{enumerate}
\item
Let $t,c$ be such that $\Delta_{t,c}(n-\ell+1, 1^{\ell-1})$
is $\A$-semisimple.
Then we have $\Hu{e_1} \Wm U \subseteq \Delta_{t,c}(n-\ell+1, 1^{\ell-1})$
is an $\Hu{e_1}$-submodule which is isomorphic to 
$\Vnolo$ as an $H_{n-1}(\mathbf{u})$-module on which
$u_1$ acts identically as $c(\ell-n)+mt$. 
\item
In the case $t=1, c =\frac mn, \gcd(m,n)=1$, then $u_1$
acts as $c \ell$ and
$\Hnu \Wm U \simeq \Ind_{\Hu{e_1}}^{\Hnu} (c \ell) \boxtimes
\Vnolo$.
\end{enumerate}
\end{lemma}

\begin{proof}
The first statement follows from
Lemma \ref{lemma- hook one}.
Since $\Delta_{t,c}(n-\ell+1, 1^{\ell-1})$ is $\A$-semisimple
the $\sigma_i$
act triangularly with respect to the $s_i$. 
So the action of the $\sigma_i$ on the inflation via $\ev$
of an $\Sk{n-1}$-module completely determines the $\Sk{n-1}$
structure. 
Recall that via $\ev$ the $u_i$ will act as Jucys-Murphy operators.
\\
For the second statement, we use Lemma \ref{lemma- hook one}
to determine the action of $u_1$. Because $\FXi(\wb{m \mathbf{e}_n})
= \tau (s_{n-1} \cdots s_2 s_1 \tau)^{m-1} $ is a minimal length
double coset representative we get the second statement.
\end{proof}

\begin{lemma}
\label{lemma- hook induced}
$\Ind_{\Hu{e_1}}^{\Hnu} (c \ell) \boxtimes \Vnolo$
has an $\Hnu$-submodule isomorphic to $\Vnl$
(inflated along $\ev$). In particular $u_1$ is identically zero
on this submodule. 
\end{lemma}
The proof is a standard result for the degenerate affine Hecke algebra.

\begin{lemma}
\label{lemma- u0 y0}
Let $M =\Delta_{t,c}(V)$ be an $\RCA$-module which has a
 $\Hnu$-submodule $N$ on which $u_1$ acts identically as zero.
Then for $1 \le i \le n$ the $y_i$  act as zero on $N$.
\end{lemma}
\begin{proof}
Recall $u_1 =x_1 y_1$. 
Since $\Delta_{t,c}(V)$ is free as a
 $\C[x_1, \ldots, x_n]$-module,
$x_1$ has no torsion so in particular $y_1$ is zero on $N$.
As $N$ is $\Sn$-invariant and $y_i = (1,2,\ldots, i ) y_1 (i,\ldots,
2,1)$, all the $y_i$ must act as zero.
\end{proof}

\begin{proof}[Proof of Proposition \ref{prop:bgg}]As a consequence of Lemma \ref{lemma- u0 y0} we get that $\Delta(n-\ell+1, 1^{\ell-1})$
has a $\Sn$-submodule
isomorphic to  $\Vnl$ on which all $y_i$ vanish. 
Thus Frobenius Reciprocity gives us a nonzero $\RCA$ homomorphism
$$\Delta_c(n-\ell, 1^{\ell}) \xrightarrow{\phi_\ell} 
\Delta_c(n-\ell+1, 1^{\ell-1}).$$
This yields a proof of
Proposition \ref{prop:bgg}.
\end{proof}

More concretely, we can normalize the basis $\{ v_T \mid T \in
\SYT(\mu_\ell)\} $  of $\Vnlo$ so that we 
fix $v_{\mathbf{T}}$ for  $\mathbf{T} = T_{2 < 3 < \cdots < \ell+1}$
and take the other basis vectors to be
$\sigma_\ww v_{\mathbf{T}} =: v_{\ww \cdot \mathbf{T}}$
for $ \id \le \ww \le 
[1, n-\ell+1, \ldots, n-1, n, 2,3, \ldots, n-\ell]$ in
weak Bruhat order. (Recall as the $\sigma_i$ satisfy the braid
relations, $\sigma_\ww$ makes sense.)
Then $\phi_\ell$ is determined by
$$v_{\mathbf{T}} \mapsto \sigma_\ell \cdots \sigma_2 \sigma_1 \Wm
v_{T_{2 < 3< \cdots < \ell}}$$
where all tableau  on the left of $\mapsto$
 have shape $(n-\ell, 1^{\ell})$
but all tableau  on the right have shape $(n-\ell+1, 1^{\ell-1})$
More generally (noting $i_{\ell-1} \ge \ell$)  we have
$$v_{{T_{i_1 < i_2 < \cdots < i_\ell}}} \mapsto \sigma_{i_{\ell-1}}
 \cdots \sigma_2 \sigma_1 \Wm
v_{T_{i_1 < i_2< \cdots < i_{\ell-1}}}.$$
In particular when $i_\ell = n$ we get
$$v_{{T_{i_1 < i_2 < \cdots < n}}} \mapsto 
(\sigma_{n-1} \cdots \sigma_2 \sigma_1 \tau)^m 
v_{T_{i_1 < i_2< \cdots < i_{\ell-1}}}.$$
Recall 
$(s_{n-1} \cdots s_2 s_1 \tau)^m  = \tav{m \mathbf{e}_n}$.
One can easily check the above vectors' $u$-weights are
preserved by $\phi_\ell$.
It is only slightly more work to check with the above
assignment that $\phi_\ell$ intertwines the $\sigma_i$
acting on the $v_T, T \in \SYT(n-\ell, 1^{\ell})$.

\begin{corollary}
For any $\ell = 0, \dots, n-1$, the standard module $\Delta_c(\mu_{\ell})$ has a unique composition series $0 \subseteq I_{\ell} \subseteq \Delta_c(\mu_{\ell})$. Moreover, $I_{\ell} \cong L_c(\mu_{\ell +1})$ and $\Delta_c(\mu_{\ell})/I_{\ell} = L_c(\mu_{\ell})$.
\end{corollary}
\begin{proof}
From Remark \ref{rmk:comp factors} and Proposition \ref{prop:bgg} it follows that 

$$
[\Delta_c(\mu_{\ell}):L_c(\mu)] = \begin{cases} 1 & \mu = \mu_{\ell}, \mu_{\ell + 1} \\ 0 & \text{otherwise}. \end{cases}
$$

\noindent moreover, $L_c(\mu_{\ell+1})$ cannot appear as a quotient of $\Delta_c(\mu_{\ell})$. So defining $I_{\ell} := \phi_{\ell+1}(\Delta_c(\mu_{\ell + 1}))$ the result follows.
\end{proof}

\begin{corollary}
We have $\operatorname{im}(\phi_{\ell + 1}) = \ker(\phi_{\ell})$. In other words, the complex $\Delta_c(\mu_{\ell+1}) \buildrel \phi_{\ell + 1} \over \longrightarrow \Delta_c(\mu_{\ell})$ is exact outside of degree $0$ and coincides with \eqref{eqn:resolution}.
\end{corollary}
\begin{proof}
It is enough to see that $\ker(\phi_{\ell}) = I_{\ell}$. For this, it is enough to see that $\phi_{\ell}$ is neither zero nor injective. That it is nonzero is obvious. Thanks to Lemma \ref{lem:far hooks} we must have $\phi_{\ell+1}\circ \phi_{\ell} = 0$. So $\phi_{\ell}(I_{\ell}) = 0$ and $\phi_{\ell}$ is not injective.
\end{proof}




\subsection{Weight basis of simples} 

We continue assuming $c = m/n$ with $m$ and $n$ coprime positive integers. In this section,we describe weights belonging to the maximal proper submodule of every standard module $\Delta_c(\mu)$. Thanks to Corollary \ref{cor:not hook}, this question is only interesting when $\mu = \mu_{\ell}$ is a hook partition. Moreover, since $\Delta_c(\mu_{n-1})$ is simple, we may and will assume throughout this section that $0 \leq \ell < n-1$.  

\begin{lemma}\label{lemma:same weight}
Let $(\ba, T) \in \Z_{\geq 0}^{n} \times \SYT(\mu_{\ell})$. Then, there exists $(\bb, T') \in \Z_{\geq 0}^{n} \times \SYT(\mu_{\ell + 1})$ such that $\wt(\ba, T) = \wt(\bb, T')$ if and only if either

\begin{itemize}
\item $a_{g_{\ba}^{-1}(n)} - m > a_{g_{\ba}^{-1}(i_{\ell})}$ or 
\item $a_{g_{\ba}^{-1}(n)} - m = a_{g_{\ba}^{-1}(i_{\ell})}$ and $g_{\ba}^{-1}(n) > g_{\ba}^{-1}(i_{\ell})$
\end{itemize}

\noindent where $i_{\ell}$ is the number labeling the box with smallest content of $\mu_{\ell}$ on the tableau $T$. Moreover, if this is the case, then $(\bb, T')$ is uniquely determined.
\end{lemma}
\begin{proof}
Following the notation of Section \ref{sect:BGG}, let us denote $T = T_{i_{1} < \cdots < i_{\ell}}$. We will, first, see that there is a unique  $\bb \in \Z^{n}$ (possibly with negative entries) and $T'$  a tableau on $\mu_{\ell+1}$ (possibly non-standard) such that $\wt(\ba, T) = \wt(\bb, T')$. Indeed, if such pair $(\bb, T')$ exists we must have 
$$
n(a_{i} - b_{i}) = m(\ct{T}{g_{\ba}(i)} - \ct{T'}{g_{\bb}(i)})
$$

\noindent for every $i = 1, \dots, n$. Since $m$ and $n$ are coprime and $\mu_{\ell}, \mu_{\ell + 1}$ are adjacent hooks, we must have that either
\begin{itemize}
\item[(i)] $a_{i} = b_{i}$ and $T_{g_{\ba}(i)} = T'_{g_{\bb}(i)}$ (meaning that this box is in $\mu_{\ell} \cap \mu_{\ell + 1}$) or
\item[(ii)] $a_{i} - b_{i} = m$, $T_{g_{\ba}(i)}$ is the box of highest content in $\mu_{\ell}$, and $T'_{g_{\bb}(i)}$ is the box of lowest content in $\mu_{\ell + 1}$. 
\end{itemize}

Let $k \in \{1, \dots, n\}$ be such that $T_{k}$ is the box of highest content of $\mu_{\ell}$. From (i) and (ii), the vector $b$ is uniquely specified: $b_{i} = a_{i}$ if $i \neq g_{\ba}^{-1}(k)$, and $b_{g_{\ba}^{-1}(k)} = a_{g_{\ba}^{-1}(k)} - m$. Moreover, the tableau $T'$ is also uniquely specified: $T'_{g_{\bb}(i)} = T_{g_{\ba}(i)}$ if $i \neq g_{\ba}^{-1}(k)$, and $T'_{g_{\bb}(g_{\ba}^{-1}(k))}$ is the box with lowest content in $\mu_{\ell + 1}$. Our job now is to check that all coordinates of $\bb$ are non-negative and $T'$ is standard if and only if the conditions of the lemma are satisfied. Clearly, $\bb$ is non-negative if and only if $a_{g_{\ba}^{-1}(k)} \geq m$, so we will focus on the condition that $T'$ is standard. 

Let us first verify that $T'$ is standard on $\mu_{\ell} \cap \mu_{\ell + 1}$. Indeed, consider two consecutive boxes in $\mu_{\ell} \cap \mu_{\ell + 1}$ and let $j_{1} < j_{2}$ be their labels under $T$. Note that $j_{1}, j_{2} \neq k$. By definition of $g_{\ba}$ and $\bb$ we have
$$
b_{g_{\ba}^{-1}(j_{1})} = a_{g_{\ba}^{-1}(j_{1})} \leq a_{g_{\ba}^{-1}(j_{2})} = b_{g_{\ba}^{-1}(j_{2})}
$$
\noindent and, if we have an equality, $g_{\ba}^{-1}(j_{1}) < g_{\ba}^{-1}(j_{2})$. From the definition of $g_{\bb}$ it follows that $g_{\bb}g_{\ba}^{-1}(j_{1}) < g_{\bb}g_{\ba}^{-1}(j_{2})$, as wanted. 

So $T'$ is standard if and only if $g_{\bb}g_{\ba}^{-1}(i_{\ell}) < g_{\bb}g_{\ba}^{-1}(k)$. If $i_{\ell} = n$, we have $b_{g_{\ba}^{-1}(i_{\ell})} = a_{g_{\ba}^{-1}(n)} > a_{g_{\ba}^{-1}(k)} - m = b_{g_{\ba}^{-1}(k)}$ and therefore $g_{\bb}g_{\ba}^{-1}(i_{\ell}) > g_{\bb}g_{\ba}^{-1}(k)$. Thus, we must have $k = n$ and $i_{\ell} < n$. It follows now that the tableau $T'$ is standard if and only if either $b_{g_{\ba}^{-1}(n)} > b_{g_{\ba}^{-1}(i_{\ell})}$ or $b_{g_{\ba}^{-1}(n)} = b_{g_{\ba}^{-1}(i_{\ell})}$ and $g_{\ba}^{-1}(n) > g_{\ba}^{-1}(i_{\ell})$, which translates precisely into the conditions of the statement of the lemma. Finally, note that $a_{g_{\ba}^{-1}(n)} - m \geq a_{g_{\ba}^{-1}(i_{\ell})}$ automatically implies $a_{g_{\ba}^{-1}(n)} - m \geq 0$. We are done. 
\end{proof}

\begin{remark} Note that for $\ell = 0$ there is a unique tableau $T$
on $\mu_{0}$ and $i_{\ell} = 1$. In this case, $a_{g_{\ba}^{-1}(1)} =
\min\ba$ and $a_{g_{\ba}^{-1}(n)} = \max \ba$, so we recover the
conditions defining the set $\cS$ in Section \ref{sect:radical triv}. 
 \end{remark}
\begin{corollary}
Let $(\ba, T) \in \Z_{\geq 0}^{n} \times \SYT(\mu_{\ell})$. Then, $v(\ba, T) \in I_{\ell}$ if and only if there exists $(\bb, T') \in \Z^{n}_{\geq 0} \times \SYT(\mu_{\ell+1})$ such that $\wt(\ba, T) = \wt(\bb, T')$.
\end{corollary}
\begin{proof}
Since $I_{\ell} = \phi_{\ell + 1}(\Delta_c(\mu_{\ell + 1}))$, the necessity is clear. For sufficiency, assume that such $(\bb, T')$ exists. It is enough to see that $v(\bb, T') \not\in I_{\ell + 1}$ and to see this we can check that there does not exist $(\bd, T'') \in \Z_{\geq 0}^{n} \times \SYT(\mu_{\ell + 2})$ such that $\wt(\bb, T') = \wt(\bd, T'')$. So we have to check that $(\bb, T')$ does not satisfy the conditions of Lemma \ref{lemma:same weight}. Let $i_{\ell + 1} := g_{\bb}g_{\ba}^{-1}(n)$, and note that $T'_{i_{\ell + 1}}$ is precisely the box with lowest content in $\mu_{\ell + 1}$. Now,
$$
b_{g_{\bb}^{-1}(n)} - m = a_{g_{\bb}^{-1}(n)} - m \leq a_{g_{\ba}^{-1}(n)} - m = b_{g_{\ba}^{-1}(n)} = b_{g_{\bb}^{-1}(i_{\ell + 1})}
$$

If the inequality is strict, we are done. Else, we need to show that $g_{\bb}^{-1}(n) < g_{\bb}^{-1}(i_{\ell + 1}) = g_{\ba}^{-1}(n)$. But in this case we have $a_{g_{\bb}^{-1}(n)} = a_{g_{\ba}^{-1}(n)}$ and the result now follows by the definition of $g_{\ba}$. 
\end{proof}

\begin{corollary}\label{cor:weights of simples}
Assume $0 \leq \ell < n-1$ and let $(\ba, T) \in \Z_{\geq 0}^{n} \times \SYT(\mu_{\ell})$. Let us denote by $i_{\ell}$ the label of the box with smallest content of $\mu_{\ell}$ under $T$. Then, $v(\ba, T) \in I_{\ell}$ if and only if either
\begin{itemize}
\item $a_{g_{\ba}^{-1}(n)} - m > a_{g_{\ba}^{-1}(i_{\ell})}$ or 
\item $a_{g_{\ba}^{-1}(n)} - m = a_{g_{\ba}^{-1}(i_{\ell})}$ and $g_{\ba}^{-1}(n) > g_{\ba}^{-1}(i_{\ell})$
\end{itemize}

It follows that  $L_c(\mu_{\ell}) = \Delta_c(\mu_{\ell})/I_{\ell}$ has a weight basis indexed by pairs $(\ba, T) \in \Z_{\geq 0} \times \SYT(\mu_{\ell})$ such that

\begin{itemize}
\item[$\circ$] $a_{g_{\ba}^{-1}(n)} - m < a_{g_{\ba}^{-1}(i_{\ell})}$ or
\item[$\circ$] $a_{g_{\ba}^{-1}(n)} - m = a_{g_{\ba}^{-1}(i_{\ell})}$ and $g_{\ba}^{-1}(n) < g_{\ba}^{-1}(i_{\ell})$.
\end{itemize}
\end{corollary}

\begin{remark}
Note that, if $0 < \ell < n-1$, then $L_c(\mu_{\ell}) \cong I_{\ell - 1}$. Thus, there is a weight-preserving bijection between pairs $(\ba, T) \in \Z_{\geq 0}^{n} \times \SYT(\mu_{\ell - 1})$ satisfying the condition marked with $\bullet$ in Corollary \ref{cor:weights of simples} and those pairs $(\bb, T') \in \Z_{\geq 0}^{n} \times \SYT(\mu_{\ell})$ satisfying the conditions market with $\circ$. This bijection is described in the proof of Lemma \ref{lemma:same weight}. 
\end{remark}

\section{The polynomial representation of \texorpdfstring{$H_{c}$}{Hc}}
\label{sect:pol rep}

All of our geometric applications deal only with the representation $\Delta_{c}(\triv)$ where, as in Section \ref{sec:rep theory}, we have $c = m/n > 0$ with $\gcd(m;n)$ and $\triv$ is the trivial representation of $\Sk{n}$. Thus, we specialize many of the results of Section \ref{sec:rep theory} to this special case. Note that as as a $\Hnx$-module we have $\Delta_c(\triv) = \C[x_1, \dots, x_n]$, where the action of $x_i \in \Hnx$ is simply given by multiplication and the action of $\Sk{n} \subseteq \Hnx$ is given by permuting the variables. For this reason, $\Delta_c(\triv)$ is commonly known as the \emph{polynomial representation} of $H_c$. For the rest of this section, we assume $t = 1$. The case $t = 0$ will be trated in Section \ref{sect:m to infinity}.

\subsection{The action of the Dunkl-Opdam subalgebra} 

Following Definition \ref{def: weight labels} in the case where $\mu$ is the trivial partition, for $\ba \in \Z_{\geq 0}^{n}$ we let $\tw := \tw(\ba) \in \C^{n}$
denote the weight whose $i$th component is $\tw_{i} = a_{i} -
(g_{\ba}(i) - 1)c$. 
In other words, $\tw(\ba) = \wa{a} \cdot (0, -c, -2c, \cdots, (1-n)c)$ where, as mentioned above, we specialize $t = 1$. 
Now we are ready to describe the spectrum of $\A$
on $\Delta_{c}(\triv)$, in the case where $c$ is generic. The proof of the following result is similar to that of Theorem \ref{thm:vermas} so we omit it. 

\begin{proposition}\label{prop:generic}
Assume that either $c \in \mathbb{C} \setminus \mathbb{Q}$ or that $c$ is a rational number with denominator greater than $n$. Let  $M = \Delta_{c}(\triv)$. Then, $M_{\tw(\ba)}^{\gen} \neq 0$ for every $\ba \in \Z_{\geq 0}^{n}$, these are all the weight spaces of $\A$ on $M$, and each one of them is $1$-dimensional (so that $M_{\tw(\ba)}^{\gen} = M_{\tw(\ba)}$.)
\end{proposition}

The following result is obtained by specializing Theorem \ref{thm:vermas} to the case where $\mu$ is the trivial partition of $n$. 

\begin{theorem}\label{thm:1}
	Let $c = m/n > 0$ with $\gcd(m,n) = 1$ and $M =
\Delta(\triv)$. For any $\ba \in \Z_{\geq 0}^{n}$, $M_{\wta} =
M_{\wta}^{\gen} \neq 0$. Moreover, if $a_{i} > a_{i+1}$, then
$\sigma_{i}|_{M_{\wta}} \neq 0$. 
\end{theorem}

\begin{remark}\label{rmk:basis}
Moreover, it is easy to show that for every element $\ba \in \Z^{n}_{\geq 0}$, there exists a unique element $v_{\ba} \in M_{\wta}$ of the form

$$
v_{\ba} = x^{\ba} + \sum_{\ba' \prec \ba}k_{\ba, \ba'}x^{\ba'}
$$

\noindent where $\prec$ is the partial order on monomials defined in Section \ref{sect:combinatorics}.
\end{remark}

\begin{remark}\label{rmk:sln}
We can get an analogous result for the $\sln$ RCA, with the operators defined as in \cite{gicheol}. In particular, we can find a basis $\{v_{\ba} : \ba \in \mathbb{Z}^{n-1}_{\geq 0}\}$ of simultaneous eigenvectors for the $\sln$-Dunkl-Opdam operators.

Note that the action of $\tau$ is injective on $\Delta_c(\triv)$. Combinatorially, the action of $\pi$ on the set of nonnegative sequences $(a_1,\ldots,a_n)$ is injective, and any such sequence can be uniquely written as 
$$
(a_1,\ldots,a_n)=\pi^k\cdot (0,b_2,\ldots,b_{n-1}).
$$

\noindent where $(0, b_{1}, \dots, b_{n-1}) \in \Z_{\geq 0}^{n}$. 
That is, the generating set for this action consists of sequences with $a_1=0$, and it is in bijection with the basis in the polynomial representation of $\sln$ RCA.  
\end{remark}


\subsection{Recovering the action of \texorpdfstring{$H_{c}$}{Hc}} \label{sect:renormalization}
We keep assuming that $c = m/n > 0$ with $\gcd(m,n) = 1$ and $t = 1$. The following theorem is a consequence of Theorem \ref{thm:action standard}.

\begin{theorem}\label{thm:action}
The module $\Delta_c(\triv)$ has a basis given by $\{v_{\ba} : \ba \in \Z^n_{\geq 0}\}$, and the action of the algebra $H_{c}$ on $\Delta_c(\triv)$ is given by the following operators:

\begin{align*}
u_{i}v_{\ba} & = \tw_{i}v_{\ba}  \\
\tau v_{\ba} &  = v_{\pi\cdot\ba} \\
\lambda v_{\ba} & = \tw_{1}v_{\pi^{-1}\cdot\ba} \\
s_{i}v_{\ba} &  = \begin{cases} v_{s_{i}\cdot \ba} + \frac{c}{\tw_{i} - \tw_{i+1}}v_{\ba} & a_{i} > a_{i+1}, \\ 
\frac{(\tw_{i+1} - \tw_{i} - c)(\tw_{i+1} - \wt_{i} + c)}{(\tw_{i} - \tw_{i+1})^{2}}v_{s_{i}\cdot \ba} + \frac{c}{\tw_{i} - \tw_{i+1}}v_{\ba} & a_{i} < a_{i+1}, \\ 
v_{\ba} & a_{i} = a_{i+1} \end{cases}
\end{align*}

\noindent where we denote $\tw_{i} = a_{i} - (g_{\ba}(i) - 1)c$. 
\end{theorem}


For geometric applications, we will need a different basis of $\Delta_c(\triv)$ that gives nicer formulas for the action of the operators $s_{i}$. This basis is a renormalization of the basis $v_{\ba}$, but we have to be careful with the renormalization factor.  The main result of this section is the following.

\begin{proposition}\label{prop:renormalization}
There exists a function $\varphi: \Z_{\geq 0}^{n} \to \C^{\times}$ such that, defining $\widetilde{v}_{\ba} := \varphi(\ba) v_{\ba}$, we have

$$
(1 - s_{i})\widetilde{v}_{\ba} = \frac{\tw_{i} - \tw_{i+1} - c}{\tw_{i} - \tw_{i+1}}\left(\widetilde{v}_{\ba} - \widetilde{v}_{s_{i}\cdot \ba}\right)
$$

\noindent for every $\ba \in \mathbb{Z}^{n}_{\geq 0}$ and $i = 1,
\dots, n-1$.
 \end{proposition}
 \begin{proof}
 We can define the $\widetilde{v}_{\ba}$
using the renormalized intertwiners $\widetilde{\sigma}_i$
of Remark \ref{remark: renormalize intertwiners}. 
In particular, the vectors $\widetilde{v}_{\ba} = \widetilde{\sigma}_{\wa{a}} \widetilde{v}_{0}$ for  $\wa{a} \in \Min$ are uniquely determined after specifying $\widetilde{v}_{0}$.   
Then the existence of the function $\varphi$ is obvious
and $\varphi(\ba)$ can be  given as a product formula
with terms indexed by the inversions of $\wa{a}$.
In particular
\begin{equation}\label{eqn:renorm}
\frac{\varphi(\ba)}{\varphi(s_{i}\cdot \ba)} = \begin{cases} \frac{\tw_{i} - \tw_{i+1}-c}{\tw_{i} - \tw_{i+1}} & a_{i} >  a_{i+1}, \\
\frac{\tw_{i+1} - \tw_{i}}{\tw_{i+1} - \tw_{i}-c} & a_{i} < a_{i+1}.
\end{cases}
\end{equation}
The formula for the action of $1- s_i$ follows from Theorem \ref{thm:action}.
 \end{proof}

\begin{corollary}
The action of the operators $u_i, \tau$ and $\lambda$ in the renormalized basis $\widetilde{v}_{\ba}$
is given by the same equations as in Theorem \ref{thm:action}:
\begin{equation}
\label{eq: action renormalized}
u_{i}\widetilde{v}_{\ba}  = \tw_{i}\widetilde{v}_{\ba},\quad
\tau \widetilde{v}_{\ba}   = \widetilde{v}_{\pi\cdot\ba},\quad
\lambda \widetilde{v}_{\ba}  = \tw_{1}\widetilde{v}_{\pi^{-1}\cdot\ba}.
\end{equation}
\end{corollary}

\subsection{The radical of \texorpdfstring{$\Delta_c(\triv)$}{Delta(triv)}}\label{sect:radical triv} Finally, we will need an explicit weight basis of the simple quotient $L_{c}(\triv)$ of $\Delta_{c}(\triv)$. We define the set 

\begin{multline*}
\cS := \{(a_{1}, \dots, a_{n}) \in \Z^{n}_{\geq 0} \,  : \, 
\max(a_{i} - a_{j}) > m,  \\
\, \text{or} \, 
\max(a_{i} - a_{j}) = m \, \text{and} \, a_{i} - a_{j} = m \, \text{for some} \, j < i\}
\end{multline*}

The next result then follows from Corollary \ref{cor:weights of simples} applied to $\ell = 0$.

\begin{proposition}\label{prop:radical}
The space

$$
S := \bigoplus_{\ba \in \cS} \C v_{\ba}
$$

\noindent is an $H_{c}$-submodule of $\Delta_c(\triv)$, and it coincides with the unique proper submodule of $\Delta_c(\triv)$ \cite{ES}.
\end{proposition}

Now let $\cT := \Z_{\geq 0}^{n} \setminus \cS$. More explicitly, 
\begin{equation}
\cT=\{\ba\in   \Z^{n}_{\geq 0} : a_{i} - a_{j} \leq m\ \text{for every}\ i, j;\text{moreover, if}\ a_{i} - a_{j} = m\ \text{then}\ j > i\}.
\end{equation}

\begin{corollary}
\label{cor: basis in L(triv)}
The module $L_c(\triv)$ has a basis $\{v_{\ba} : \ba \in \cT\}$. The
action of $H_{c}$ on $L_c(\triv)$ is given by the same formulas as in
Theorem \ref{thm:action}, with the understanding that we set $v_{\ba} =
0$ if $\ba \not\in \cT$.
 \end{corollary}


\begin{remark}
\label{rem: simple basis sln}
Proposition \ref{prop:radical} can be easily adapted to the $\sln$-setting, cf. Remark \ref{rmk:sln}. In that case, we have $\cS^{\sln} := \{(a_{1}, \dots, a_{n-1}) \in \mathbb{Z}_{\geq 0}^{n-1} : \max(a_{i}) \geq m\}$. In particular, $\cT^{\sln} = \{(a_{1}, \dots, a_{n-1}) \in \mathbb{Z}^{n-1}_{\geq 0} : a_{i} < m \; \text{for every} \; i\}$ and we recover the formula $\dim(L^{\sln}_c(\triv)) = m^{n-1}$.  

As in Remark \ref{rmk:sln}, we can prove that the action of $\pi$ is injective on the basis in $\cT$.
The generating set for this action consists of sequences $(0,a_2,\ldots,a_n)$, and it is easy to see that such
sequence is in $\cT$ if and only if $a_i<m$ for all $i$.
\end{remark}

The following lemma provides an interpretation of the indexing set $\cT$ in terms of affine permutations.

\begin{lemma}
Let $\tw_i(\ba):=\tw(\ba)_i = a_i-\frac{m}{n}(g_{a}(i)-1)$ be the weights of $u_i$ as above. Consider the affine permutation 
$$
\omega:=[-n\wt_{1}(\ba),\ldots,-n\wt_{n}(\ba)]^{-1}.
$$
Then the following statements hold: 

(a) $(a_1,\ldots,a_n)\in \cT$ if and only if $\omega$ is $m$-stable.

(b) $a_i\ge 0$ for all $i$ if and only if $\omega\perm\in \Min$ , where $\perm=[0,m,\ldots,(n-1)m]$.
\end{lemma}

\begin{proof}
We have 
$$
\omega^{-1}(i)=-n\wt_{i}(\ba)=-na_{i}+m(g_{\ba}(i)-1),
$$
so by \eqref{eq: omega alpha beta} we have $\omega\perm=\ta{a}g_{\ba}^{-1}$. By Lemma \ref{lem: m stable in window notation} $\omega$ is $m$-stable if and only if $\ba\in \cT$. Finally, $a_i\ge 0$ for all $i$ if and only if $\ta{a}g_{\ba}^{-1}\in \Min$.
\end{proof}


\begin{remark}\label{rmk:GMV}
The action of $\pi$ on $(a_1,\ldots,a_n)$ corresponds to the conjugation of $\omega_a$ by $\pi\in \affSn$
which effectively slides the window in $\omega_{\ba}$. Remark \ref{rem: simple basis sln} gives a choice of a representative in each $\pi$-orbit with $a_1=0$ and $m>a_i\ge 0$ for $i>1$. 

From the viewpoint of affine permutations, a more natural choice of a representative is given by the balancing condition $\sum_{i=1}^{n}\omega(i)=\frac{n(n+1)}{2}$. The corresponding permutations will be still $m$-stable,
and by Remark \ref{rem: balancing} they are in bijection with the alcoves insider the $m$-dilated fundamental alcove. 

Therefore we get an explicit bijection between the alcoves insider the $m$-dilated fundamental alcove,
$m$-stable balanced affine permutations, and vectors of the form
$(0,a_2,\ldots,a_n)$ with $0\le a_i<m$. 
\end{remark}

\section{Singular curves}\label{sect:geometry}

For coprime $m,n \ge 1$ we consider the plane curve singularity $C=\{x^m=y^n\}$
at the origin. It has an action of $\C^*$ given by $(x,y)\mapsto (s^nx,s^my)$.
This action extends to the local ring of functions on $C$ which is isomorphic to
$\cO_C=\C[[x,y]]/(x^m-y^n)$. A homogeneous basis in $\cO_C$ can be described as follows:
\begin{equation}
\label{eq: basis}
\cO_C=\C[[x]]\langle 1,\ldots,y^{n-1}\rangle
\end{equation}
This presentation shows that $\cO_C$ is a free module over $\C[[x]]$ of rank $n$, and the multiplication by $y$ is given by the matrix 
\begin{equation}\label{eq:Y}
Y=\left(
\begin{matrix}
0 & 0 & \cdots & 0 & x^m\\
1 & 0 & \cdots & 0 & 0 \\
0 & 1 & \cdots & 0 & 0\\
\vdots & \vdots& \ddots & \vdots & 0\\
0 & 0 & \cdots & 1& 0 \\
\end{matrix}
\right).
\end{equation}

\subsection{Hilbert schemes on singular curves}\label{sect: hilb sch sing}

By definition, the Hilbert scheme of $k$ points on $C$ is the moduli space of co-dimension $k$ ideals 
$$
\Hilbk(C)=\{I\subset \cO_C: I\ \text{ideal},\dim \cO_C/I=k\}.
$$
The action of $\C^*$ on $C$ extends to an
action on $\Hilbk(C)$ for all $k$. The fixed points of this action are monomial ideals. In terms of the identification \eqref{eq: basis} such 
an 
ideal is generated over $\C[[x]]$ by monomials of the form $\langle x^{c_1},yx^{c_2}, \ldots,y^{n-1}x^{c_n}\rangle$. Since it is invariant under the multiplication of $y$ (or the matrix $Y$ above),
we get a system of inequalities
\begin{equation}
\label{eq: c sorted}
c_1\ge c_2\ge \ldots \ge c_n\ge c_1-m.
\end{equation}
Note that $\dim \cO_C/I=k=\sum c_i$. 
In the notation of \cite{OS}, such ideals can be represented by staircases of height $n$ and width at most $m$.
See Figure 
\ref{fig: ideal}.

\begin{figure}
\begin{tikzpicture}
\filldraw[color=lightgray] (10,3)--(3,3)--(3,2)--(5,2)--(5,1)--(7,1)--(7,0)--(10,0)--(10,3);

\draw(0,0)--(10,0);
\draw(0,0)--(0,3)--(10,3);
\draw (0,1)--(10,1);
\draw (0,2)--(10,2);
\draw (1,0)--(1,3);
\draw (2,0)--(2,3);
\draw (3,0)--(3,3);
\draw (4,0)--(4,3);
\draw (5,0)--(5,3);
\draw (6,0)--(6,3);
\draw (7,0)--(7,3);
\draw (8,0)--(8,3);
\draw (0.5,0.5) node {$1$};
\draw (0.5,1.5) node {$y$};
\draw (0.5,2.5) node {$y^2$};
\draw (1.5,0.5) node {$x$};
\draw (1.5,1.5) node {$xy$};
\draw (1.5,2.5) node {$xy^2$};
\draw (2.5,0.5) node {$x^2$};
\draw (2.5,1.5) node {$x^2y$};
\draw (2.5,2.5) node {$x^2y^2$};
\draw (3.5,0.5) node {$x^3$};
\draw (3.5,1.5) node {$x^3y$};
\draw (3.5,2.5) node {$x^3y^2$};
\draw (4.5,0.5) node {$x^4$};
\draw (4.5,1.5) node {$x^4y$};
\draw (4.5,2.5) node {$x^4y^2$};
\draw (5.5,0.5) node {$x^5$};
\draw (5.5,1.5) node {$x^5y$};
\draw (5.5,2.5) node {$x^5y^2$};
\draw (6.5,0.5) node {$x^6$};
\draw (6.5,1.5) node {$x^6y$};
\draw (6.5,2.5) node {$x^6y^2$};
\draw (7.5,0.5) node {$x^7$};
\draw (7.5,1.5) node {$x^7y$};
\draw (7.5,2.5) node {$x^7y^2$};

\draw[line width=2] (3,3)--(3,2)--(5,2)--(5,1)--(7,1)--(7,0);

\end{tikzpicture}
\caption{An ideal on $C=\{x^4=y^3\}$ generated by $x^3y^2$ and $x^5y$. \newline
Note that $y\cdot x^3y^2=x^7$. The codimension of the ideal is $15=7+5+3 = c_1+c_2+c_3$ which is also the number of boxes under the staircase.
\label{fig: ideal}
}
\end{figure}

\begin{lemma}
\label{lem: sorting}
Suppose that $I\subset \cO_C$ is spanned over $\C[[x]]$ by $y^{\alpha_1}x^{c_1},\ldots, y^{\alpha_n}x^{c_n}$
where $\{\alpha_1,\ldots,\alpha_n\}=\{0,\ldots,n-1\}$. Then the following holds:
\begin{itemize}
\item[(a)] If $\max(c_i)-\min(c_i)>m$ then $I$ is not an ideal in $\cO_C$ for any choice of $\alpha_i$.
\item[(b)] If $\max(c_i)-\min(c_i)\le m$ then there exists a unique ideal $I$ of this form.
\end{itemize}
\end{lemma}

\begin{proof}
Assume that $I=\C[[x]]\langle y^{\alpha_1}x^{c_1},\ldots y^{\alpha_n}x^{c_n}\rangle$ is an ideal in $\cO_C$.
Let $\widetilde{g}\in \Sn$ be the 
permutation in $\Sn$ which sorts the $\alpha_i$ in increasing order.
Then by \eqref{eq: c sorted} 
$c_{\widetilde{g}^{-1}(1)}\ge c_{\widetilde{g}^{-1}(2)}\ge \ldots \ge
c_{\widetilde{g}^{-1}(n)}$.
Observe that $c_{\widetilde{g}^{-1}(1)}=\max(c_i)$ and $c_{\widetilde{g}^{-1}(n)}=\min (c_i)$.
Therefore the condition $c_{\widetilde{g}^{-1}(n)}\ge c_{\widetilde{g}^{-1}(1)}-m$ in \eqref{eq: c sorted} is equivalent to $\max(c_i)-\min(c_i)\le m$.

For part (b), the uniqueness is clear.
\end{proof}

Let $I=\C[[x]]\langle x^{c_1},yx^{c_2},\ldots,y^{n-1}x^{c_n}\rangle$ be a monomial ideal in $\cO_C$ where $c_i$ satisfy \eqref{eq: c sorted}. We define
a composition $\widetilde{\lambda}=(\widetilde{\lambda_1},\ldots,\widetilde{\lambda_\ell}),\sum \widetilde{\lambda_i}=n$ by looking at vertical runs of the staircase defined by $c_i$:
$$
c_1=\ldots=c_{\widetilde{\lambda_1}}>c_{\widetilde{\lambda_1}+1}=\ldots=c_{\widetilde{\lambda_1}+\widetilde{\lambda_2}}>\ldots >c_{n-\widetilde{\lambda_\ell}+1}=\ldots=c_n
$$
We also define a composition $\lambda$ as follows:
$$
\lambda=\begin{cases}
\widetilde{\lambda} & \ c_1-c_n<m,\\
(\widetilde{\lambda_1}+\widetilde{\lambda_{\ell}},\widetilde{\lambda_2},\ldots,\widetilde{\lambda_{\ell-1}}) &  c_1-c_n=m.\\
\end{cases}
$$

\begin{lemma}
\label{lem: Jordan blocks}
The operator $Y$ acting on the space $I/xI$ has Jordan blocks of sizes $\lambda_i$.
\end{lemma}

\begin{proof}
The space $I/xI$ is spanned (over $\C$) by $\langle v_1=x^{c_1},v_2=yx^{c_2},\ldots,v_n=y^{n-1}x^{c_n}\rangle$.
Clearly, if $c_1=c_2$ then $Y(v_1)=v_2$, otherwise $Y(v_1)$ vanishes in $I/xI$ . Similarly, we see chains of vectors
$$
v_1 \xrightarrow{Y} \ldots \xrightarrow{Y} v_{\widetilde{\lambda_1}} \xrightarrow{Y} 0,\qquad
v_{\widetilde{\lambda_1}+1}\xrightarrow{Y}\ldots\xrightarrow{Y}v_{\widetilde{\lambda_1}+\widetilde{\lambda_2}}\xrightarrow{Y}0,\qquad \ldots,$$
$$
\qquad
v_{n-\widetilde{\lambda_\ell}+1}\xrightarrow{Y}\ldots\xrightarrow{Y} v_n.
$$
Finally, $Y(v_n)=x^{c_n+m}$, so if $c_n+m>c_1$ then $Y(v_n)=0$, otherwise $c_n+m=c_1$ and $Y(v_n)=v_1$. 
\end{proof}

\subsection{Parabolic Hilbert schemes on singular curves}

Observe that for any ideal $I$ we have $\dim I/xI=n$. So we can define the parabolic Hilbert scheme as the moduli space of flags of ideals
$$
\PHilb_{k,n+k}(C):=\left\{\cO_C\supset I_k\supset I_{k+1}
\cdots\supset I_{k+n}=xI_k: I_s\ \text{ideal},\ \dim \cO_C/I_s=s\right\}
$$

\noindent and we define

$$
\PHilb^{x}(C) := \bigsqcup_{k \geq 0} \PHilb_{k, n+k}(C).
$$

Again, we would like to describe the fixed points of the $\C^*$ action on this variety explicitly, similarly to \cite{hikita}. These are described by flags where all $I_s$ are monomial ideals. As above, for $i=1,\ldots,n$ we can assume that the one-dimensional space $I_{k+i-1}/I_{k+i}$ is spanned by the monomial $y^{\alpha_i}x^{c_i}$ where $0\le \alpha_i\le n-1$.  In particular, 
$$
I_k=\C[[x]]\langle y^{\alpha_1}x^{c_1},\ldots y^{\alpha_n}x^{c_n}\rangle.
$$ 
Note that if $c_i = c_j$ and $i < j$ then $\alpha_i < \alpha_j$, so by Lemma \ref{lem: sorting} $\alpha_i$ are uniquely determined by $c_i$.

Furthermore, we can extend this construction 
by defining $I_{i+n}=xI_i$ for all integers $i\ge k$. 
Note that it follows that
$\alpha_{i+n}=\alpha_i$ and $c_{i+n}=c_i+1$ for $i\ge 1$. 

\begin{lemma}
\label{lem: fixed point description}
The vector $\bc=(c_1,\ldots,c_n)$ determines a fixed point in $\PHilb_{k,n+k}$ if and only if either of the two equivalent conditions hold:
\begin{itemize}
\item[(a)] For all $t>0$ one has
\begin{equation}
\label{eq: max min in window}
\max\nolimits_{i=t}^{t+n-1}(c_i) -\min\nolimits_{i=t}^{t+n-1}(c_i)\le m
\end{equation}
\item[(b)] One has $\max_{i=1}^{n}(c_i) -\min_{i=1}^{n}(c_i)\le m$ and whenever $c_j+m=c_i$
then $j<i$. 
\end{itemize}
\end{lemma}
\begin{proof}
By construction, for all $t>0$ the subspace $I_{k-1+t}$ is spanned over $\C[[x]]$ by the monomials 
$\langle y^{\alpha_t}x^{c_t},\ldots,y^{\alpha_t+n-1}x^{c_{t+n-1}}\rangle$, so  by Lemma \ref{lem: sorting}
it is an ideal if and only if \eqref{eq: max min in window} holds. This proves (a).

Now let us prove that (a) and (b) are equivalent. Indeed, the left hand side of \eqref{eq: max min in window} is $n$-periodic,
so it is sufficient to consider $t\le n$. Assume that $\max_{i=1}^{n}(c_i) -\min_{i=1}^{n}(c_i)\le m$, then \eqref{eq: max min in window} does not hold if and only if  there exists $i<t$ and $j\ge t$ such that
$c_{i}=c_j+m$ and $c_{i+n}=c_i+1=c_j+m+1$.
\end{proof}

We picture a fixed point in $\PHilb_{k, n+k}(C)$ as the  staircase representing the ideal $I_{k}$, together with an enumeration of the boxes bordering it to its right, in such a way that the quotient $\cO_{C}/I_{k+j}$ is spanned by the boxes under the staircase together with the boxes numbered $1, 2, ..., j$. See Figure \ref{fig: PHilb}.

\begin{remark}
The proof of Lemma \ref{lem: fixed point description} is very similar to the proof of Lemma \ref{lem: m stable in window notation}. Indeed, this is not a coincidence: let us parametrize the curve $C$ by $(x,y)=(z^n,z^m)$, then any monomial in $x$ and $y$ corresponds to a monomial in $z$. A monomial ideal in $\cO_C$ then corresponds to an $(m,n)$-invariant subset in $\Z_{\ge 0}$, and a flag of monomial ideals corresponds to a flag of $(m,n)$-invariant subsets. By Proposition \ref{prop: m stable} such flag determines an $m$-stable affine permutation. We conclude that fixed points on parabolic Hilbert scheme are in bijection with $m$-stable affine permutations $\omega$ such that $\omega\perm\in \Min$. 
\end{remark}

\begin{figure}
\begin{tikzpicture}
\filldraw[color=lightgray] (10,3)--(3,3)--(3,2)--(5,2)--(5,1)--(7,1)--(7,0)--(10,0)--(10,3);

\draw(0,0)--(10,0);
\draw(0,0)--(0,3)--(10,3);
\draw (0,1)--(10,1);
\draw (0,2)--(10,2);
\draw (1,0)--(1,3);
\draw (2,0)--(2,3);
\draw (3,0)--(3,3);
\draw (4,0)--(4,3);
\draw (5,0)--(5,3);
\draw (6,0)--(6,3);
\draw (7,0)--(7,3);
\draw (8,0)--(8,3);

\draw (0.5,0.5) node {$1$};
\draw (0.5,1.5) node {$y$};
\draw (0.5,2.5) node {$y^2$};
\draw (1.5,0.5) node {$x$};
\draw (1.5,1.5) node {$xy$};
\draw (1.5,2.5) node {$xy^2$};
\draw (2.5,0.5) node {$x^2$};
\draw (2.5,1.5) node {$x^2y$};
\draw (2.5,2.5) node {$x^2y^2$};
\draw (3.5,0.5) node {$x^3$};
\draw (3.5,1.5) node {$x^3y$};
\draw (4.5,0.5) node {$x^4$};
\draw (4.5,1.5) node {$x^4y$};
\draw (5.5,0.5) node {$x^5$};
\draw (6.5,0.5) node {$x^6$};

\draw (3.5,2.5) node { \Large \bf 1}; 
\draw (5.5,1.5) node {\Large \bf 3};
\draw (7.5,0.5) node {\Large \bf 2};

\draw[line width=2] (3,3)--(3,2)--(5,2)--(5,1)--(7,1)--(7,0);

\end{tikzpicture}
\caption{A flag of monomial ideals  
in $\PHilb_{15,15+3}(x^4=y^3)$: \newline $I_{15}=\langle x^3y^2, x^5y \rangle, I_{16}=\langle x^4y^2,x^5y,x^7\rangle, I_{17}=\langle x^4y^2,x^5y\rangle.$ \newline 
Here $y^{\alpha_1}x^{c_1}=x^3y^2,\ y^{\alpha_2}x^{c_2}=x^7,\ y^{\alpha_3}x^{c_3}=x^5y$. \label{fig: PHilb}}
\end{figure}

We define the line bundles $\cL_i$, $1\le i\le n$ on the parabolic flag Hilbert scheme as follows. 
The fiber of $\cL_i$ over the flag $I_k\supset I_{k+1}\supset \cdots\supset I_{k+n}=xI_k$ is $I_{k+i-1}/I_{k+i}$. Then we have the following:

\begin{lemma}
\label{lem: weights of line bundles}
There is a bijection between the eigenbasis  $v_{\ba}$ in $L_{m/n}(\triv)$ (defined in Corollary \ref{cor: basis in L(triv)}) and 
the set of $\C^*$ fixed points in $\PHilb^{x}(C)$. Under this bijection, the weight of $\cL_i$ at a fixed point corresponds to the eigenvalue $n\wt_{n+1-i}(\ba)+m(n-1)$ of the operator $nu_{n+1-i}+m(n-1)$ on $v_{\ba}$.
\end{lemma}

\begin{proof}
Recall that by Corollary \ref{cor: basis in L(triv)} the basis $v_{\ba}$ in $L_{m/n}(\triv)$ is parametrized by sequences of nonnegative integers $\ba=(a_1,\ldots,a_n)$ such that $a_{i} - a_{j} \leq m$ for every $i, j$, and if $a_{i} - a_{j} = m$ then $j > i$. The eigenvalues of $u_i$ are given by 
$\wt_{i} = a_{i} - (g_{\ba}(i) - 1)\frac{m}{n}$ where $g_{\ba}$ is the permutation which sorts $\ba$ in non-decreasing order (here we substituted $c=\frac{m}{n}$). 

On the other hand, the fixed points in $\PHilb_{k,n+k}$ are determined by sequences of monomials $(y^{\alpha_i}x^{c_i})$ where $\max_{i=1}^{n}(c_i) -\min_{i=1}^{n}(c_i)\le m$ and whenever $c_j+m=c_i$
then $j<i$.
We remark that since $y^{\alpha_i}x^{c_i}$ spans the quotient $I_{k + i - 1}/I_{k+i}$  it follows that when $c_i=c_j$ with $i<j$ we have $\alpha_i<\alpha_j$. 
Clearly, the assignment $c_i=a_{n+1-i}$ is a bijection intertwining the restrictions on $a_i$ and on $c_i$.
Note $k= \sum c_i = \sum a_i = ||\ba||$.

Finally, the line bundle $\cL_i$ has the equivariant weight $m\alpha_i+nc_i$.
We have $\alpha_i=\widetilde{g}(i)-1$, where $\widetilde{g}$ is the permutation
defined in the proof of Lemma \ref{lem: sorting}
which sorts the $\alpha_i$ in increasing order.
Clearly, $\widetilde{g}(i)=n+1-g_{\ba}(n+1-i)$, hence 
\begin{multline*}
m\alpha_i+nc_i=m(n+1-g_{\ba}(n+1-i)-1)+na_{n+1-i}=\\ m(n-1)+m(1-g_{\ba}(n+1-i))+na_{n+1-i}=
n\wt_{n+1-i}+m(n-1).
\end{multline*}
\end{proof}

\begin{example}
For $\ba=(0,\ldots,0)$ we get $\wt_i(\ba)=-(i-1)\frac{m}{n}$ while the corresponding fixed point in $\PHilb_{0,n}$ corresponds to the flag $\cO_{C} \supseteq y\cO_{C} \supseteq \cdots \supseteq y^{n-1}\cO_{C}$. The section of $\cL_i$ is given by monomial $y^{i-1}$ which has weight $m(i-1)$. Now 
$$
n\wt_{n+1-i}(\ba)+m(n-1)=-m(n+1-i-1)+m(n-1)=m(i-1).
$$
\end{example}

\subsection{Geometric operators}
\label{sec:geometric operators}

There is a natural projection $\pi:\PHilb_{k,n+k}\to \Hilbk$ which
sends a flag $I_k\supset I_{k+1}\supset \cdots\supset I_{k+n}=xI_k$ to
$I_k$.  The fibers of this projection are just the classical Springer
fibers consisting of complete flags in $I_k/xI_k$ invariant under the
action of $Y$. In particular, Lemma \ref{lem: Jordan blocks}
immediately implies the following.

\begin{lemma}\label{lemma:induced representation}
Given an ideal $I=\C[[x]]\langle y^{\alpha_1}x^{c_1},\ldots y^{\alpha_n}x^{c_n}\rangle$ in $\Hilbk$
there are $\binom{n}{\lambda_1,\ldots,\lambda_{\ell}}$ fixed points
in $\PHilb_{k,n+k}$ projecting to $I$.
There is a Springer action of $\Sn$ on these fixed points, in which they span the induced representation from 
$\Sk{\lambda_1}\times \cdots \times \Sk{\lambda_{\ell}}$ to $\Sn$.

Here $\lambda$ is determined by $c_i$ as in Lemma \ref{lem: Jordan blocks}, and $\ell$ is the length of $\lambda$.
\end{lemma}

In what follows we will need a more explicit description of this action in the fixed point basis. For this, we can also give a more explicit geometric description. Let 
$\PHilb^{(i)}_{k,n+k}$ denote the moduli space of flags of ideals
$$\PHilb^{(i)}_{k,n+k}=\left\{I_k\supset I_{k+1}\supset \cdots\supset I_{k+i}\supset I_{k+i+2}\supset \cdots \supset I_{k+n}=xI_k\right\}.$$
There  is a natural projection $\pi_i:\PHilb_{k,n+k}\to \PHilb^{(i)}_{k,n+k}$. Let $Z_i\subset \PHilb^{(i)}_{k,n+k}$ denote the locus where $yI_{k+i}\subset I_{k+i+2}$. The key properties of $\pi_i$ are captured by the following lemma:

\begin{lemma}
\label{lem:partial flag}
(a) The map $\pi_i$ is an isomorphism outside $Z_i$ and a $\PP^1$-fibration over $Z_i$.

(b) The preimage $\pi_i^{-1}(Z_i)$ is cut out by a section of the line bundle $\cL_i^{-1}\cL_{i+1}$.

(c) A fixed point corresponding to $v_{\ba}$ is not in $\pi_i^{-1}(Z_i)$ if and only if 
$\wt_{n+1-i}(\ba)=\wt_{n-i}(\ba)-\frac mn$. 

(d) The tangent bundle to the fiber of $\pi_i$
 over $Z_i$ is isomorphic to $\cL_i\cL_{i+1}^{-1}$.
\end{lemma}

\begin{proof}
(a) The fiber of $\pi_i$ naturally corresponds to the space of $y$-invariant lines in two-dimensional space $I_{k+i}/I_{k+i+2}$. Since $y$ is nilpotent on $I_{k+i}/I_{k+i+2}$, it is either identically zero and every line is $y$-invariant, or it is a Jordan block and has unique $y$-invariant line.

(b) A flag $I_k\supset I_{k+1}\supset \cdots\supset I_{k+n}=xI_k$ is in $\pi_i^{-1}(Z_i)$ if and only if 
$yI_{k+i}\subset I_{k+i+2}$. Since $yI_{k+i}\subset I_{k+i+1}$, we have a map $s_{y}:\cL_{i}\to \cL_{i+1}$
which is equivalent to a section of $\cL_i^{-1}\cL_{i+1}$.

(c) A fixed point is not in $\pi_i^{-1}(Z_i)$ if and only if the weight of $\cL_{i+1}$ differs from the weight of $\cL_{i}$ by $m$. By Lemma \ref{lem: weights of line bundles} we get
$$
n\wt_{n-i}+m(n-1)=n\wt_{n+1-i}+m(n-1)+m,\ \,
 \wt_{n-i}=\wt_{n+1-i}+\frac mn.
$$  

(d) Recall that the tangent space to $\PP^1=\PP(V)$ at a line $\ell$ is canonically isomorphic to $\Hom(\ell,V/\ell)$. In our case $\ell\simeq I_{k+i+1}/I_{k+i+2}=\cL_{i+1}$ and $V/\ell\simeq I_{k+i}/I_{k+i+1}=\cL_i.$ So the tangent bundle to the fiber is isomorphic to the space 
$\Hom(\cL_{i+1},\cL_i)\simeq \cL_i\cL_{i+1}^{-1}$.
\end{proof}

We can use the maps $\pi_i$ to define the Springer action of $\Sn$ on the homology of $\PHilb_{k,n+k}$.
Let $\gamma_i:\pi_i^{-1}(Z_i)\hookrightarrow \PHilb_{k,n+k}$ denote the natural inclusion map. 
By Lemma \ref{lem:partial flag} we have well-defined Gysin maps $\gamma_i^*:H_*(\PHilb_{k,n+k})\to H_*(\pi_i^{-1}(Z_i))$ and $\pi_i^{*}:H_*(Z_i)\to H_*(\pi_i^{-1}(Z_i))$. Consider the composition
\begin{multline}
B_i:H_*(\PHilb_{k,n+k})\xrightarrow{\gamma_i^*}H_*(\pi_i^{-1}(Z_i))\xrightarrow{\pi_{i*}}H_*(Z_i)\xrightarrow{\pi_i^{*}}H_*(\pi_i^{-1}(Z_i))
\\
\xrightarrow{\gamma_{i*}}H_*(\PHilb_{k,n+k}).
\end{multline}

By Lemma \ref{lem: weights of line bundles} we can identify the fixed point basis in the equivariant cohomology of $\sqcup_k \PHilb_{k,n+k}$ with the basis $v_{\ba}$ in the representation $L_{m/n} = L_{m/n}(\triv)$. In fact, it is more natural to identify it with the renormalized basis $\widetilde{v}_{\ba}$.

\begin{lemma}
\label{lem: action of s_n}
The action of $B_i$ in the equivariant cohomology of $\sqcup_k \PHilb_{k,n+k}$ agrees with the action of $1-s_{n-i}$ on $L_{m/n}$, if we identify the fixed point basis in the former with $\widetilde{v_{\ba}}$.
\end{lemma}
\begin{proof}
We just need to compute the matrix elements of all the operators involved in the definition of $B_i$. 
By Lemma \ref{lem: weights of line bundles} (b) the subvariety $\pi_i^{-1}(Z_i)$ is cut out by a section of $\cL_i^{-1}\cL_{i+1}$ corresponding to the map $s_{y}:\cL_{i}\to \cL_{i+1}$. This map has weight $m$, and so the Gysin map $\gamma_i^*$ correspond to the multiplication by $c_1(\cL_{i+1})-c_1(\cL_{i})-m$ which at a fixed point corresponds to the multiplication by $(n\wt_{n-i}-n\wt_{n+1-i}-m)$. Note that by 
Lemma \ref{lem: weights of line bundles} (c) this annihilates the classes of all fixed points outside 
$\pi_i^{-1}(Z_i)$.

The map $\pi_{i*}$ just maps the class of the fixed point in $\PHilb_{k,n+k}$ to the class of the corresponding class in $\PHilb^{(i)}_{k,n+k}$. The map $\pi^*_{i}$, however, amounts to dividing by the 
cotangent weight of the fiber computed in Lemma \ref{lem:partial flag} (d). 

By combining these factors, it is now easy to compare the matrix elements of $B_i$ with the ones in Proposition \ref{prop:renormalization} and observing that for $c=m/n$ one gets:

$$
(1 - s_{i})\widetilde{v}_{\ba} = \frac{n\wt_{i} - n\wt_{i+1} - m}{n\wt_{i} - n\wt_{i+1}}\left(\widetilde{v}_{\ba} - \widetilde{v}_{s_{i}\cdot \ba}\right)
$$
where $\wt =\wt(\ba)$.
\end{proof}

We also have a geometric analogue of the shift operator $\tau$. Given a flag $I_k\supset I_{k+1}\supset \cdots\supset I_{k+n}=xI_k$, we can consider the flag $I_{k+1}\supset \cdots\supset I_{k+n}=xI_k\supset I_{k+n+1}=xI_{k+1}$. This defines a map $T:\PHilb_{k,n+k}\to \PHilb_{k+1,n+k+1}$. 

\begin{definition}
We define  $W_{k,n+k}\subset \PHilb_{k,n+k}$ as the set of flags $I_k\supset I_{k+1}\supset \cdots\supset I_{k+n}=xI_k$ such that $I_{k+n-1}\subset x\cO_C$.
\end{definition}
 
 It is easy to see that $W_{k,n+k}$ is a closed subvariety in $\PHilb_{k,n+k}$.
 
 \begin{lemma}
 The map $T:\PHilb_{k,n+k}\to \PHilb_{k+1,n+k+1}$ is injective and its image coincides with $W_{k+1,n+k+1}$.
 In particular, $\PHilb_{k,n+k}$ and $W_{k+1,n+k+1}$ are isomorphic.
 \end{lemma}
 
 \begin{proof}
 The image of $T$ is contained in $W_{k+1,n+k+1}$ by construction. Given a flag $I_{k+1}\supset I_{k+2}\supset \cdots\supset I_{k+n+1}=xI_{k+1}$ in $W_{k+1,n+k+1}$, we have $I_{k+n}\subset x\cO_c$, so we can define an ideal $I_{k}:=x^{-1}I_{k+n}$. Since $I_{k+n}\supset xI_{k+1}$, we have $I_{k}\supset I_{k+1}$. Therefore 
  $I_k\supset I_{k+1}\supset \cdots\supset I_{k+n}=xI_k$ is a well defined point in $\PHilb_{k,n+k}$ 
sent to the original flag by $T$.
 \end{proof}
 
 Recall that the  $\cL_n$ has fibers $I_{k+n-1}/I_{k+n}=I_{k+n-1}/xI_{n}$. The inclusion 
 $I_{k+n-1}\hookrightarrow \cO_C$ induces a map $i:\cL_n\to \cO_C/x\cO_C$.

\begin{lemma}
\label{lem: zero section}
Define the covector $\eta: \cO_C/x\cO_C\to \C$ by the equation $\eta(y^{n-1})=1,\eta(y^k)=0$ for $0\le k<n-1$.
Then $W_{k,n+k}$ is the zero locus of the composition 
\begin{equation}
s:\cL_n\xrightarrow{i} \cO_C/x\cO_C\xrightarrow{\eta} \C
\end{equation}
or, equivalently, the zero locus of the section $s:\C \to \cL_n^{-1}$.
\end{lemma}

\begin{proof}
Recall that $W_{k,n+k}$ is cut out by condition $I_{k+n-1}\subset x\cO_C$ which is equivalent to vanishing of $i(\cL_n)$.
Since $i(\cL_n)$ is a $y$-invariant subspace of $\cO_C/x\cO_C$ of dimension at most 1, either $i(\cL_n)=0$ or
$i(\cL_n)=\langle y^{n-1}\rangle$. Therefore $i(\cL_n)=0$ if and only if $\eta(i(\cL_n))=0$. 
\end{proof}

Note that $\PHilb_{k,n+k}$ is in general very singular and has several irreducible components. The section $s$ might vanish on some of these components identically. Still, by Lemma \ref{lem: zero section} we can define Gysin map\cite{Fulton}
$$
j^*:H_*(\PHilb_{k,n+k})\to H_{*-2}(W_{k,n+k}).
$$
where $j=j_k$ is the inclusion
$j:W_{k,n+k}\hookrightarrow \PHilb_{k,n+k}$.
We define $\Lambda$ as the composition
$$
\Lambda:H_*(\PHilb_{k+1,n+k+1})\xrightarrow{j^*} H_{*-2}(W_{k+1,n+k+1}) \xrightarrow{\simeq}H_{*-2}(\PHilb_{k,n+k}). 
$$
 
\begin{lemma}
\label{lem: T Lambda}
We have $T_*\circ \Lambda(-)=c_1(\cL_n)\cap (-)$.
\end{lemma}  
 

\begin{proof}
Indeed, if $j:W_{k,n+k}\hookrightarrow \PHilb_{k,n+k}$ is the inclusion, then 
$$
T_*\circ \Lambda(-)=j_*j^*(-)=c_1(\cL_n)\cap (-)
$$
by Lemma \ref{lem: zero section}.
\end{proof}
 
\begin{theorem}
\label{thm: geometric action}
(a) The total localized equivariant homology 
$$
U=\bigoplus_{k=0}^{\infty}H^{\C^*}_*(\PHilb_{k,n+k}) 
$$
has an action of the rational Cherednik algebra $H_{n,m}$. The action of $\Sn$ is the Springer action described above, $u_{n+1-i}+m(n-1)$ correspond to capping with $c_1(\cL_i)$ and the operators $T$ and $\Lambda$ on $U$ correspond to the action of $\tau$ and $\lambda$. 

(b) The representation $U$ is irreducible and isomorphic to $L_{n,m}(\triv)$. Under this isomorphism, fixed points of $\C^*$ action correspond to the eigenbasis 
$\widetilde{v}_{\ba}$. 
\end{theorem}
 
\begin{proof}
Let $\overline{s}_i,\overline{u_i},\overline{\tau}$
and $\overline{\lambda}$ be the generators of $H_{1,c}$ where $c=m/n$. Recall that $H_{n,m}$ is isomorphic to $H_{1,m/n}$, under this isomorphism the generators $s_i,u_i,\tau$ and $\lambda$ of $H_{n,m}$ are mapped to $\overline{s}_i,n\overline{u_i},\overline{\tau}$
and $n\overline{\lambda}$ respectively. Below, we will use this isomorphism to identify $L_{n,m}(\triv)$ with $L_{m/n}$. 

By localization theorem \cite{Brion,GKM2} $U$ is spanned by classes of fixed points. 
By Lemma \ref{lem: weights of line bundles} these are in bijection
with the basis $v_{\ba}$ 
 (or, equivalently, $\widetilde{v}_{\ba}$) in
$L_{m/n}$. This defines an isomorphism between $U$ and $L_{m/n}$ as
vector spaces. 

Next, we prove that the geometrically defined actions of $u_i, s_i, T$ and $\Lambda$ agree with the corresponding actions on $L_{m/n}$. This is done by explicitly comparing their matrix elements.
For $u_i$ this follows from Lemma \ref{lem: weights of line bundles}. For $s_i$ this follows from Lemma \ref{lem: action of s_n}.
For $T$ and $\tau$ it is easy to see from equation \eqref{eq: action renormalized}. 

The action of $\Lambda$ is uniquely determined by Lemma \ref{lem: T Lambda}. More precisely, the map 
$\eta$ in Lemma \ref{lem: zero section} has equivariant weight $-m(n-1)$ (since the weight of $y^{n-1}$ equals $m(n-1)$), while by Lemma \ref{lem: weights of line bundles} $\cL_n$ has weight $n\wt_1+m(n-1)$. Therefore section $s$ in Lemma \ref{lem: zero section} has weight $n\wt$. By Lemma \ref{lem: T Lambda} we conclude that $T_*\circ \Lambda=n\overline{u_1}=u_1$. 
 
Finally, the operators $u_i, s_i, T$ and $\Lambda$ satisfy the relations in $H_{n,m}$ since their counterparts on $L_{n,m}(\triv)$ do. Therefore there is indeed an action of $H_{n,m}$ on $U$
and it is an irreducible representation.
\end{proof} 
 
\begin{remark}
In principle, one can check all the relations between the geometric operators directly (similarly to the computations in \cite{CGM}), but the above proof seems to be more transparent.
\end{remark} 

\begin{remark}
\label{rem:Euler grading}
Note  the grading of $L_{m/n}$ by eigenvalues of the Euler operator, where $\deg(v_{\ba}) = ||\ba|| = \sum a_{i}$ corresponds to  grading by $k$ in $\bigoplus_{k} H^{\C^*}_{*}(\PHilb_{k, n+k})$. 
\end{remark}

Consider now the Hilbert scheme $\Hilb(C) := \sqcup_{k}\Hilbk(C)$, and recall that we have defined $\PHilb^{x}(C) := \sqcup_{k}\PHilb_{k, n+k}(C)$. We have a $\C^{*}$-equivariant projection $\Pi: \PHilbx \to \Hilb(C)$, $(I_{k} \supset \cdots \supset I_{k+n} = xI_{k}) \mapsto I_{k}$, that induces an $\Sk{n}$-invariant map on (localized) equivariant homology

$$
\Pi_{*}: H_{*}^{\C^*}(\PHilbx) \to H_{*}^{\C^*}(\Hilb(C)).
$$

Now let $I_{k} \in \Hilbk(C)$ be a monomial ideal.
 Thanks to Lemma \ref{lemma:induced representation}, and using the notation there, the span of the elements in $H_{*}^{\C^*}(\PHilb_{k, n+k}(C))$ mapping to $[I_{k}]$ is the induced representation $\Ind_{\Sk{\lambda_1} \times \cdots \times \Sk{\lambda_{\ell}}}^{\Sk{n}}\triv$.
 Now by adjunction

$$
\Hom_{\Sk{n}}(\triv, \Ind_{\Sk{\lambda_1} \times \cdots \times \Sk{\lambda_{\ell}}}^{\Sk{n}}\triv) = \Hom_{\Sk{\lambda_1} \times \cdots \times \Sk{\lambda_{\ell}}}(\Res^{\Sk{n}}_{\Sk{\lambda_1} \times \cdots \times \Sk{\lambda_{\ell}}}\triv, \triv) = \C
$$

\noindent so up to scalars there is a unique $\Sk{n}$-equivariant section to the projection $\Pi_{*}: \Pi_{*}^{-1}(\C[I_{k}])\cap H_{*}^{\C^*}(\PHilb_{k, n+k}(C)) \to \C[I_{k}]$. As a consequence we get the following result.

\begin{proposition}\label{prop:spherical hilb}
We may identify
$H^{\C^*}_{*}(\Hilb(C)) = H^{\C^*}_{*}(\PHilbx)^{\Sk{n}}$ naturally. In particular, we obtain a geometric action of the {\bf spherical} rational Cherednik algebra $eH_{1, m/n}e$ on $H^{\C^*}_{*}(\Hilb(C))$, that makes it an irreducible module isomorphic to
$eL_{m/n} = L_{m/n}^{\Sk{n}}$,
where $e := \frac{1}{n!}\sum_{p \in \Sk{n}} p$ is the trivial idempotent in $\C\Sk{n}$. 
\end{proposition}

\begin{remark}
In \cite{GK}, Garner and Kivinen study an action of the spherical rational Cherednik algebra on the homology of $\Hilb(C)$ using the Coulomb branch perspective. They identify $\Hilb(C)$ with a generalized affine Springer fiber and use the realization of $eH_{1, m/n}e$ as a quantized Coulomb branch algebra \cite{KN, W} to define an action via convolution diagrams. We will compare their construction to ours, in the parabolic setting, in Section \ref{sec: comparison to GK}. 

\end{remark}

\subsection{Parabolic Hilbert schemes as generalized affine Springer fibers}\label{sec:gasf} The goal of this section is to show that $\PHilbx= \bigsqcup_{k}\PHilb_{k, n+k}$ can be realized as a generalized affine Springer fiber. Thanks to \cite{GKM}, a consequence of this is that $\PHilb_{k, n+k}$ admits a paving by affine cells and therefore its cohomology is equivariantly formal.

Let us set $G := \operatorname{GL}_{n}$, acting on the vector space $N := \C^{n} \oplus \gln$, so that $N$ is the representation space of the framed Jordan quiver:

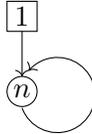
\begin{center}
\begin{tikzpicture}
\draw (0,5) node {1};
\draw (-0.2,4.8)--(-0.2,5.2)--(0.2,5.2)--(0.2,4.8)--(-0.2,4.8);
\node at (0,4) (n) {$n$};
\draw (0,4) circle (0.2);
\draw [->] (0,4.8)--(0,4.2);
\draw [->] (n.south) arc (-160:140:0.5);
\end{tikzpicture}
\end{center}

We will denote $\KK := \C((\epsilon))$ and $\OO := \C[[\epsilon]]$. We consider the groups $G_{\OO} \subseteq G_{\KK}$ of invertible $\OO$-linear (resp. $\KK$-linear) transformations on $\OO^{n}$ (resp. $\KK^{n}$).

We choose an $\OO$-basis $\{b_{1}, \dots, b_{n}\}$ of $\OO^{n}$. We define $b_{i}$ for $i \in \Z$ by setting $b_{i+n} := \epsilon b_{i}$. The \emph{standard flag} is the flag of $\OO$-lattices in $\KK^{n}$ 
$$
\cdots \supseteq \cI_{j-1} \supseteq \cI_{j} \supseteq \cI_{j+1} \supseteq \cdots 
$$

\noindent where $\cI_{j}$ is the $\OO$-span of $\{b_{j}, b_{j+1}, \dots, b_{j+n-1}\}$. We denote by $I \subseteq G_{\KK}$ the standard Iwahori subgroup, that is, the stabilizer of the standard flag. The quotient space $\Fl := G_{\KK}/I$ is known as the affine flag variety. This is an ind-scheme parametrizing flags of $\OO$-lattices $\cdots \supseteq \cJ_{j-1} \supseteq \cJ_{j} \supseteq \cJ_{j+1} \supseteq \cdots$ in $\KK^{n}$ subject to the condition $\cJ_{j+n} = \epsilon \cJ_{j}$ for every integer $j \in \Z$.

We will need a dual realization of $\Fl$. Let us consider the $\KK$-dual $(\KK^n)^*$ of $\KK^n$. This comes equipped with a dual basis $b_{1}^{*}, \dots, b_{n}^{*}$. As above, we set $b^{*}_{i+n} := \epsilon b_{i}^{*}$, and the standard flag in $(\KK^n)^{*}$ is the flag of $\OO$-lattices in $\KK^n$
$$
\cdots \supseteq \cI_{j-1}^{*} \supseteq \cI_{j}^* \supseteq \cI_{j+1}^* \supseteq \cdots 
$$
where $\cI_{j}^{*}$ is the $\OO$-span of $\{b_j^{*}, \dots, b_{j+n-1}^{*}\}$. We remark that $\cI_{1}^{*}$ is the \emph{standard lattice}
\[
(\OO^n)^{*} = \{f \in (\KK^n)^{*} \mid f(\OO^n) \subseteq \OO\}.
\]
Note that $G_{\KK}$ acts naturally on the space $(\KK^n)^{*}$, and we can identify $\Fl$ with the set of $\OO$-lattices $\cdots \supseteq \cJ_{j-1} \supseteq \cJ_{j} \supseteq \cJ_{j+1} \supseteq \cdots$ in $(\KK^{n})^{*}$ subject to the condition $\cJ_{j+n} = \epsilon \cJ_{j}$ for every integer $j \in \Z$, via the identification
\[
[g] \mapsto \cdots \supseteq \cI_{j-1}^{*}g^{-1} \supseteq \cI_{j}^*g^{-1} \supseteq \cI_{j+1}^*g^{-1} \supseteq \cdots 
\]

The group $G_{\KK}$ acts on the module $N_{\KK} := \KK \otimes N = \KK^{n} \oplus \gln(\KK)$ in the natural way, and the subgroup $G_{\OO} \subseteq G_{\KK}$ preserves the $\OO$-submodule $N_{\OO} := \OO \otimes N \subseteq N_{\KK}$. Now we consider the element $Y \in \gln(\OO)$ that in the basis $\{b_{1}, \dots, b_{n}\}$ is represented by the matrix \eqref{eq:Y}, with $x$ replaced by $\epsilon$, and the element $(b_{1}, Y) \in N_{\OO}$. We will consider the \emph{generalized affine Springer fiber}, cf. \cite{BFN, GK, GKM}

$$
\Spr(b_{1}, Y) := \{[g] \in \Fl \mid (g^{-1}b_{1}, g^{-1}Yg) \in \mathbb{O}^{n} \oplus \mathfrak{i}\} \subseteq \Fl
$$

\noindent where $\mathfrak{i}$ is the Lie algebra of the Iwahori subgroup $I$. More concretely, $\mathfrak{i} := \{X \in \gln(\OO) \mid X|_{\epsilon = 0} \; \text{is lower triangular}\}$. 

\begin{proposition}\label{prop:gasf}
We have an isomorphism 

$$
\Spr(b_{1}, Y) \cong \bigsqcup_{k}\PHilb_{k, n+k}
$$
\end{proposition}
\begin{proof}
We use the presentation of $\cO_{C}$ at the beginning of this section as a free $\C[[x]]$-module of rank $n$. In this presentation, an ideal of $\cO_{C}$ corresponds to a $\C[[x]]$-submodule $I \subseteq \C[[x]]^{n}$ closed under the action of the matrix $Y$ in \eqref{eq:Y}. Similarly, an element of $\bigsqcup_{k} \PHilb_{k, n+k}$ corresponds to a flag of $\C[[x]]$-submodules
$\C[[x]]^{n} \supseteq I_{k} \supseteq \cdots \supseteq I_{k+n-1}
\supseteq xI_{k}$ such that $\dim \C[[x]]^{n}/I_{j} = j < \infty$, each
ideal $I_{j}$ is stable under the action of $Y$ and $\dim I_{j}/I_{j+1}
= 1$. 
Now, we identify $\cO_{C}$ with the dual $(\KK^n)^{*}$ and, as above, we identify $[g] \in \Fl$ with the flag $\cI^*_{1}g^{-1} \supseteq \cdots \supseteq \cI^*_{n-1}g^{-1} \supseteq \cI^*_{n}g^{-1} = \epsilon \cI^{*}_{1}g^{-1}$ where, as above, $\cI^*_{1} \supseteq \cI^*_{2} \supseteq \cdots $ is the standard flag. Identifying $x = \epsilon$, we see that to prove the proposition we have to check that the following conditions are equivalent:

\begin{enumerate}
\item $\cI^*_{1}g^{-1} \subseteq (\OO^{n})^*$ and $\cI^{*}_{j}g^{-1}$ is closed under the action of $Y$ for every $j \geq 1$.
\item $g^{-1}b_{1} \in \OO^{n}$ and $g^{-1}Yg \in \mathfrak{i}$.
\end{enumerate}

Since $\cI^*_{j}g^{-1}$ is the $\OO$-span of $\{b^*_{j}g^{-1}, \dots, b^*_{j+n-1}g^{-1}\}$ and $b^*_{j+n} = \epsilon b^*_{j}$ for every $j$, it is easy to see that (1) $\Rightarrow$ (2). Let us check that (2) $\Rightarrow$ (1). First, we need to check that $b^{*}_{1}g^{-1}, \dots, b^{*}_{n}g^{-1} \in (\mathbb{O}^{n})^*$. That is, we need to show that $b^{*}_{i}g^{-1}(\OO^n) \subseteq \OO$ or, equivalently, that $b_{i}^{*}g^{-1}(b_{j})\in \OO$ for every $i, j = 1,\dots, n$. Since  $b_{i}^{*} \in (\OO^n)^{*}$ for $i = 1, \dots, n$, it is enough to show that $g^{-1}b_{j} \in \OO^n$ for $j = 1, \dots, n$. For $j = 1$, this is one of the conditions in (2). For $j > 1$, observe that $g^{-1}b_{j} = g^{-1}Yb_{j-1} = g^{-1}Yg(g^{-1}b_{j-1})$. Since $g^{-1}Yg$ in $\mathfrak{i} \subseteq \mathfrak{gl}_{n}(\OO)$, the result follows by induction.

Now we need to show that $\cI^{*}_{j}g^{-1}$ is closed under the action of $Y$ for every $j \geq 1$, that is, we need to show that
\[
\cI^{*}_{j}g^{-1}Y \subseteq \cI^{*}_{j}g^{-1}
\]
It is clearly enough to do this for $j = 1, \dots, n-1$. The condition $g^{-1}Yg \in \mathfrak{i}$ is equivalent to $b^{*}_{i}g^{-1}Y \in \OO\operatorname{-span}\{b^*_{i}g^{-1}, \dots, b^*_{i+n-1}g^{-1}\}$ for every $i = 1, \dots, n$. This clearly implies that $\cI^*_{j}g^{-1}$ is closed under $Y$.
\end{proof}

\begin{remark}
Proposition \ref{prop:gasf} is a special case of a flag version of the main result of \cite{GK}, which the authors kindly provided a preliminary version of.
\end{remark}

Now we would like to verify that the generalized affine Springer fiber $\Spr(b_{1}, Y)$ satisfies the conditions of \cite[(3.2)]{GKM}. Following that paper, let us denote by $\mathfrak{a} := X_{*}(A) \otimes_\Z \RR$, where $A \subseteq G = \operatorname{GL}_{n}$ is a maximal torus, that we identify with the set of diagonal invertible matrices. For each weight $\xi \in \mathfrak{a}^{*}$, let us denote by $N_{\xi} \subseteq N$ the corresponding weight space. For $a \in \mathfrak{a}$ and $t \in \RR$, we denote

$$
N_{\KK, a, t} := \prod_{\substack{\xi \in \mathfrak{a}^{*}, d \in \Z \\ \langle \xi, a \rangle + d \geq t}}N_{\xi}\epsilon^{d} \subseteq N_{\KK}.
$$

For $a \in \mathfrak{a}$, let $\mathfrak{g}_{a} := \mathfrak{g}_{\KK, a, 0} \cap \mathfrak{g}_{\OO}$. This is a Lie subalgebra of $\mathfrak{g}_{\OO}$ and we let $G_{a} \subseteq G_{\OO}$ be the corresponding subgroup, which is an Iwahori subgroup. 

\begin{lemma}\label{lemma:GKM}
There exist $a \in \mathfrak{a}$ and $t \in \RR$ such that 
$G_{a} = I$ is the Iwahori subgroup,
and $N_{\KK, a, t} = \OO^{n} \oplus \mathfrak{i}$.  \end{lemma}
\begin{proof}
Take any $a = \operatorname{diag}(a_{1}, \dots, a_{n}) \in \mathfrak{a}$ with $0 < a_{1} < \dots < a_{n} < 1$ and $t = 0$. It is straightforward to verify the result. 
\end{proof}

Note that Lemma \ref{lemma:GKM} tells us that $\Spr(b_{1}, Y)$ is one of the varieties considered in \cite[Section 3]{GKM}. In the notation of that paper, we have
$$
\Spr(b_{1}, Y) = \mathcal{F}_{a}(t, (b_{1}, Y)).
$$
In \cite[Section 3.2]{GKM} it was proved that $\mathcal{F}_{a}(t, (b_{1}, Y))$  has an affine paving provided  that there exist $b\in \mathfrak{a}$ and $c\in \RR$ such that the following conditions are satisfied:
\begin{itemize}
\item $c\ge t$
\item $(b_{1}, Y) \in N_{\mathbb{K}, b, c}$
\item The projection $\overline{(b_{1}, Y)}$  is $G$-good (in the sense of \cite{GKM}), that is, no nonzero $G$-unstable covector in $N^{*}$ vanishes on the $\gln$-orbit of $\overline{(b_{1}, Y)}$
\end{itemize}

To verify these conditions, we consider $b = \operatorname{diag}(c, 2c, \dots, nc) \in \mathfrak{a}$, where $c := m/n$. Obviously $c > t = 0$, and is easy to check that $(b_{1}, Y) \in N_{\mathbb{K}, b, c}$. 
For the last condition, we need to verify that the element

$$
\overline{(b_{1}, Y)} =  \left(\begin{matrix} 1 \\ 0 \\ 0 \\ \vdots \\ 0\end{matrix}\right), \,
 \left(
\begin{matrix}
0 & 0 & \cdots & 0 & 1\\
1 & 0 & \cdots & 0 & 0 \\
0 & 1 & \cdots & 0 & 0\\
\vdots & \vdots& \ddots & \vdots & 0\\
0 & 0 & \cdots & 1& 0 \\
\end{matrix}
\right) \in N
$$

\noindent is $G$-good.
This is a consequence of the following result. 

\begin{proposition}\label{prop:g-good}
Let $X$ be a regular semisimple matrix and $v$ a cyclic vector for $X$. Then $(v, X) \in N$ is $G$-good. 
\end{proposition}

To prove Proposition \ref{prop:g-good}, we first give a necessary condition for a vector in the adjoint representation $\gln$ to be unstable.

\begin{lemma}\label{lemma:g-unstable}
Assume $B \in \gln$ is $G$-unstable. Then, $B$ is nilpotent. 
\end{lemma}
\begin{proof}
By definition, cf. \cite{GKM}, $B$ is unstable if and only if there exists a semisimple matrix $y$ and $t_{1}, \dots, t_{k} > 0$ such that $B = B_{1} + \cdots + B_{k}$, with $[y, B_{k}] = t_{k}B_{k}$. Since the $t_{i}$ are strictly positive and the filtration given by $y$ is bounded above, the result follows. 
\end{proof}

Returning to the setting of Proposition \ref{prop:g-good}, we may assume that $X$ is already in diagonal form. So $X = \operatorname{diag}(x_{i})$ with $x_{i} \neq x_{j}$ for $i \neq j$ and $v = (v_{i})$, the cyclicity condition is equivalent to $v_{i} \neq 0$ for every $i$. 

\begin{lemma}\label{lemma:traces}
Let $(w, B) \in N$ be such that $\operatorname{tr}(B[\xi, X]) + w\cdot \xi v = 0$ for every $\xi \in \gln$. Then
\begin{enumerate}
\item $w_{i} = 0$ for every $i = 1, \dots, n$.
\item $b_{ij} = 0$ for $i \neq j$.
\end{enumerate}
\end{lemma}
\begin{proof}
The proof is straightforward, but let us give it for the sake of completeness. We have $[\xi, X] = (\xi_{ij}(x_{i} - x_{j}))_{ij}$ and $\xi.v = (\sum_{j}\xi_{ij}v_{j})_{i}$. Thus,

$$
\operatorname{tr}(B[\xi, X]) + w\cdot \xi v = \sum_{i, j = 1}^{n}b_{ij}\xi_{ji}(x_{j} - x_{i}) + \xi_{ij}v_{i}w_{j} = 0
$$

\noindent for every matrix $\xi \in \gln$. Taking the matrix $\xi$ with $\xi_{ii} = 1$ and all other coordinates $0$ we see, using $v_{i} \neq 0$, that $w_{i} = 0$.  Now take $i \neq j$.  Taking the matrix $\xi$ with $\xi_{ij} \neq 0$ and all other coordinates $0$ we see, using $x_{i} - x_{j} \neq 0$, that $b_{ji} = 0$. The result follows.
\end{proof}

\begin{proof}[Proof of Proposition \ref{prop:g-good}]
Let $(w, B) \in N^{*} \cong N$ be an unstable covector vanishing on $\gln\cdot(v, X)$,
 where we use the trace form to identify $N^{*} \cong N$. Thanks to Lemma \ref{lemma:traces} (1) we have that $w = 0$. It follows now from Lemma \ref{lemma:g-unstable} that $B$ is nilpotent. But Lemma \ref{lemma:traces} (2) implies that $B$ is semisimple as well. So $B = 0$, and it follows that $(v, X)$ is $G$-good. 
\end{proof}

From \cite{GKM}, we obtain the following result.

\begin{corollary}
The generalized affine Springer fiber $\Spr(b_{1}, Y) = \bigsqcup\limits_{k}\PHilb_{k, n+k}$ is paved by affine spaces. Thus, its cohomology is equivariantly formal. 
\end{corollary}

\begin{remark}\label{rmk:positive affine g}
 The usual (as opposed to generalized)  affine Springer fiber $\Spr(Y)$ can be obtained by a similar construction for $N=\gln$. Similarly to Proposition \ref{prop:gasf}, it can be defined as the space of $Y$-invariant flags  
$$
\cI^{*}_{1}g^{-1} \supseteq \cdots \supseteq \cI^{*}_{n-1}g^{-1} \supseteq \cI^{*}_{n}g^{-1} = \epsilon \cI^{*}_{1}g^{-1}
$$ 
in $(\KK^n)^{*}$, but these flags are no longer required to be contained in $(\OO^n)^{*}$. It was proved in \cite{LS,GKM} that for the same matrix $Y$ given by \eqref{eq:Y} the usual affine Springer fiber $\Spr(Y)$ is paved by affine spaces, and the combinatorics of this paving was studied e.g. in \cite{LS,GMV}. 

The Springer action of $\Sn$ and the operator $T$ in cohomology of $\Spr(Y)$ were considered in \cite{Yun,OY,OY2,VV}. They were shown to generate the extended affine symmetric group, in particular, $T$ is invertible. Indeed, 
$$
T\left[\cI_{1}^{*}g^{-1} \supseteq \cdots \supseteq \cI_{n-1}^{*}g^{-1} \supseteq \cI_{n}^{*}g^{-1} = \epsilon \cI_{1}^{*}g^{-1}\right]=\left[\cI_{2}^{*}g^{-1} \supseteq \cdots \supseteq \cI_{n-1}^{*}g^{-1} \supseteq \epsilon \cI_{1}^{*}g^{-1}\supseteq \epsilon\cI_{2}^{*}g^{-1}\right]
$$
while
$$
T^{-1}\left[\cI^{*}_{1}g^{-1} \supseteq \cdots \supseteq \cI^{*}_{n-1}g^{-1} \supseteq \cI^{*}_{n}g^{-1} = \epsilon \cI^{*}_{1}g^{-1}\right]=\left[\epsilon^{-1}\cI_{n-1}^{*}g^{-1}\supseteq \cI_{1}^{*}g^{-1} \supseteq \cdots \supseteq \cI_{n-1}^{*}g^{-1}\right].
$$
Furthermore, $\Sn$, $T$ and line bundles $\cL_i$ were used in
\cite{OY,OY2} to construct the action of the {\bf trigonometric}
Cherednik algebra on the equivariant homology of the affine Springer fiber.

In our setting, the failure of $T$ to be invertible gives rise to a new operator $\Lambda$ and together they generate the {\bf rational} Cherednik algebra. This shows both the similarity and a subtle distinction between the trigonometric and rational setups.
\end{remark}


\subsection{Comparison to action by convolution diagrams}\label{sec: comparison to GK} The main result of \cite{HKW} constructs an action of the Coulomb branch algebra for $(G,N)$ in the equivariant homology of any generalized affine Springer fiber for $(G,N)$ satisfying some mild assumptions. If the generalized affine Springer fiber is invariant under the loop rotation, then the action extends to the equivariant homology. The main result of \cite{GK} identifies the Hilbert schemes of points on arbitrary plane curve singularities with the generalized affine Springer fibers for $(G,N)=(\operatorname{GL}_{n}, \C^n \oplus \mathfrak{gl}_{n})$,  as in Section \ref{sec:gasf}. By combining these results, \cite{GK} defines an action of the rational Cherednik algebra in the (equivariant) homology of Hilbert schemes of points on arbitrary plane curve singularities. The goal of this section is to compare their action with ours for the singularity $\{x^m=y^n\}$, see also \cite[Section 4.3.2]{GK}. 

Let $\bt, \bc$ be formal variables and consider the $\C[\bt, \bc]$-algebra $H_{\bt, \bc}(\Sk{n}, \C^n)$ defined by the same relations as the usual Cherednik algebra but with the parameters $t$, $c$ replaced by the variables $\bt, \bc$. Thanks to work of Webster, see \cite{LW, W},
$H_{\bt, \bc} := H_{\bt, \bc}(\Sk{n}, \C^n)$
is a generalized BFN Coulomb branch algebra. 

Recall that if we have a reductive group $G$ acting on a vector space $N$, the BFN Coulomb branch algebra is defined as the equivariant Borel-Moore homology $H_{*}^{G_{\OO} \rtimes \C^{*}_{\rot}}(\mathcal{R}_{G, N})$ where $\mathcal{R}_{G, N}$ is a space modeled after the affine Grassmannian and $\C^{*}_{\rot}$ is the torus acting by loop rotations, see \cite{BFN} for details. When $G = \operatorname{GL}_{n}$ and $N = \C^n \oplus \mathfrak{gl}_{n}$ we get precisely the spherical rational Cherednik algebra. To get the full Cherednik algebra, we need to replace $\mathcal{R}_{G, N}$ with a larger space $\mathcal{R}'_{G, N}$ that is rather modeled after the affine flag variety, so we have 
$H_{\bt, \bc} \cong H^{(I \rtimes \C^{*}_{\rot})\times \C^{*}_{\fl}}_{*}(\mathcal{R}'_{G, N})$ where $I \subseteq G_{\KK}$ is the standard Iwahori and the action of $\C^{*}_{\fl}$ comes from the framing vector. The parameter $\bt$ is the $\C^{*}_{\rot}$-equivariant parameter, and the parameter $\bc$ is the $\C^{*}_{\fl}$-equivariant parameter. See \cite{LW, W} for details. 

To compare the actions we look at the isomorphism 
$H_{\bt, \bc} \cong H_{*}^{(I \rtimes \C^{*}_{\rot})\times \C^{*}_{\fl}}(\mathcal{R}'_{G, N})$ constructed by Webster in \cite[Lemma 4.2]{W}. First, we have both algebras acting on a polynomial algebra $\C[\bt, \bc][U_1, \dots, U_n]$. On the Cherednik algebra side, this comes from identifying $U_{i}$ with the Dunkl-Opdam elements $u_{i}$, and we remark that this is \emph{not} the usual polynomial representation of 
$H_{\bt, \bc}$, see \cite[(2.17)--(2.22)]{W}. On the Coulomb side, this comes from identifying $\C[\bt, \bc][U_1, \dots, U_n] \cong H_{*}^{(I\rtimes \C^{*}_{\rot}) \times \C^{*}_{\fl}}(\operatorname{pt})$, where the $U_{i}$ are the Chern classes of the tautological line bundles on the affine flag variety. Both representations are faithful, and we need to identify the operators on $\C[\bt, \bc][U_1, \dots, U_n]$ corresponding to $\tau, \lambda$ and $\Sk{n}$. 

According to \cite[Lemma 4.2]{W}, the action of $\tau$ corresponds to the action of the correspondence: \footnote{Note that our $\tau$ is Webster's $\sigma$, while our $\lambda$ is denoted $\tau$ by Webster.}

$$
\mathtt{T} := \{(F_{\bullet}, F'_{\bullet}) \in \mathcal{F}l \times \mathcal{F}l : F_{i} = F'_{i-1}\},
$$

\noindent while the action of $\lambda$ corresponds to the action of the correspondence:

$$
\mathtt{L} := \{(F_{\bullet}, F'_{\bullet}) \in \mathcal{F}l \times \mathcal{F}l : F_{i} = F'_{i+1}\}.
$$

\begin{remark}
\label{remark:GK Fourier}
Note that the rational Cherednik algebra $H_{\bt,\bc}$ admits a \emph{Fourier transform}, that is, a $\C[\bt,\bc]$-involution sending $y_{i} \mapsto x_{i}$, $x_{i} \mapsto -y_{i}$ and $s_{i} \mapsto s_{i}$. On the Coulomb branch setting, this automorphism interchanges the correspondences $\mathtt{T}$ and $\mathtt{L}$. So there is a choice of isomorphism $H_{\bt,\bc} \to  H^{(I \rtimes \C^{*}_{\rot})\times \C^{*}_{\fl}}_{*}(\mathcal{R}'_{G, N})$. To resolve this, we note that according to \cite[Proposition 1.4]{GK} the action of $H_{\bt, \bc}$ on $H^{\C^{*}}_{*}(\PHilb^{x}(C))$ coming from a $\C[\bt, \bc]$-isomorphism $H_{\bt,\bc} \to  H^{(I \rtimes \C^{*}_{\rot})\times \C^{*}_{\fl}}_{*}(\mathcal{R}'_{G, N})$ factors through $H_{m/n} = H_{\bt, \bc}/(\bt - 1, \bc - m/n)$  and we choose the isomorphism that sends the module constructed in \cite[Theorem 4.9]{GK} to the category $\cO_{m/n}$. 
\end{remark}


It follows from the comparison of convolution diagrams to correspondences in \cite[Section 4.2.1]{GK} that the actions of $\tau, \lambda$ that we defined coincide with those defined by \cite[Theorem 4.9 and Corollary 4.16]{GK}. The action of $\Sk{n}$ that we defined comes from projections to partial flag varieties, cf. Section \ref{sec:geometric operators} while that in \cite{GK} comes from the usual Springer action of $\Sk{n}$ on the homology of Springer fibers. The coincidence of these is well-known. Since the algebra 
$H_{t, c}$ is generated by $\tau, \lambda$ and $\Sk{n}$, Proposition \ref{prob: eliminate u}, we obtain the following result, see also \cite[Theorem 4.29]{GK}. 

\begin{proposition}
The action of  $H_{m/n}$ on $H^{\C^{*}}_{*}(\PHilb^{x}(C))$ defined in Theorem \ref{thm: geometric action} coincides with that constructed by Garner and Kivinen in \cite[Proposition 1.4]{GK}.
\end{proposition}

\begin{corollary}
There is an action of  $H_{m/n}$ on the non-localized equivariant homology $H_{*}^{\C^{*}}(\PHilb^{x}(C))$ lifting the action from Theorem \ref{thm: geometric action}. 
\end{corollary}

\begin{remark}
Let $C$ be a plane curve singularity and assume  the $x$-projection $C \to \C$ has degree $n$. In \cite{GK}, Garner and Kivinen   construct an action of the algebra $H_{0,0} = {\C[x_{1}, \dots, x_{n}, y_{1}, \dots, y_{n}]\rtimes \Sk{n}} = H_{\bt, \bc}/(\bt, \bc)$ on the \emph{non-equivariant} homology 
$H_{*}(\PHilb^{x}(C))$; see also \cite{HKW}. 
\end{remark}

\begin{remark}
\label{rem: noncoprime}
If $C=\{x^m=y^n\}$ and $\gcd(m,n)=d>1$ then the curve $C$ has $d$ irreducible components. There is a $\C^*$ action on $C$ and on Hilbert schemes on $C$, and the results of \cite{GK} still apply, so one gets an interesting representation of the rational Cherednik algebra $H_{m/n}$ in the equivariant homology of $\sqcup_{k} \PHilb^{k,n+k}(C)$. It would be very interesting to study this representation.

Note that the $\C^*$ action on the Hilbert schemes no longer has isolated fixed points, so even computing the character of this representation is a nontrivial problem. Nevertheless, we expect the representation to have minimal support in the sense of \cite{EGL}. 
Indeed, the conjectures of \cite{ORS} relate the homology of $\Hilb(C)$ to the HOMFLY-PT invariant of the $(m,n)$ torus link.  
On the other hand, by \cite[Theorem 4.11]{EGL} the same invariant can be obtained as a character of a certain explicit minimally supported representation of the spherical rational Cherednik algebra with parameter $m/n$. 
\end{remark}

\section{Parabolic Hilbert schemes and quantized Gieseker varieties}\label{sect:gieseker}

In this section, we use Theorem \ref{thm: geometric action} together with \cite{EKLS} to study the geometric representation theory of quantized Gieseker varieties.
 
\subsection{Quantized Gieseker varieties} Fix positive integers $n, r > 0$ and consider the vector space 

$$
R := \gln \oplus \Hom(\C^r, \C^n).
$$

We have a natural action of the group $\operatorname{GL}_{n}$ on $R$, so every element $\xi \in \gln$ induces a vector field on $R$, that we denote by $\xi_{R}$. In particular, $\xi_{R} \in D(R)$, the algebra of polynomial differential operators on $R$. Note that $\operatorname{GL}_{n}$ acts on $D(R)$. Let $c \in \C$. It is straightforward to see that the following space is in fact an associative algebra,

$$
\cA_{c}(n, r) := \left[\frac{D(R)}{D(R)\{\xi_{R} - c\operatorname{tr}(\xi) : \xi \in \gln\}}\right]^{\operatorname{GL}_{n}}
$$

\noindent we call $\cA_{c}(n, r)$ a \emph{quantized Gieseker variety.} 

\begin{example}
When $r = 1$ then $\cA_{c}(n, r) = eH_{c}e$, the spherical subalgebra in the type $\gln$-Cherednik algebra. This follows from the main result of \cite{GG}. 
\end{example}

Let us now deal with the representation theory of $\cA_{c}(n, r)$. We follow \cite[Section 3]{EKLS}. Let $T_{0} \subseteq \operatorname{GL}_{r}$ be a maximal torus, and $T := \C^{*} \times T_{0}$. For each co-character $\nu: \C^{*} \to T$ we can define a category $\cO_{\nu}(\cA_{c}(n,r))$ of highest-weight $\cA_{c}(n, r)$-modules. The co-character $\nu$ has the form $t \mapsto (t^{\nu_{0}}, \nu'(t))$ for some co-character $\nu'$ of $\operatorname{GL}_{r}$. If $\nu_{0} \neq 0$, then $\cO_{\nu}(\cA_{c}(n, r))$ admits a module of Gelfand-Kirillov (GK)-dimension $1$ if and only if $c = m/n$, where $\gcd(m,n) = 1$ and $c \not\in (-r, 0)$. In this case, $\cO_{\nu}(\cA_{c}(n, r))$ admits a unique irreducible representation of GK-dimension $1$, that we denote $\mathcal{L}^{\nu}_{m/n}(n, r)$. Moreover, $\mathcal{L}^{\nu}_{m/n}(n, r)$ depends only on the sign of $\nu_{0}$, so we have two cases: $\mathcal{L}^{+}_{m/n}(n, r)$ and $\mathcal{L}^{-}_{m/n}(n, r)$. We denote $\Lnr_{m/n} := \mathcal{L}^{-}_{m/n}(n, r)$. Our goal is to give a geometric description of this representation. 

The next proposition follows from \cite{EKLS}. 
\begin{proposition}\label{prop:Lmnr}
Assume $m, n > 0$. We have a vector space isomorphism
$$
\Lnr_{m/n}(n, r) = (L_{n/m}(\triv) \otimes (\C^r)^{\otimes m})^{\Sk{m}}
$$
\noindent where $L_{n/m}(\triv)$ is the simple highest weight representation of $H_{n/m}(\Sk{m}, \C^m)$ and the action of $\Sk{m}$ on $L_{n/m}(\triv) \otimes (\C^r)^{\otimes m}$ is diagonal.
\end{proposition}
\begin{proof}
The $\sln$-version of this result is \cite[Corollary 2.18]{EKLS}. The $\gln$-version is proved identically. Alternatively, it follows from the $\sln$-version by multiplying both sides of \cite[Corollary 2.18]{EKLS} by a polynomial algebra in one variable.
\end{proof}
 
We would like to emphasize that in the statement of Proposition \ref{prop:Lmnr} there is a swap in the parameters $n, m$. 

Let us elaborate on the statement of Proposition \ref{prop:Lmnr}. A priori, it is only a vector space identification. However, we can recover the action of $\cA_{m/n}(n, r)$ on the space $(L_{n/m}(\triv) \otimes (\C^r)^{\otimes m})^{\Sk{m}}$ as follows. First, we construct a matrix version of the rational Cherednik algebra.

\begin{definition}\label{def:matrix RCA}
Let $t, c \in \C$ and $m, r  \in \Z_{>0}$. We define the algebra $H_{t, c}(m, r)$ as the quotient of the semidirect product $(\C\langle x_1, \dots, x_m, y_1, \dots, y_m\rangle \otimes (\End(\C^r))^{\otimes m})\rtimes \Sk{m}$ by the relations 

\begin{itemize}
\item $[y_{\ell}, y_{N}] = 0 = [x_{\ell}, x_{N}]$ for any $\ell, N = 1, \dots, m$.
\item $[y_{\ell}, x_{N}] = c\left(\sum_{i, j = 1}^{r}(E_{ij})_{\ell}(E_{ji})_{N}\right)(\ell, N)$ if $\ell \neq N$.
\item $[y_{\ell}, x_{\ell}] = t - c\sum_{N \neq \ell}\left(\sum_{i, j = 1}^{r}(E_{ij})_{\ell}(E_{ji})_{N}\right)(\ell, N)$
\end{itemize}

\noindent where $E_{ij}$ is the $r \times r$ matrix that has a $1$ in the $(i,j)$-th position and zeroes everywhere else, and $(E_{ij})_{\ell} \in \End(\C^r)^{\otimes m}$ is $\Id \otimes \cdots \otimes \Id \otimes E_{ij} \otimes \Id \otimes \cdots \otimes \Id$, where $E_{ij}$ is in the $\ell$-th position.
\end{definition}

For example, when $r = 1$ we simply recover the rational Cherednik
algebra $H_{t, c}(\Sk{m}, \C^m)$.   To lighten notation but still
 emphasize the role of $m$ over $n$, we will write
$H_{t, c}(m)$ in place of $H_{t, c}(\Sk{m}, \C^m)$ or $H_{t,c}$ below.

It is clear from the relations that if $M$ is an
$H_{t,c}(m)$-module,
then $M \otimes (\C^r)^{\otimes m}$ becomes an $H_{t, c}(m, r)$-module, where the elements $x_{1}, \dots, x_{m}, y_{1}, \dots, y_{m}$ act only on the $M$ tensor factor, the elements from $\End(\C^r)^{\otimes m}$ act only on the $(\C^r)^{\otimes m}$ tensor factor, and $\Sk{m}$ acts diagonally. In fact, this defines a category equivalence
$H_{t,c}(m)\text{-mod} \to H_{t, c}(m, r)\text{-mod}$, see \cite{EKLS}.
 Thus, the algebra $H_{1,n/m}(m, r)$ acts on $L_{n/m}(\triv) \otimes (\C^r)^{\otimes m}$. 

Now we can form the spherical subalgebra $eH_{1, n/m}(m, r)e$, 
with respect to the trivial idempotent
$e = \frac{1}{m!}\sum_{p \in \Sk{m}} p$, that acts on the space $(L_{n/m}(\triv) \otimes (\C^r)^{\otimes m})^{\Sk{m}}$. Upon the identification $\Lnr_{m/n}(n, r) = (L_{n/m}(\triv) \otimes (\C^r)^{\otimes m})^{\Sk{m}}$ of Proposition \ref{prop:Lmnr}, the actions of $\cA_{m/n}(n, r)$ and $eH_{1, n/m}(m, r)e$ on their respective spaces get identified. This follows from \cite[Section 2]{EKLS} after minor modifications.  

\subsection{Compositional parabolic Hilbert schemes, combinatorially} We consider the curve $C = \{x^{m} = y^{n}\}$. Let us consider the scheme
$$
\FHilb^{r, y} := \{\cO_{C} \supseteq J^{0} \supseteq \cdots \supseteq J^{r-1} \supseteq J^{r} = yJ^{0}\}
$$
\noindent where $J^{k}$ are ideals in $\cO_{C}$ of finite codimension
(not necessarily $k$). We have an action of $\C^{*}$ on $\FHilb^{r, y}$, and the fixed points can be identified with chains of monomial ideals. We can encode these as follows. Start with the monomial ideal $J^{0} = \C[[y]]\langle y^{c_{1}}, xy^{c_{2}}, \dots, x^{m-1}y^{c_{m}}\rangle \subseteq \cO_{C}$. For $k = 1, \dots, r$ let $\gamma_{k} := \dim(J^{k-1}/J^{k}) \geq 0$. Note that $\sum_{k = 1}^{r}\gamma_{k} = m$. The space $J^{k-1}/J^{k}$ is spanned by the monomials $x^{\alpha_{k,1}}y^{c_{\alpha_{k,1}}}, \dots, x^{\alpha_{k, \gamma_{k}}}y^{c_{\alpha_{k,\gamma_{k}}}}$ where $\alpha_{k,1} < \cdots < \alpha_{k, \gamma_{k}}$. Note that if $c_{\alpha_{k, i}} = c_{\alpha_{k', j}}$ for some $k < k'$ then $\alpha_{k, i} < \alpha_{k', j}$. Moreover, if $c_{\alpha_{k, i}} - c_{\alpha_{k', j}} = n$ then $k' \leq k$. 

Pictorially, we consider the staircase diagram defined by the ideal $J^{0}$ and we fill in the box corresponding to the monomial $x^{\alpha_{k, i}}y^{c_{\alpha_{k,i}}}$ with the number $k$. In particular, the number of boxes labeled by $k$ is precisely $\gamma_{k}$.  See Figure \ref{fig:length r flag}. Note that the labels of these boxes are weakly increasing along each vertical run of the staircase diagram, where we read bottom-to-top. Moreover, if two labeled boxes are $n$ horizontal steps apart, then the label of the top box is no greater than that of the bottom box.


\begin{figure}
\begin{tikzpicture}
\filldraw[color=lightgray] (10,3)--(3,3)--(3,2)--(5,2)--(5,1)--(7,1)--(7,0)--(10,0)--(10,3);

\draw(0,0)--(10,0);
\draw(0,0)--(0,3)--(10,3);
\draw (0,1)--(10,1);
\draw (0,2)--(10,2);
\draw (1,0)--(1,3);
\draw (2,0)--(2,3);
\draw (3,0)--(3,3);
\draw (4,0)--(4,3);
\draw (5,0)--(5,3);
\draw (6,0)--(6,3);
\draw (7,0)--(7,3);
\draw (8,0)--(8,3);

\draw (0.5,0.5) node {$1$};
\draw (0.5,1.5) node {$x$};
\draw (0.5,2.5) node {$x^2$};
\draw (1.5,0.5) node {$y$};
\draw (1.5,1.5) node {$xy$};
\draw (1.5,2.5) node {$x^2y$};
\draw (2.5,0.5) node {$y^2$};
\draw (2.5,1.5) node {$xy^2$};
\draw (2.5,2.5) node {$x^2y^2$};
\draw (3.5,0.5) node {$y^3$};
\draw (3.5,1.5) node {$xy^3$};
\draw (4.5,0.5) node {$y^4$};
\draw (4.5,1.5) node {$xy^4$};
\draw (5.5,0.5) node {$y^5$};
\draw (6.5,0.5) node {$y^6$};

\draw (3.5,2.5) node {\Large \bf 4}; 
\draw (5.5,1.5) node {\Large \bf 2};
\draw (7.5,0.5) node {\Large \bf 4};

\draw[line width=2] (3,3)--(3,2)--(5,2)--(5,1)--(7,1)--(7,0);

\end{tikzpicture}
\caption{An element of $\FHilb^{6, y}(\{x^{3} = y^{4}\})$.  Here, $J^{0} = J^{1} = \langle x^{2}y^{3}, xy^{5}\rangle$, $J^{2} = J^{3} = \langle x^{2}y^{3}, xy^{6}\rangle$ and $J^{4} = J^{5} = J^{6} =  yJ^{0} = \langle x^{2}y^{4}, xy^{6}\rangle$.
Also $\gamma = (0,1,0,2,0,0) \in \mathcal{C}_6(3)$
which corresponds to $2^1 4^2$.
Note that the roles of $m$ and $n$, as well as those of $x$ and $y$ are different from those in Figures \ref{fig: ideal} and \ref{fig: PHilb}.  \label{fig:length r flag}}
\end{figure}

The localized equivariant homology $H^{\C^{*}}_{*}(\FHilb^{r, y}(C))$ then admits a basis indexed by classes of fixed points. As in Section \ref{sect: hilb sch sing}, see in particular Lemma \ref{lem: Jordan blocks} for a monomial ideal $J^0 = \C[[y]]\langle y^{c_{1}}, xy^{c_{2}}, \dots, x^{m-1}y^{c_{m}}\rangle$ we can define a composition $(\lambda_{1}, \dots, \lambda_{\ell})$ of $m$.
 Thanks to the discussion above, the flags of monomial ideals that start with $J^{0}$ can be labeled by $\ell$-tuples of monomials $(m_{1}, \dots, m_{\ell})$, where $m_{i}$ is a monomial of degree $\lambda_{i}$ in $r$ variables.

On the other hand, it follows from Lemma \ref{lemma:induced representation} that, as a $\Sk{m}$-module we have
$$
L_{n/m}(\triv) = \bigoplus_{\substack {J^0 \subseteq \cO_{C} \\ J^0 \; \text{monomial ideal}}} \Ind_{\Sk{\lambda_1} \times \cdots \times \Sk{\lambda_{\ell}}}^{\Sk{m}}\triv
$$
So that
\begin{align}\label{eq:gieseker invariants}
\Lnr_{m/n}(n, r) = & (L_{n/m}(\triv) \otimes (\C^r)^{\otimes m})^{\Sk{m}} \\
= & \bigoplus_{\substack {J^0 \subseteq \cO_{C} \\ J^0 \; \text{monomial ideal}}} (\Ind_{\Sk{\lambda_1} \times \cdots \times \Sk{\lambda_{\ell}}}^{\Sk{m}}\triv \otimes (\C^r)^{\otimes m})^{\Sk{m}} \nonumber \\
= & \bigoplus_{\substack {J^0 \subseteq \cO_{C} \\ J^0 \; \text{monomial ideal}}} \Hom_{\Sk{m}}(\Ind_{\Sk{\lambda_1} \times \cdots \times \Sk{\lambda_{\ell}}}^{\Sk{m}}\triv, (\C^r)^{\otimes m}) \nonumber \\
= & \bigoplus_{\substack {J^0 \subseteq \cO_{C} \\ J^0 \; \text{monomial ideal}}}  \Hom_{\Sk{\lambda_1} \times \cdots \times \Sk{\lambda_{\ell}}}(\triv, \Res_{\Sk{\lambda_1} \times \cdots \times \Sk{\lambda_{\ell}}}^{\Sk{m}}(\C^r)^{\otimes m}) \nonumber \\
= & \bigoplus_{\substack {J^0 \subseteq \cO_{C} \\ J^0 \; \text{monomial ideal}}} \Sym^{\lambda_{1}}(\C^r) \otimes \cdots \otimes \Sym^{\lambda_{\ell}}(\C^r) \nonumber 
\end{align}

This suggests that we have an identification $\Lnr_{m/n}(n, r) \cong H^{\C^*}_{*}(\FHilb^{r,y}(C))$. In the next section, we are going to realize this identification geometrically.

\subsection{Compositional parabolic Hilbert schemes, geometrically}
Let us recall that we have the decomposition
$$
\FHilb^{r, y}(C) = \bigsqcup_{\gamma \in \comp{m}} \PHilb^{\gamma,y}(C)
$$
\noindent where $\PHilb^{\gamma,y}(C) = \{\cO_{C} \supseteq J^{0} \supseteq \cdots \supseteq J^{r} = yJ^{0} \mid \dim(J^{k-1}/J^{k}) = \gamma_{k}\}$ and $\comp{m}$ is the set of  weak compositions of $m$ with $r$ parts. In particular, $H^{\C^*}_{*}(\FHilb^{r,y}(C)) = \bigoplus_{\gamma \in \comp{m}}H^{\C^*}_{*}(\PHilb^{\gamma,y}(C))$.

Now, for each $\gamma \in \comp{m}$ we have a map
\begin{align*}
\Pi^{\gamma}:  \PHilb^{y}(C) \to \PHilb^{\gamma,y}(C) \\
 (I_{k} \supseteq I_{k+1} \supseteq \cdots \supseteq I_{k+m} = yI_{k}) & \mapsto (J^{0} \supseteq J^{1} \supseteq \cdots \supseteq J^{r} = yJ^{0}) \end{align*} 
\noindent where $J^{\ell} := I_{k + \sum_{i =  1}^{\ell}\gamma_{i}}$. 

\begin{lemma}\label{lemma:parabolic invariants}
Let $\gamma \in \comp{m}$ and consider the standard parabolic subgroup $\Sk{\gamma^{\rev}} = \Sk{\gamma_{r}} \times \cdots \times \Sk{\gamma_{1}} \subseteq \Sk{m}$. The map $\Pi^{\gamma}_{*}: H^{\C^*}_{*}(\PHilb^{y}(C)) \to H^{\C^*}_{*}(\PHilb^{\gamma,y}(C))$ induces an identification
 $$H^{\C^*}_{*}(\PHilb^{y}(C))^{\Sk{\gamma^{\rev}}} = H^{\C^*}_{*}(\PHilb^{\gamma,y}(C)).$$
\end{lemma}
\begin{proof}
First, we verify that $\Pi^{\gamma}_{*}$ is $\Sk{\gamma^{\rev}}$-invariant, that is, it is constant on $\Sk{\gamma^{\rev}}$-orbits. The group $\Sk{\gamma^{\rev}}$ is generated by simple reflections $s_{i}$, $i \not\in \{\gamma_{r}, \gamma_{r} + \gamma_{r-1}, \dots, \gamma_{r} + \cdots + \gamma_{2}\}$. It is enough to verify that $\Pi^{\gamma}_{*}$ is invariant under each of these simple reflections. 

By definition, $\Pi^{\gamma}$ sends an element $(I_{k} \supseteq  I_{k+1} \supseteq \dots \supseteq I_{k+m} = yI_{k})$ to a flag involving only the ideals $I_{k}, I_{k+\gamma_{1}}, I_{k+\gamma_{1} + \gamma_{2}}, \dots, I_{k+\gamma_{1} + \cdots + \gamma_{r-1}}$, and each one of these ideals has a multiplicity determined by the zeroes in $\gamma$. The invariance now follows from the explicit form of the action of $s_{i}$ obtained in Lemma \ref{lem: action of s_n}. 

Now we have the following commutative diagram:
$$
\xymatrix{ \PHilb^{y}(C) \ar[rr]^{\Pi^{\gamma}} \ar[dr]^{\Pi} & & \PHilb^{\gamma,y}(C) \ar[dl]_{\widetilde{\Pi}} \\ & \Hilb(C) & }.
$$

\noindent The fiber of an ideal $I$ over $\Pi$ is precisely the Springer fiber $\Spr(x) \subseteq \mathcal{F}l(I/yI)$ consisting of \emph{full} flags of subspaces in $I/yI \cong \C^m$ that are stable under the action of the nilpotent operator $x$. Likewise, the fiber of $I$ over $\widetilde{\Pi}$ is the Spaltenstein variety $\Spr^{\gamma}(x) \subseteq \mathcal{F}l^{\gamma}(I/yI)$, consisting of \emph{partial} flags of subspaces in $I/yI$ that are stable under the action of $x$. It is a standard result from Springer theory, see e.g. \cite{BO} or \cite[Section 2.6]{Yun} that 
$$
H_{*}(\Spr(x))^{\Sk{\gamma}} = H_{*}(\Spr^{\gamma}(x))
$$

\noindent from which the result follows. 
\end{proof}

Thanks to the previous lemma and observing that $\gamma \mapsto \gamma^{\rev}$ is an involution on $\comp{m}$ we get 
$$
H^{\C^*}_{*}(\FHilb^{r,y}(C)) = \bigoplus_{\gamma \in \comp{m}}H^{\C^*}_{*}(\PHilb^{\gamma,y}(C)) = \bigoplus_{\gamma \in \comp{m}}H^{\C^*}(\PHilb^{y}(C))^{\Sk{\gamma}}
$$
\noindent on the other hand, we have the following well-known result.

\begin{lemma}\label{lemma:schur weyl}
Let $V$ be a representation of $\Sk{m}$ and $r > 0$. Then
$$
(V \otimes (\C^r)^{\otimes m})^{\Sk{m}} = \bigoplus_{\gamma \in \comp{m}} V^{\Sk{\gamma}}.
$$
Moreover, $V^{\Sk{\gamma}}$ is the $\gamma$-weight space for the $\mathfrak{gl}(r)$ action on the left hand side.
\end{lemma}
\begin{proof}
Fix a basis $e_{1}, \dots, e_{r}$ of $\C^r$. For $\gamma \in \comp{m}$, let $(\C^r)^{\otimes m}_{\gamma}$ be the span of those tensors $e_{i_{1}} \otimes \cdots \otimes e_{i_{m}}$ such that $\gamma_{j} = \sharp\{k : i_{k} = j\}$ (this is $\gamma$-weight subspace in $(\C^r)^{\otimes m}$). It follows by definition that $(\C^r)^{\otimes m}_{\gamma}$ is stable under the action of $\Sk{m}$ and moreover that $(\C^r)^{\otimes m}_{\gamma} = \Ind_{\Sk{\gamma}}^{\Sk{m}}\triv$. Thus, we get $(\C^r)^{\otimes m} = \oplus_{\gamma \in \comp{m}}\Ind_{\Sk{\gamma}}^{\Sk{m}}\triv$ and the result now follows by adjunction. 
\end{proof}

\begin{theorem}\label{thm:gieseker action}
Let $m$ and $n$ be coprime positive integers, and $r > 0$. There is an action of the 
algebra $\cA_{m/n}(n, r)$ on the (localized) equivariant homology $H^{\C^*}_{*}(\FHilb^{r,y})(C)$, where $C$ is the singular curve $\{x^{m} = y^{n}\}$, and with this action we have $H^{\C^{*}}_{*}(\FHilb^{r,y}(C)) \cong \Lnr_{m/n}(n, r)$. 
\end{theorem}
\begin{proof}
We have a natural action of the spherical  subalgebra $eH_{1, n/m}(m, r)e$ on $$(L_{n/m}(\triv) \otimes (\C^r)^{\otimes m})^{\Sk{m}}.$$ Thanks to Theorem \ref{thm: geometric action} the latter space can be identified with $(H^{\C^*}_{*}(\PHilb^{y}(C)) \otimes (\C^r)^{\otimes m})^{\Sk{m}}$ which in turn, by Lemmas \ref{lemma:parabolic invariants} and \ref{lemma:schur weyl} is naturally identified with $H^{\C^*}_{*}(\FHilb^{r,y}(C))$. The result now follows from Proposition \ref{prop:Lmnr}.
\end{proof}
\begin{example}
When $r = 1$, we have $\FHilb^{1, y}(C) = \Hilb(C)$ and, up to \cite[Proposition 9.5]{CEE}, we recover Proposition \ref{prop:spherical hilb}. 
\end{example}

\begin{remark}
We can realize the generators $E_i, F_i$ of $\mathfrak{gl}(r)$  by explicit correspondences between $\PHilb^{\gamma,y}$ and $\PHilb^{\gamma',y}$ similar to \cite[Theorem 3.4]{Brundan}.
\end{remark}

\subsection{Compositional 
Hilbert schemes as generalized affine Springer fibers}\label{sect: CPH GASF} 
Just as with parabolic Hilbert schemes, the
compositional parabolic Hilbert scheme $\FHilb^{r,y}(C)$ admits an interpretation as a generalized affine Springer fiber. In this setting, we let the group $G := \operatorname{GL}_{n}^{\times r}$ act on the vector space $N := \C^n \oplus \gln^{\oplus r}$ in the following way:
$$
(g_{0}, g_{1}, \dots, g_{r-1}).(v, X_{0}, \dots, X_{r-1}) = (g_{0}v, g_{1}X_{0}g_{0}^{-1}, \dots, g_{0}X_{r-1}g_{r-1}^{-1}).
$$
We can visualize $N$ in terms of representations of the following cyclic quiver:
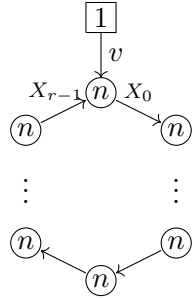
\begin{center}
\begin{tikzpicture}
\draw (0,5) node {1};
\draw (-0.2,4.8)--(-0.2,5.2)--(0.2,5.2)--(0.2,4.8)--(-0.2,4.8);
\draw (0,4) node {$n$};
\draw (0,4) circle (0.2);
\draw (1,3.5) node {$n$};
\draw (1,3.5) circle (0.2);
\draw (1,2) node {$n$};
\draw (1,2) circle (0.2);
\draw (-1,3.5) node {$n$};
\draw (-1,3.5) circle (0.2);
\draw (-1,2) node {$n$};
\draw (-1,2) circle (0.2);
\draw (0,1.5) node {$n$};
\draw (0,1.5) circle (0.2);

\draw [->] (0,4.8)--(0,4.2);
\draw [->] (0.2,3.9)--(0.8,3.6);
\draw (1,2.8) node {$\vdots$};
\draw [<-] (0.2,1.6)--(0.8,1.9);
\draw [<-] (-0.2,3.9)--(-0.8,3.6);
\draw (-1,2.8) node {$\vdots$};
\draw [->] (-0.2,1.6)--(-0.8,1.9);

\draw (0.2,4.5) node {$v$};
\draw (0.5,4) node {\scriptsize $X_0$};
\draw (-0.6,4) node {\scriptsize $X_{r-1}$};
\end{tikzpicture}
\end{center}

As in Section \ref{sec:gasf} we consider the groups $G_{\OO} \subseteq G_{\KK}$. We will consider the \emph{affine Grassmannian} 
$$
\Gr_{G} := G_{\KK}/G_{\OO} = \left(\operatorname{GL}_{n, \KK}/\operatorname{GL}_{n, \OO}\right)^{\times r}
$$
\noindent that parameterizes $r$-tuples of $\OO$-lattices inside $(\KK^{n})^*$ via
\[
[g_0, \dots, g_{r-1}] \mapsto ((\OO^n)^{*}g_{0}^{-1}, \dots, (\OO^n)^{*}g_{r-1}^{-1})
\]
 The group $G_{\KK}$ acts on $N_{\KK} := N \otimes \KK$, and $G_{\OO}$ preserves $N_{\OO}$. Recall the definition of $b_{1} \in \OO^{n}$ and $Y \in \gln(\OO)$ from Section \ref{sec:gasf}. Here, we will consider the following generalized affine Springer fiber
$$
\Spr(b_{1}, \Id, \Id, \dots, Y) := \{[g] \in \Gr_{G} \mid (g_{0}^{-1}b_{1}, g_{1}^{-1}g_{0}, \dots, g_{r-1}^{-1}g_{r-2}, g_{0}^{-1}Yg_{r-1}) \in N_{\OO}\} \subseteq \Gr_{G}.
$$

\begin{proposition}\label{prop:gieseker gasf}
We have an isomorphism
$$
\Spr(b_{1}, \Id, \Id, \dots, Y) \cong \FHilb^{r,y}(C).
$$
\end{proposition}
\begin{proof}
By definition, an element $[g] = [g_{0}, \dots, g_{r-1}] \in \Gr_{G}$ belongs to the space $\Spr(b_{1}, \Id, \Id, \dots, Y)$ if and only if $g_{0}^{-1}b_{1} \in \OO^{n}$, $g_{i+1}^{-1}g_{i} \in \gln(\OO)$ for $i = 0, \dots, r-2$ and $g_{0}^{-1}Yg_{r-1} \in \gln(\OO)$. It easily follows from here that $g^{-1}_{i}b_{1} \in \OO^{n}$ and $g^{-1}_{i}Yg_{i} \in \OO^{n}$ for every $i = 0, \dots, r-1$. Thanks to \cite[Theorem 3.3]{GK} this implies that $\Spr(b_{1}, \Id, \dots, Y) \subseteq \Hilb(C)^{\times r}$. Let $(J^{0}, \dots, J^{r-1}) \in \Hilb(C)^{\times r}$ be the point corresponding to $[g_0, \dots, g_{r-1}]$. The condition $g_{i+1}^{-1}g_{i} \in \gln(\OO)$ for $i = 0, \dots, r-2$  translates to $J^{0} \supseteq J^{1} \supseteq \cdots \supseteq J^{r-1}$, while the condition $g^{-1}_{0}Yg_{r-1} \in \gln(\OO)$ translates to $J^{r-1} \supseteq y J^{0}$. The result follows.
\end{proof}

\begin{remark}\label{rmk: general curve} 
Similar to the work of \cite{GK}, the same proof shows that for an arbitrary plane curve singularity $C$ such that the $x$-projection has degree $n$,
the scheme $\FHilb^{r,y}(C)$ can be presented as a generalized affine Springer fiber for $G=\operatorname{GL}_{n}^{\times r}$ and $N := \C^n \oplus \gln^{\oplus r}$.
\end{remark}

\begin{remark}\label{rmk: affine spaltenstein}
Note that, just as the generalized affine Springer fibers considered in Section \ref{sec:gasf} are the intersection of an affine Springer fiber with the positive affine Grassmannian (cf. Remark \ref{rmk:positive affine g}), the generalized affine Springer fibers we consider here are (disjoint union of) affine Spaltenstein varieties with positive partial affine flag varieties.  
\end{remark}

Just as in Section \ref{sec:gasf}, $\Spr(b_{1}, \Id, \Id, \dots, Y)$ can be realized as one of the varieties considered by \cite{GKM}. Indeed, it is straightforward to verify that 
$$
\Spr(b_{1}, \Id, \dots, \Id, Y) = \mathcal{F}_{a}(t, (b_{1}, \Id, \dots, \Id, Y))
$$
\noindent where $t = 0$ and $a \in \mathfrak{a} := X_{*}(A) \otimes \RR$ is also $0$ where, recall, $A \subseteq G$ is a maximal torus. We can verify that $\Spr(b_{1}, \Id, \dots, Y)$ admits an affine paving as follows. Recall that we need to find $b \in \mathfrak{a}$ and $c \in \RR$ satisfying the three conditions of Section \ref{sec:gasf}. We can take $c = m/n > t = 0$ and $b = (b^0, b^1, \dots, b^{r-1})$, where $b^0 = b^1 = \dots = b^{r-2} = \operatorname{diag}(0,0, \dots, 0)$ and $b^{r-1} = \operatorname{diag}(c, c, \dots, c)$. We need to verify that the element 
$$
\overline{(b_1, \Id, \dots, Y)} = (b_1, \Id, \dots, Y|_{\epsilon = 1}) \in N
$$
\noindent is $G$-good. This follows because the element
$$
\left(\begin{matrix} 1 \\ 0 \\ 0 \\ \vdots \\ 0\end{matrix}\right), \,
 \left(
\begin{matrix}
0 & 0 & \cdots & 0 & Y|_{\epsilon = 1} \\
\Id & 0 & \cdots & 0 & 0 \\
0 & \Id & \cdots & 0 & 0\\
\vdots & \vdots& \ddots & \vdots & 0\\
0 & 0 & \cdots & \Id & 0 \\
\end{matrix}
\right) \in \C^{rn} \oplus \mathfrak{gl}_{nr}
$$

\noindent is $\operatorname{GL}_{nr}$-good, which in turn is a consequence of Proposition \ref{prop:g-good}. Thus, thanks to \cite{GKM} we get the following.

\begin{proposition}\label{prop:gasf gieseker}
The Hilbert scheme $\FHilb^{r,y}(C)$ is paved by affine spaces. Thus, its cohomology is equivariantly formal.
\end{proposition}

\begin{remark}
Similarly to what is done in Section \ref{sec:gasf} one can show that for a composition $\gamma \in \comp{m}$ the variety $\PHilb^{\gamma,y}(C)$ admits a paving by affine spaces. This gives another proof of Proposition \ref{prop:gasf gieseker}. 
\end{remark}


\begin{remark}
The algebra of functions on the Gieseker variety $\mathcal{M}(n, r)$ is known, thanks to the results of Nakajima-Takayama \cite{NT}, see also \cite{dBHOO}, to be the (non-quantized) Coulomb branch algebra for the gauge theory with gauge group $G = \operatorname{GL}_{n}^{\times r}$ and matter representation $N = \C^n \oplus \mathfrak{gl}_{n}^{\oplus r}$ as defined in this section. Uniqueness of quantizations proved by Losev \cite[Theorem 3.4]{L} then shows that the algebra $\cA_{c}(n, r)$ is the corresponding quantized Coulomb branch algebra. It would be interesting to compare the action of $\cA_{c}(n, r)$ on $H^{\C^*}_{*}(\FHilb^{r,y}(C))$ we have constructed here with an action by convolution diagrams as in \cite{GK, HKW}. 
\end{remark}

\begin{remark}
Let $C$ be a plane curve singularity such that the $x$-projection $C \to \C$ has degree $n$. One can use the techniques developed by Hilburn-Kamnitzer-Weekes in \cite{HKW} and Garner-Kivinen in \cite{GK} to show that there is an action of the algebra of functions $\C[\mathcal{M}(n, r)]$ on the non-equivariant homology $H_{*}(\FHilb^{y}(C))$, cf. Section \ref{sec: comparison to GK} and Remark \ref{rmk: general curve}. 
\end{remark}

\begin{remark}
As in Remark \ref{rem: noncoprime}, we can consider the case $C=\{x^m=y^n\}$ for $\gcd(m,n)=d>1$. In this case by \cite{GK,HKW} there is an action of the quantum Gieseker algebra $\cA_{\frac{m}{n}}(n, r)$ on $H^{\C^*}_{*}(\FHilb^{r,y}(C))$. We expect this representation to have minimal support in the sense of \cite{EKLS}. Note that by \cite[Theorem 2.17, Lemma 4.1]{EKLS}  minimally supported representations of $\cA_{\frac{m}{n}}(n, r)$ are related to the minimally supported representations of $H_{\frac{n}{m}}$ in a way similar to Proposition \ref{prop:Lmnr}.
\end{remark}

 \section{Limit \texorpdfstring{$m\to \infty$}{m to infinity}}\label{sect:m to infinity}

In this section, we will see that, in the limit $m \to \infty$, the action of the Dunkl-Opdam subalgebra on $\Delta(\triv) = \C[x_{1}, \dots, x_{n}]$ is still diagonalizable, and we will provide an explicit basis of $\Delta(\triv)$ completely analogous to that of Theorem \ref{thm:action}. Since $H_{c} = H_{1,c} \cong H_{1/c, 1}$, having $c \to \infty$ will yield an action of the algebra $H_{0, 1}$. 

\subsection{The polynomial representation} Recall that, for generic $c$ or for $c$ having denominator precisely $n$, the action of the Dunkl-Opdam subalgebra on the polynomial representation $\Delta_{c}(\triv)$ is diagonalizable. This is, of course, not true for every $c$, as an easy calculation in the case $n = 2$, $c = 1$ shows. However, we have the following result.

\begin{proposition}
For any $c \in \C$, the action of the Dunkl-Opdam subalgebra on $\Delta_c(\triv) = \C[x_{1}, \dots, x_{n}]$ is diagonalizable up to degree $\lfloor |c(n-1)|\rfloor$. Moreover, up to this degree, the action of the algebra $H_{c}$ is given by the same operators as in Theorem \ref{thm:action}.
\end{proposition}
\begin{proof}
Following the strategy of the proof of Theorem \ref{thm:1}, we need to
construct the eigenvectors $v_{\ba}$ for $||\ba|| < \lfloor
|c(n-1)|\rfloor$.
 The only obstruction to constructing these
eigenvectors is that the intertwining operator $\sigma_{i}$ may not be
well-defined on the eigenspace $M_{\tw(\ba)}$. But this is only the
case when $\tw(\ba)_{i} = \tw(\ba)_{i+1}$. Recall that $\tw(\ba)_{i} -
\tw(\ba)_{i+1} = a_{i} - a_{i+1} - (g_{\ba}(i) - g_{\ba}(i+1))c$ where
$g_{\ba}$ is the shortest permutation that sorts $\ba$. Since
$g_{\ba}(i) - g_{\ba}(i+1) \in \{\pm 1, \dots, \pm (n-1)\}$, the
result follows.
 \end{proof}

Thanks to the previous proposition, letting $c \to \infty$ and appropriately rescaling, we get the following \lq\lq$t = 0$\rq\rq\, analogue of Theorem \ref{thm:action}. 



\begin{theorem}\label{thm:t=0}
The $H_{0, 1}$-module $\Delta_{0,1}(\triv) := H_{0,1} \otimes_{\C[y_{1}, \dots, y_{n}] \rtimes \Sn} \triv$ has a basis given by $\{v_{\ba} : \ba \in \Z_{\geq 0}^{n}\}$, and the action of the algebra $H_{0,1}$ on $\Delta_{0,1}(\triv)$ is given by the following operators.
\begin{align*}
u_{i}v_{\ba} & = \tw_{i}v_{\ba} \\
\tau v_{\ba} &  = v_{\pi\cdot \ba} \\
\lambda v_{\ba} & = \tw_{1}v_{\pi^{-1}\cdot \ba} \\
s_{i}v_{\ba} & = \begin{cases}
v_{s_{i}\cdot \ba} + \frac{1}{g_{\ba}(i+1) - g_{\ba}(i)} v_{\ba} & a_{i} > a_{i+1}, \\
\frac{(g_{\ba}(i) - g_{\ba}(i+1) - 1)(g_{\ba}(i) - g_{\ba}(i+1) + 1)}{(g_{\ba}(i) - g_{\ba}(i+1))^{2}}v_{s_{i}\cdot \ba} + \frac{1}{g_{\ba}(i) - g_{\ba}(i+1)}v_{\ba} & a_{i} < a_{i+1}, \\
v_{\ba} & a_{i} = a_{i+1}
\end{cases}
\end{align*}

\noindent where $\tw_{i} := \tw_{i}(\ba) = (1 - g_{\ba}(i))$ and, as before, $g_{\ba}$ is the minimal-length permutation that sorts $\ba$. 
\end{theorem}

\begin{remark}
As above, one can also define the renormalized basis $\widetilde{v}_\ba$ such that
$$
(1 + s_{i})\widetilde{v}_{\ba} = \frac{g_{\ba}(i+1) - g_{\ba}(i) - 1}{g_{\ba}(i+1) - g_{\ba}(i)}\widetilde{v}_{s_{i}\cdot \ba} + \frac{g_{\ba}(i+1) - g_{\ba}(i) + 1}{g_{\ba}(i+1) - g_{\ba}(i)}\widetilde{v}_{\ba}.
$$
\end{remark}

\begin{remark}
We remark that, unlike the $t = 1$ case, the module $\Delta_{0, 1}(\triv)$ is never irreducible. It has a  unique irreducible \emph{graded} quotient  $L_{0,1}(\triv) = \C[x_{1}, \dots, x_{n}]/(\C[x_{1}, \dots, x_{n}]^{\Sn}_{+})$. 
\end{remark}

Note that the proof of Theorem \ref{thm:t=0} can be extended to any Verma module $\Delta_{0,1}(\mu) := H_{0,1} \otimes_{\C[\mathbf{y}] \rtimes \Sk{n}} V_{\mu}$. In particular, we get that $\Delta_{0,1}(\mu)$ has a basis given by $v(\ba, T)$, where $\ba \in \Z_{\geq 0}^{n}$ and $T \in \SYT(\mu)$. The action of $H_{0, 1}$ on $\Delta_{0,1}(\mu)$ is given by
\MVcomment{ 
\begin{align*}
u_{i}v(\ba, T) & = \tw_i(\ba,T)v(\ba,T) \\
\tau v(\ba, T) & = v(\pi\cdot \ba, T) \\
\lambda v(\ba, T) & = \tw_1(\ba, T)v(\pi^{-1}\cdot \ba, T) \\
s_{i}v(\ba, T) = & \begin{cases}
v(s_{i}\cdot \ba, T) - A_{2}v(\ba, T) & a_{i} > a_{i+1}, \\
A_{1}v(s_{i}\cdot \ba, T) + A_{2}v(\ba, T) & a_{i} < a_{i+1}, \\
(\ct{T}{g_{\ba}(i+1)} - \ct{T}{g_{\ba}(i)})v(\ba, T) & a_{i} = a_{i+1} \; \text{and} \; s_{g_{\ba}(i)}(T) \not\in \SYT(\mu), \\
v(\ba, s_{g_{\ba}(i)}(T)) - A_{2}v(\ba,T) & a_{i} = a_{i+1} \; \text{and} \; s_{g_{\ba}(i)}(T)  \in \SYT(\mu) 
 \end{cases}
\end{align*}
} 
\begin{align*}
u_{i}v(\ba, T) & = \tw_i(\ba,T)v(\ba,T) \\
\tau v(\ba, T) & = v(\pi\cdot \ba, T) \\
\lambda v(\ba, T) & = \tw_1(\ba, T)v(\pi^{-1}\cdot \ba, T) \\
s_{i}v(\ba, T) = & \begin{cases}
v(s_{i}\cdot \ba, T) - A_{2}v(\ba, T) & a_{i} > a_{i+1}, \\
A_{1}v(s_{i}\cdot \ba, T) + A_{2}v(\ba, T) & a_{i} < a_{i+1}, \\
(\ct{T}{j+1} - \ct{T}{j})v(\ba, T) & a_{i} = a_{i+1} \; \text{and} \; s_{j}(T) \not\in \SYT(\mu), \\
v(\ba, s_{j}(T)) - A_{2}v(\ba,T) & a_{i} = a_{i+1} \; \text{and} \; s_{j}(T)  \in \SYT(\mu) 
 \end{cases}
\end{align*}
\noindent where $\tw_{i}(\ba, T) = -\ct{T}{g_{\ba}(i)}$,  we 
write $j=g_{\ba}(i)$ in the cases $a_i = a_{i+1},$ and where

$$
A_{1} = \frac{(\ct{T}{g_{\ba}(i)} - \ct{T}{g_{\ba}(i+1)} - 1)(\ct{T}{g_{\ba}(i)} - \ct{T}{g_{\ba}(i+1)} + 1)}{(\ct{T}{g_{\ba}(i)} - \ct{T}{g_{\ba}(i+1)})^{2}}
$$

\noindent and

$$
A_{2} =  \frac{1}{\ct{T}{g_{\ba}(i)} - \ct{T}{g_{\ba}(i+1)}}
=  \frac{1}{\ct{T}{j} - \ct{T}{j+1}}.
$$

\subsection{Hilbert scheme of the non-reduced line}

On the geometric side, the curve $\{x^m=y^n\}$ has a natural limit at $m\to \infty$, namely, the non-reduced 
line $\{y^n=0\}$. The ring of functions on $C_0=\{y^n=0\}$ has a basis $x^{i}y^{j}$ for $i\ge 0,n-1\ge j\ge 0$, as above. 

The Hilbert scheme of points on $\{y^n=0\}$ is the moduli space of ideals in the local ring
$$
\cO_{C_0,0}=\C[[x,y]]/y^n=\C[[x]]\langle 1,\ldots, y^{n-1}\rangle.
$$
The multiplication by $y$ is given by the matrix similar to \eqref{eq:Y}:
$$
Y=\left(
\begin{matrix}
0 & 0 & \cdots & 0 & 0\\
1 & 0 & \cdots & 0 & 0 \\
0 & 1 & \cdots & 0 & 0\\
\vdots & \vdots& \ddots & \vdots & 0\\
0 & 0 & \cdots & 1& 0 \\
\end{matrix}
\right).
$$
We consider the $\C^*$ action on $C_0$ and on $\cO_{C_0, 0}$ such that $y$ has weight $1$ and $x$ has weight 0.
It naturally extends to the action on the punctual Hilbert scheme $\Hilbk(C_0,0)$.
\begin{lemma}
The fixed points of this action are isolated and correspond to monomial ideals.
\end{lemma}
\begin{proof}
An ideal $I$ in $\cO_{C_0}$ is fixed under this $\C^*$ action if and only it it is generated by functions $y^{\alpha_i}p_i(x)$ which are homogeneous in $y$ but not necessary in $x$. On the other hand, 
in the ring of formal power series $p_i(x)$ is proportional to $x^{c_i}$ up to a unit, and 
hence $I$ is the monomial ideal generated by $y^{\alpha_i}x^{c_i}$
for $1 \le i\le n$. 
\end{proof}

\begin{remark}
It is important for the above proof that we work with the punctual Hilbert scheme of ideals supported at the origin, rather than with the full Hilbert scheme. 
\end{remark} 

\begin{remark}\label{remark:other c*}
Unlike the curve $\{x^m=y^n\}$, the curve $C_0$ has an action of another $\C^*$ such that $y$ has weight 0 and $x$ has weight 1. The weight of this action on a monomial ideal $I$ generated by the  $y^{\alpha_i}x^{c_i}$ equals $\sum c_i=\dim \cO_C/I=k.$
\end{remark}

Similarly, one can define the parabolic Hilbert scheme $\PHilb_{k,n+k}(C_0)$ as the space of flags of ideals $I_k\supset I_{k+1}\supset\cdots \supset xI_{k}=I_{k+n}$ in $\cO_{C_0}$, and $\PHilb^{x}(C_0) := \sqcup_{k} \PHilb_{k, n+k}(C_0)$. The fixed points in $\PHilb^{x}(C_0)$ are determined by sequences of monomials $(y^{\alpha_i}x^{c_i})$ with no restrictions on $c_i$. As in Lemma \ref{lem: weights of line bundles}, we have $\alpha_i=\widetilde{g}_{\bc}(i)-1$, where $\widetilde{g}_{\bc}$ is the permutation which sorts $c_i$ in non-increasing order (recall that when $c_i=c_j$ with $i<j$ we have
$\alpha_i < \alpha_j$) .
We can write $c_i=a_{n+1-i}$ and  $\widetilde{g}_{\bc}(i)=n+1-g_{\ba}(n+1-i)$.

The construction of geometric operators corresponding to $u_i,s_i,\tau$ and $\lambda$ extends verbatim to this case, however, one needs to be careful with the equivariant weights.  Now $\cL_i$ has the weight of the monomial $(y^{\alpha_i}x^{c_i})$, that is
$$
c_1(\cL_i)=\alpha_i=\widetilde{g}_{\bc}(i)-1=n-g_{\ba}(n+1-i)=(n-1)+\tw_{n+1-i}.
$$
The operators $T$ and $\Lambda$ can be defined as in Section \ref{sec:geometric operators}, and their 
matrix elements can be computed similarly. Observe that $\cO_{C_0}/x\cO_{C_0}$ still has a unique $Y$-invariant one dimensional subspace generated by $y^{n-1}$ which has weight $(n-1)$. The computation in Theorem \ref{thm: geometric action} then implies $T\circ \Lambda=u_{1}$. We conclude the following:

\begin{theorem}\label{thm:hilb t=0}
Consider the non-reduced curve $C_0=\{y^n=0\}$ with the $\C^*$ action $(x,y)\mapsto (x,sy)$. Then the $\C^*$ equivariant cohomology
$$
U_{\infty}=\bigoplus_{k=0}^{\infty}H^{\C^*}_*(\PHilb_{k,n+k}(C_0))
$$
has an action of the rational Cherednik algebra $H_{0,1}$ defined by the same operators in Theorem \ref{thm: geometric action}. This representation is isomorphic to the polynomial representation of $H_{0,1}$.
\end{theorem}

\begin{remark}
Note that the work of Garner and Kivinen \cite{GK} also accounts for the case of the non-reduced curve $C_0$. However, \cite{GK} obtains an action of the Cherednik algebra $H_{0,1}$ on the equivariant cohomology of $\PHilb^{x}(C_0)$ with respect to the group $\C^{*} \times \C^{*}$, see Remark \ref{remark:other c*}. In particular, we can obtain Theorem \ref{thm:hilb t=0} from \cite[Proposition 1.5]{GK} after specializing one of the equivariant parameters to $0$. 
\end{remark}


Finally, we would like to mention that the constructions of Section \ref{sect:gieseker} can be extended to this setup, and we get the following result.

\begin{theorem}
With the same notation as in Theorem \ref{thm:hilb t=0} the $\C^{*}$-equivariant cohomology
$$
H_{*}^{\C^*}(\FHilb^{r, x}(C_0))
$$

\noindent has an action of the spherical algebra $eH_{0, 1}(n, r)e$, where $H_{0,1}(n, r)$ is the matrix version of the Cherednik algebra defined in Definition \ref{def:matrix RCA}. This representation is isomorphic to the representation $(\C[x_1, \dots, x_n] \otimes (\C^r)^{\otimes n})^{\Sk{n}}$ defined in a natural way. 
\end{theorem}

\begin{remark}
From its interpretation as a generalized affine Springer fiber, see Section \ref{sect: CPH GASF}, it follows that the homology $H_{*}^{\C^*}(\FHilb^{r, x}(C_0))$ admits an action of 
a flavor deformation of 
the algebra of functions on the Gieseker variety $\mathcal{M}(n, r)$. When $r = 1$, this flavor deformation is precisely $eH_{0, 1}(n, 1)e$, which is known to be commutative and it is in fact the algebra of functions on the Calogero-Moser space, \cite{EG}. It is unclear the relationship that the flavor deformation bears to $eH_{0, 1}(n, r)e$ when $r > 1$. 
\end{remark}

\subsection*{Conflict of interest}
 The authors declare that they have no conflict of interest.

\end{document}